\documentclass [11pt,reqno]{amsart}
\usepackage {amsmath,amssymb,verbatim,geometry,color,esint}
\usepackage[all]{xy}
\usepackage[pdftex,hyperindex]{hyperref}
\usepackage[pdftex]{graphicx}
\usepackage{tikz}
\usetikzlibrary{math}

\newcommand{\A}{\mathbb{A}}

 \newcommand{\C}{\mathbb{C}}

\newcommand{\N}{\mathbb{N}}
\renewcommand{\P}{\mathbb{P}}
 \newcommand{\Q}{\mathbb{Q}}
 \newcommand{\R}{\mathbb{R}}
 \newcommand{\Z}{\mathbb{Z}}

\newcommand{\fM}{\mathfrak{M}}

\renewcommand{\hom}{\mathrm{hom}}

\newcommand{\hD}{\widehat{\Delta}}

\newcommand{\cE}{\mathcal{E}}

\newcommand{\cH}{\mathcal{H}}

\newcommand{\cL}{\mathcal{L}}
\newcommand{\cM}{\mathcal{M}}
\newcommand{\cN}{\mathcal{N}}

\newcommand{\cT}{\mathcal{T}}

\newcommand{\cX}{\mathcal{X}}

\newcommand{\tcL}{\widetilde{\mathcal{L}}}
\newcommand{\tcX}{\widetilde{\mathcal{X}}}

\newcommand{\hto}{\hookrightarrow}

\renewcommand{\a}{\alpha}
\renewcommand{\b}{\beta}
\renewcommand{\d}{\delta}
\newcommand{\e}{\varepsilon}
\newcommand{\f}{\varphi}
\newcommand{\unipar}{\varpi}

\newcommand{\la}{\lambda}
\newcommand{\om}{\omega}

\newcommand{\p}{\psi}

\newcommand{\inter}{\cdot\ldots\cdot}

\newcommand{\eg}{{\rm e.g.\ }} 
\newcommand{\ie}{{\rm i.e.\ }}

\newcommand{\Ka}{\mathrm{K}}

\newcommand{\ld}{\mathrm{A}}

\newcommand{\inte}{\mathrm{int}}

\newcommand{\an}{\mathrm{an}}

\DeclareMathOperator{\en}{E}

\DeclareMathOperator{\mab}{M}

\DeclareMathOperator{\elle}{L}

\DeclareMathOperator{\ding}{D}

\DeclareMathOperator{\ii}{I}
\DeclareMathOperator{\jj}{J}

\DeclareMathOperator{\Cz}{C^0}
\DeclareMathOperator{\Num}{N^1}

\DeclareMathOperator{\Aut}{Aut}

\DeclareMathOperator{\tee}{T}

\DeclareMathOperator{\DF}{DF}

\DeclareMathOperator{\Ent}{Ent}

\DeclareMathOperator{\IN}{IN}

\DeclareMathOperator{\Hnot}{H^0}

\DeclareMathOperator{\MA}{MA}

\DeclareMathOperator{\Spec}{Spec}
\DeclareMathOperator{\supp}{supp}
\DeclareMathOperator{\vol}{vol}

\DeclareMathOperator{\tr}{tr}

\DeclareMathOperator{\dd}{{d}}

\DeclareMathOperator{\Amp}{Amp}
\DeclareMathOperator{\Pic}{Pic}

\DeclareMathOperator{\id}{id}

\DeclareMathOperator{\ord}{ord}
\DeclareMathOperator{\Nef}{Nef}

\DeclareMathOperator{\Psef}{Psef}

\DeclareMathOperator{\PSH}{PSH}

\DeclareMathOperator{\redu}{red}

\DeclareMathOperator{\te}{T}
\DeclareMathOperator{\esse}{S}
\DeclareMathOperator{\PL}{PL}

\DeclareMathOperator{\ddc}{dd^c}

\newcommand{\n}{\chi}

\renewcommand{\div}{\mathrm{div}}

\newcommand{\triv}{\mathrm{triv}}
\newcommand{\val}{\mathrm{val}}
\newcommand{\lin}{\mathrm{lin}}

\newcommand{\fld}{\mathrm{fld}}

\newcommand{\NA}{\mathrm{NA}}

\newcommand{\D}{\Delta}

\newcommand{\simto}{\overset\sim\to}

\numberwithin{equation}{section}       

\newtheorem{prop} {Proposition} [section]
\newtheorem{thm}[prop] {Theorem} 
\newtheorem{defi}[prop] {Definition}
\newtheorem{lem}[prop] {Lemma}
\newtheorem{cor}[prop]{Corollary}
\newtheorem{prop-def}[prop]{Proposition-Definition}

\newtheorem*{mainthm}{Main Theorem} 
\newtheorem*{thmA}{Theorem A} 
 
\newtheorem*{thmB}{Theorem B} 
 
\newtheorem*{thmD}{Theorem D}

\newtheorem*{corC}{Corollary C}

\newtheorem*{conjO}{Conjecture}
\newtheorem{exam}[prop]{Example}
\newtheorem{rmk}[prop]{Remark}

\newtheorem{conj}[prop]{Conjecture}
\theoremstyle{remark}
\newtheorem*{ackn}{Acknowledgment}

\title[A non-Archimedean approach to K-stability, II]{A non-Archimedean approach to K-stability, II:\\
divisorial stability and openness}
\date{\today}

\author{S{\'e}bastien Boucksom \and Mattias Jonsson}

\address{Sorbonne Universit\'e and Universit\'e Paris Cit\'e\\
CNRS\\
IMJ-PRG\\
F-75005 Paris\\
France}
\email{sebastien.boucksom@imj-prg.fr}

\address{Dept of Mathematics\\
  University of Michigan\\
  Ann Arbor, MI 48109-1043\\
  USA}
\email{mattiasj@umich.edu}

\begin{document}

\begin{abstract} 
To any projective pair $(X,B)$ equipped with an ample $\Q$-line bundle $L$ (or even any ample numerical class), we attach a new invariant $\b(\mu)\in\R$, defined on convex combinations $\mu$ of divisorial valuations on $X$, viewed as point masses on the Berkovich analytification of $X$. The construction is based on non-Archimedean pluripotential theory, and extends the Dervan--Legendre invariant for a single valuation---itself specializing to Li and Fujita's valuative invariant in the Fano case, which detects K-stability. Using our $\b$-invariant, we define divisorial (semi)stability, and show that divisorial semistability implies $(X,B)$ is sublc (\ie its log discrepancy function is non-negative), and that divisorial stability is an open condition with respect to the polarization $L$. We also show that divisorial stability implies uniform K-stability in the usual sense of (ample) test configurations, and that it is equivalent to uniform K-stability with respect to all norms/filtrations on the section ring of $(X,L)$, as considered by Chi Li. 
\end{abstract}

\dedicatory{To the memory of Jean-Pierre Demailly, with admiration}

\maketitle

\setcounter{tocdepth}{1}
\tableofcontents

\section*{Introduction}
Consider, for the moment, a complex projective manifold $X$, equipped with an ample $\Q$-line bundle $L$. The notion of K-stability of the polarized manifold $(X,L)$ was originally phrased in~\cite{Dontoric} (building on~\cite{Tian97}) in terms of certain equivariant one-parameter degenerations of $(X,L)$ known as \emph{test configurations}. The `uniform' version\footnote{More general versions, involving  in particular the action of a reductive group of automorphisms, will not be considered in this paper.} of the Yau--Tian--Donaldson (YTD) conjecture states that the cohomology class $c_1(L)$ contains a unique constant scalar curvature K\"ahler (cscK) metric iff $(X,L)$ is \emph{uniformly K-stable} (see~\cite{BHJ1,Der}). By~\cite{BDL20,CC}, the former condition is known to be equivalent to the coercivity of the Mabuchi K-energy functional, which implies in turn uniform K-stability~\cite{BHJ2}. 

As originally pointed out in~\cite{WN12}, (ample) test configurations for $(X,L)$ can be understood as filtrations of finite type on the section ring of $(X,L)$; following the work of G.~Sz\'ekelyhidi~\cite{Sze} and the authors' preprints~\cite{trivvalold,nakstabold}, C.~Li~\cite{Li22} studied K-stability for general filtrations (see Definition~\ref{defi:Kstabfilt}), and proved the remarkable result that uniform K-stability for filtrations implies coercivity of the Mabuchi K-energy functional---and hence the existence of a unique cscK metric. To sum up, for any polarized manifold $(X,L)$ we have 
\begin{equation}\label{equ:ytdimpl}
\text{uniform K-stability for filtrations}\Rightarrow\text{unique cscK metric }\Rightarrow\text{uniform K-stability}, 
\end{equation}
and the missing part of the (uniform) YTD conjecture thus consists in proving the (purely algebro-geometric) statement that uniform K-stability for filtrations already follows from its version for test configurations, viewed as filtrations of finite type. 

Our primary goal in this paper is to provide a valuative characterization of uniform K-stability for  filtrations, as used by Li in~\cite{Li22}. Using this, we will prove:
\begin{mainthm}
  Uniform K-stability for filtrations is an open condition on the polarization.
\end{mainthm}
More precisely, the condition of $(X,L)$ being uniformly K-stable for filtrations only depends on the numerical class of $L$, and is an open condition on the ample cone.
This result holds also when $X$ is singular, see Theorem~\ref{thm:stabopen}. 
In the smooth case, the openness property ties in well with the corresponding result for cscK metrics~\cite{LS}. 

\smallskip
To explain the proof of the main theorem, let us first
consider the Fano case $L=-K_X$. In this situation, the YTD conjecture was first established in full generality in~\cite{CDS}, and the missing implication in~\eqref{equ:ytdimpl} is thus known. It can alternatively be obtained by purely algebro-geometric means, combining techniques from the Minimal Model Program (MMP), which allows us to convert uniform K-stability into uniform Ding-stability~\cite{LX14,YTDold,Fujval}, together with the non-Archimedean version of Berman's `thermodynamical formalism', which yields the equivalence between uniform Ding-stability and uniform K-stability for filtrations (see Corollary~\ref{cor:normding}). 

Crucially, the MMP techniques of~\cite{LX14} further show that K-stability of a Fano manifold (or, more generally, any log Fano pair) can be tested using only `special' test configurations, which correspond to so-called `dreamy' divisorial valuations~\cite{Fujval}; this leads to a purely valuative characterization of K-stability~\cite{LiEquivariant,Fujval}, which has been instrumental in the recent spectacular progress towards the deeper understanding of K-stability of log Fano varieties and the construction of moduli spaces thereof; see~\cite{BLXZ,LXZ}, to name just a few. 

Motivated by this, R.~Dervan and E.~Legendre initiated in~\cite{DL} the study of a valuative criterion in the case of a general polarization, by providing a purely valuative expression $\b(v)\in\R$ for the Donaldson--Futaki invariant of a special test configuration in terms of the corresponding `dreamy' divisorial valuation $v$. Their invariant $\b(v)$ in fact makes sense for any divisorial valuation $v$ (see~\eqref{equ:betaval} below), and gives rise to a notion of (uniform) \emph{valuative stability}, which was recently proved in~\cite{LiuYa} to be an open condition with respect to the polarization. 

\smallskip

In contrast with the log Fano case, valuative stability is not expected to imply K-stability for a general polarization, and the MMP will likely play a less important role in that case as well. Using our previous works~\cite{trivval,nakstab1}, we will extend the $\b$-invariant to \emph{convex combinations} of divisorial valuations, viewed as atomic measures on the Berkovich analytification of $X$, and show that the corresponding stability notion, which we call \emph{divisorial stability} is equivalent to (uniform) K-stability for filtrations.

%
\subsection*{Valuative stability}
As in the main body of the paper, we consider from now an arbitrary \emph{polarized pair} $(X,B;\om)$, where $X$ is a normal projective variety over an algebraically closed field $k$ of characteristic $0$, $B$ is a (not necessarily effective) $\Q$-Weil divisor on $X$ such that $K_{X,B}:=K_X+B$ is $\Q$-Cartier, and $\om\in\Amp(X)\subset\Num(X)$ is a (possibly irrational) ample numerical class, of volume $V_\om=(\om^n)$ with $n:=\dim X$. 

Before describing our notion of divisorial stability, let us first briefly revisit the definitions of~\cite{DL,LiuYa} in the present setting. Denote by $X^\div$ the set of \emph{divisorial valuations} $v\colon k(X)^\times\to\R$, of the form $v=s\ord_F$ where $F$ is a prime divisor on a smooth birational model $\pi\colon Y\to X$ and $s\in\Q_{\ge 0}$ (the case $s=0$ corresponding to the \emph{trivial valuation} $v_\triv$). The volume function $\vol\colon\Num(Y)\to\R_{\ge0}$ is continuous, and positive precisely on the big cone. Set
\begin{equation}\label{equ:normval}
\|v\|_\om:=s V_\om^{-1}\int_0^{+\infty}\vol(\pi^\star\om-\la F)\,d\la\in\R_{\ge 0}. 
\end{equation}
This quantity, which appeared in~\cite{MR15,LiEquivariant,Fujval} under various normalization and notation, coincides with the \emph{expected vanishing order} $\esse_L(v)$ when $\om=c_1(L)$ with $L\in\Pic(X)_\Q$ (see~\cite{BlJ}). The function $\|\cdot\|_\om\colon X^\div\to\R_{\ge 0}$ so defined is homogeneous with respect to the scaling action of $\Q_{>0}$, and vanishes precisely on the trivial valuation $v_\triv$. 

Using the differentiability of the volume function on the big cone~\cite{diskant,LM09}, one checks that $\|v\|_\om$ is a differentiable function of $\om\in\Amp(X)$, see~\S\ref{sec:endiv}. This allows us to define
\begin{equation}\label{equ:betaval}
\b(v)=\b_{X,B;\om}(v):=\ld_{X,B}(v)+\frac{d}{dt}\bigg|_{t=0}\|v\|_{\om+t K_{X,B}}\in\R, 
\end{equation}
where $\ld_{X,B}(v)=s\,\ld_{X,B}(F)\in\Q$ is the \emph{log discrepancy} of $v=s\ord_F$. Differentiating under the integral sign in~\eqref{equ:normval} yields an expression of the invariant $\b(v)$ that coincides with the one in~\cite{DL,LiuYa} (up to a factor $V_\om$), see~\S\ref{sec:bfunc} for details. 

The $\b$-invariant defines a $\Q_{>0}$-homogeneous function $\b\colon X^\div\to\R$. Following~\cite{DL,LiuYa}, we say that $(X,B;\om)$ is \emph{valuatively semistable} if $\b\ge 0$ on $X^\div$, and \emph{valuatively stable}\footnote{For simplicity, we drop `uniformly' from the terminology.} if $\b\ge\e\|\cdot\|_\om$ for some $\e>0$. 
%
\subsection*{Divisorial stability}
Our stability notion involves the larger space $\cM^\div$ of \emph{divisorial measures} on $X$. This is the set of probability measures on $X^\div$ of the form 
\begin{equation*}
  \mu=\sum_{i=1}^rm_i\d_{v_i},
\end{equation*}
where $r\ge1$, $m_i\in\R_{>0}$, $v_i\in X^\div$, and $\sum_im_i=1$. If $L$ is an ample $\Q$-line bundle, then any ample test configuration $(\cX,\cL)$ for $(X,L)$ defines a divisorial measure, whose support is exactly the set of divisorial valuations associated with the irreducible components of the central fiber of the normalization of $\cX$, see~\cite{BHJ1}. However, not every divisorial measure is of this form. For example, a Dirac mass $\d_v$ with $v\in X^\div$  is of this form iff $v$ is dreamy with respect to $L$ (\ie the corresponding filtration of the section ring is of finite type). 

There is an obvious embedding $X^\div\hto\cM^\div$ given by $v\mapsto\d_v$, and we can extend the functional $\b$ from $X^\div$ to $\cM^\div$ as follows. The first term in~\eqref{equ:betaval} is extended by linearity: following~\cite{BHJ1,BHJ2} we define the (non-Archimedean) \emph{entropy} of a measure $\mu=\sum_i m_i\d_{v_i}$ as
\[
  \Ent_{X,B}(\mu):=\sum_im_i \ld_{X,B}(v_i)=\int\ld_{X,B}\,\mu; 
\]
this does not depend on the class $\om$. To extend the second term in~\eqref{equ:betaval}, we interpret the invariant~\eqref{equ:normval} as the \emph{energy} of the Dirac mass $\d_v$ in the sense of non-Archimedean pluripotential theory~\cite{trivval}. 

To explain this, recall first that the space $X^\div$ admits a natural compactification $X^\an$, the \emph{Berkovich analytification} of $X$ (with respect to the trivial absolute value on $k$), whose points can be viewed as semivaluations on $X$, \ie valuations $v\colon k(Y)^\times\to\R$ for some subvariety $Y$ of $X$. We can therefore embed $\cM^\div$ into the space $\cM$ of Radon probability measures on the compact Hausdorff space $X^\an$, and define the energy of any such measure $\mu\in\cM$, as follows. Assume first that $\om\in\Amp(X)$ is rational, \ie $\om=c_1(L)$ with $L\in\Pic(X)_\Q$. Any function $\f\in\Cz(X^\an)$ induces a filtration on the vector space $R_m:=\Hnot(X,mL)$ for $m$ sufficiently divisible, given by
$$
F^\la R_m:=\{s\in R_m\mid v(s)+m \f(v)\ge\la\text{ for all }v\in X^\an\},\quad\la\in\R,
$$
where $X^\an$ can be replaced with the dense subset $X^\div$, by continuity. Suitably normalized, the volumes of these filtrations (\ie the average of their jumping numbers) converge to a number $\vol_L(\f)\in\R$, see~\cite{BC}.  It further follows from~\cite{trivval} that $\vol_L(\f)=\vol_\om(\f)$ only depends on $\om=c_1(L)$, and that $\om\mapsto\vol_\om(\f)$ uniquely extends by continuity to the whole ample cone $\Amp(X)\subset\Num(X)$. 

Each $\om\in\Amp(X)$ thus defines a concave functional $\vol_\om\colon\Cz(X^\an)\to\R$, and the \emph{energy functional}\footnote{This corresponds to $\en^\vee$ in~\cite{trivval}; the change of notation is intended to make the formalism of the present paper easier to digest.}  $\|\cdot\|_\om\colon\cM\to [0,+\infty]$ can be described as its Legendre transform, \ie 
\begin{equation*}
 \|\mu\|_\om=\sup_{\f\in\Cz(X^\an)}\left\{\vol_\om(\f)-\int\f\,\mu\right\}
\end{equation*}
for $\mu\in\cM$. The energy is convex, lower semicontinuous in the weak topology of $\cM$, and vanishes precisely at the \emph{trivial measure} $\mu_\triv:=\d_{v_\triv}$; it is further homogeneous with respect to the natural scaling action of $\R_{>0}$.

As mentioned before, the energy $\|\d_v\|_\om$ of the Dirac mass associated to a divisorial valuation $v\in X^\div$ turns out to coincide with~\eqref{equ:normval}; in particular, it is finite, and the energy functional therefore restricts to a finite-valued, convex functional $\|\cdot\|_\om\colon\cM^\div\to\R_{\ge 0}$.

 When $\om=c_1(L)$ with $L\in\Pic(X)_\Q$ and $\mu\in\cM^\div$ is associated to a (semi)ample test configuration $(\cX,\cL)$ for $(X,L)$, $\|\mu\|_L$ equals the \emph{minimum norm} of $(\cX,\cL)$ in the sense of Dervan~\cite{Der}; this corresponds to $\ii^\NA-\jj^\NA$ in the notation\footnote{Generally speaking, the `non-Archimedean' functionals of~\cite{BHJ1,BHJ2} were denoted by adding a superfix `$\NA$' to the corresponding functional in K\"ahler geometry; in later works~\cite{YTD,trivval}, this superfix was dropped when no ambiguity can arise.} of~\cite{BHJ1}, and can be expressed as a certain linear combination of the intersection numbers $(\bar\cL^{n+1})$ and $(L_{\bar\cX}\cdot\bar\cL^n)$, with $(\bar\cX,\bar\cL)\to\P^1$ the canonical compactification of $(\cX,\cL)\to\A^1$, and $L_{\bar\cX}$ the pull-back of $L$. In fact, this fully characterizes the energy functional, as any $\mu\in\cM^\div$ can be written as the weak limit of measures $(\mu_j)$ associated to test configurations in such a way that $\|\mu_j\|_L\to\|\mu\|_L$. 

Our first main result is as follows: 
\begin{thmA}
  For any measure $\mu\in\cM^\div$, the function $\om\mapsto\|\mu\|_\om$ is of class $C^1$ on $\Amp(X)$.
\end{thmA}
The result actually holds for all measures of finite energy, \ie $\|\mu\|_\om<\infty$ (a condition that is independent of $\om$). The proof, which relies on rather sophisticated estimates for Monge--Amp\`ere integrals from~\cite{trivval}, ultimately deriving from the Hodge index theorem, further yields H\"older estimates for the derivative of the energy, a key ingredient in the proof of Theorem~B below. 

Given $\om\in\Amp(X)$, $\theta\in\Num(X)$, and $\mu\in\cM^\div$, we write
\[
  \nabla_\theta\|\mu\|_\om:=\frac{d}{dt}\bigg|_{t=0}\|\mu\|_{\om+t\theta}
\]
for the directional derivative of the energy with respect to $\om$. By approximation, this is again fully characterized by the case where $\om=c_1(L)$ and $\mu$ is associated to a (semi)ample test configuration $(\cX,\cL)$ for $(X,L)$, for which $\nabla_\theta\|\mu\|_\om$ can be expressed as a certain linear combination of the intersection numbers $(\bar\cL^{n+1})$ and $(\theta_{\bar\cX}\cdot\bar\cL^n)$ (see~\eqref{equ:nablaE}, \eqref{equ:nablavee}). In particular, $\nabla_\om\|\mu\|_\om=\|\mu\|_\om$. Note, however, that Theorem~A does not directly follow from the description of $\|\mu\|_\om$ in terms of intersection numbers, which cannot be used anymore for the perturbed energy $\|\mu\|_{\om+t\theta}$.

For any polarized pair $(X,B;\om)$, we may now define the desired extension $\b\colon\cM^\div\to\R$ of~\eqref{equ:betaval} (with respect to the embedding $X^\div\hto\cM^\div$) 
by setting 
\begin{equation*}
  \b(\mu)=\b_{X,B;\om}(\mu):=\Ent_{X,B}(\mu)+\nabla_{K_{X,B}}\|\mu\|_\om. 
\end{equation*}
We say that $(X,B;\om)$ is \emph{divisorially semistable} if $\b\ge 0$ on $\cM^\div$, and \emph{divisorially stable} if $\b\ge\e\|\cdot\|_\om$ for some $\e>0$. By considering Dirac masses, it follows immediately that divisorial (semi)stability implies valuative (semi)stability, and we expect the converse to fail in general; see~\cite[Example~2.28]{DL}.
  
When $\om=c_1(L)$ with $L\in\Pic(X)_\Q$ and $\mu$ is associated to an ample test configuration $(\cX,\cL)$, it follows from the description of $\nabla_{K_{X,B}}\|\mu\|_\om$ in terms of intersection numbers mentioned above that $\b(\mu)$ coincides with the \emph{(non-Archimedean) Mabuchi K-energy} of $(\cX,\cL)$, \ie its Donaldson--Futaki invariant up to a simple error term that disappears after base change (see~\cite{BHJ1}). As a result, divisorial semistability implies K-semistability, and divisorial stability implies uniform K-stability; here we conjecture that the converse does hold:  more on this below. 
%
%
\subsection*{Divisorial stability threshold and openness}
As in~\cite{LiuYa}, it is convenient to introduce the \emph{divisorial stability threshold} of the polarized pair $(X,B;\om)$, defined as 
\begin{equation}\label{equ:sigint}
\sigma_\div(X,B;\om):=\inf_{\mu\in\cM^\div\setminus\{\mu_\triv\}}\frac{\b_{X,B;\om}(\mu)}{\|\mu\|_\om}\in\R\cup\{-\infty\}. 
\end{equation}
Thus $(X,B;\om)$ is divisorially semistable (resp.~stable) iff $\sigma_\div(X,B;\om)\ge 0$ (resp.~$>0$). 
\begin{thmB} For any polarized pair $(X,B;\om)$, the following holds: 
\begin{itemize}
\item[(i)] $\sigma_\div(X,B;\om)>-\infty$ iff $(X,B)$ is sublc; 
\item[(ii)] $\sigma_\div(X,B;\om)$ depends continuously on $\om\in\Amp(X)$. 
\end{itemize}
\end{thmB}
Recall that the pair $(X,B)$ is \emph{sublc} (a short-hand for \emph{sub-log canonical}) if its log discrepancy function $\ld_{X,B}\colon X^\div\to\Q$ is non-negative; the pair $(X,B)$ is then lc in the usual sense iff $B$ is further effective. As an immediate consequence of Theorem~B, we get:

\begin{corC} If a polarized pair $(X,B;\om)$ is divisorially semistable, then $(X,B)$ is necessarily sublc. Furthermore, divisorial stability of $(X,B;\om)$ is an open condition on $\om\in\Amp(X)$.
\end{corC}

Note that the last part of Corollary~C implies the Main Theorem above. 
The first part can be viewed as a version in our context of a celebrated result of Y.~Odaka~\cite{Oda} (see also~\cite{BHJ1}), to the effect that any polarized pair $(X,B;L)$ with $L\in\Pic(X)_\Q$ ample and $B$ effective that is K-semistable is necessarily lc. While the proof of the latter result relies on the MMP through the existence of log canonical blowups, the proof of Theorem~B~(i) is of an entirely different nature, and rests on an estimate for the energy on certain affine segments in $\cM^\div$.

Part~(ii) of Theorem~B is a consequence of a refined version of Theorem~A, involving H\"older estimates for the directional derivative $\om\mapsto\nabla_\theta\|\mu\|_\om$ of the energy (which actually show that $\om\mapsto\sigma_\div(X,B;\om)$ is locally H\"older continuous). While delicate, these estimates derive in a rather formal way from the Hodge index theorem; this is studied in~\cite{synthetic}, where versions of Theorem~A and Theorem~B~(ii) are shown to hold over an arbitrary valued (possibly Archimedean) field.
Theorem~B~(ii) should also be compared with the work~\cite{LiuYa} of Yaxiong Liu, who considers a similar threshold of \emph{valuative} stability, defined using only Dirac masses $\mu=\d_v$, $v\in X^\div$, and shows that this threshold is a continuous function on the ample cone. In fact, given any subset $M\subset\cM^\div$, one could consider the threshold defined by restricting the infimum in~\eqref{equ:sigint} to $M$, and our proof yields the continuity with respect to $\om\in\Amp(X)$, which recovers Liu's result. However, we cannot prove that uniform K-stability is an open condition\footnote{Note that openness of uniform K-stability was stated in~\cite{SD20}, but the article has been withdrawn due to a gap in the proof. See~\cite{Fujopen} for a partial result.} in this way, since the subset of $\cM^\div$ used to test the K-stability of $(X,L)$ depends on $L$. 

\medskip

In the log Fano case, the divisorial stability threshold essentially reduces to the \emph{$\d$-invariant} defined in~\cite{FO18,BlJ}. In fact, for any sublc polarized pair $(X,B;\om)$ such that $-K_{X,B}\equiv\la\om$ with $\la\in\R$, the identity $\nabla_\om\|\cdot\|_\om=\|\cdot\|_\om$ mentioned above yields the simpler expression
$$
\b_{X,B;\om}=\Ent_{X,B}+\nabla_{K_{X,B}}\|\cdot\|_\om=\Ent_{X,B}-\la\|\cdot\|_\om. 
$$
By linearity of the entropy and convexity of the energy, this implies 
$$
\sigma_\div(X,B;\om)=\d(X,B;\om)-\la
$$
where 
$$
\d(X,B;\om)=\inf_{v\in X^\div\setminus\{v_\triv\}}\frac{\ld_{X,B}(v)}{\|v\|_\om}
$$
coincides with the usual $\d$-invariant. As a consequence, divisorial stability, (uniform) valuative stability, and uniform K-stability are all equivalent in the log Fano case (see Corollary~\ref{cor:divding}). 
%
%
\subsection*{K-stability for filtrations}
Next we discuss the relation between divisorial stability and the stability notions by C.~Li~\cite{Li22}. We note that Li assumed that $X$ is smooth, as he relied on the preprint~\cite{nakstabold}, and specifically continuity of envelopes, which is not yet known in the singular case. Using~\cite{nakstab1}, we are able to bypass this obstacle.

Consider thus a polarized pair $(X,B;\om)$, and assume that $(X,B)$ is \emph{subklt}, \ie its log discrepancy function $\ld_{X,B}\colon X^\div\to\Q$ is positive outside $\{v_\triv\}$. This function then admits a maximal lsc extension $\ld_{X,B}\colon X^\an\to [0,+\infty]$, which coincides with the one constructed in~\cite{JM,BdFFU} on the subspace $X^\val\subset X^\an$ of valuations on $k(X)$, and is infinite outside it (see~Appendix~\ref{sec:logdisc}). We may thus define an lsc extension $\Ent_{X,B}\colon\cM\to[0,+\infty]$ of the entropy functional by setting $\Ent_{X,B}(\mu):=\int\ld_{X,B}\,\mu$. This yields in turn an extension $\b\colon\cM^1\to\R\cup\{+\infty\}$ of the $\b$-functional to the space $\cM^1\subset\cM$ of measures $\mu$ of \emph{finite energy}, \ie $\|\mu\|_\om<+\infty$, which we characterize as the maximal lsc extension of $\b\colon\cM^\div\to\R$ with respect to the natural (strong) topology of $\cM^1$. In particular, the divisorial stability threshold can be computed using all measures in $\cM^1$, \ie 
\begin{equation}\label{equ:sigM1int}
\sigma_\div(X,B;\om)=\inf_{\mu\in\cM^1\setminus\{\mu_\triv\}}\frac{\b(\mu)}{\|\mu\|_\om}=\inf_{\mu\in\cM^1,\,\|\mu\|_\om=1}\b(\mu),
\end{equation}
where the last equality holds by homogeneity with respect to the scaling action of $\R_{>0}$. 

Assume further $\om=c_1(L)$ with $L\in\Pic(X)_\Q$ ample. As in~\cite{nakstab1}, denote by $\cN_\R$ the set of \emph{norms} $\chi\colon R(X,dL)\to\R\cup\{+\infty\}$ on the section ring of $(X,dL)$, with $d\ge 1$ sufficiently divisible (depending on $\n$). Such norms are in 1--1 correspondence with the more commonly used (multiplicative, graded, linearly bounded) filtrations, via the inverse maps
$$
F^\la R_m:=\{s\in R_m\mid \n(s)\ge \la\},\quad\n(s):=\max\{\la\in\R\mid s\in F^\la R_m\}
$$
where $R_m:=\Hnot(X,mL)$ for $m$ sufficiently divisible. The space $\cN_\R$ comes with a translation action $(c,\n)\mapsto\n+c$ of $\R$, such that $(\n+c)(s):=\n(s)+mc$ for $s\in R_m$.  

The Rees construction yields an identification of the subset $\cT_\Z\subset\cN_\R$ of $\Z$-valued norms of finite type $\n$ with the set of ample test configurations for $(X,L)$. Each $\n\in\cT_\Z$ thus defines a divisorial measure $\MA(\n)\in\cM^\div$, called the \emph{Monge--Amp\`ere measure} of $\n$. By~\cite{nakstab1}, the space $\cN_\R$ is equipped with a natural pseudometric $\dd_1$, with respect to which $\cT_\Z$ is dense, and the Monge--Amp\`ere operator admits a unique $\dd_1$-continuous extension $\MA\colon\cN_\R\to\cM^1$, which is invariant under the translation action of $\R$. We may now define the \emph{Mabuchi K-energy} and the \emph{minimum norm} of any $\n\in\cN_\R$ by 
$$
\mab(\n):=\b(\MA(\n)),\quad\|\n\|:=\|\MA(\n)\|_L.
$$
\begin{thmD} The divisorial stability threshold of any polarized subklt pair $(X,B;L)$ satisfies
$$
\sigma_\div(X,B;L)=\inf_{\n\in\cN_\R,\,\|\n\|>0}\frac{\mab(\n)}{\|\n\|}.
$$
\end{thmD}
In particular, $(X,B;L)$ is divisorially semistable (resp.~stable) iff $\mab\ge 0$ on $\cN_\R$ (resp.~$\mab\ge\e\|\cdot\|$ for some $\e>0$), which respectively correspond to K-semistability and uniform K-stability for filtrations, as considered in~\cite{Li22}. Note that this notion of K-stability for filtrations a priori differs from the one in~\cite{Sze} and relies on working on the full Berkovich space $X^{\an}$ rather than just $X^{\div}$. 

In view of~\eqref{equ:sigM1int}, the main step in the proof of Theorem D consists in showing that the image of $\MA\colon\cN_\R\to\cM^1$ contains the set $\cM^\div$ of divisorial measures. This follows from~\cite{nakstab1}, where it is proved that the Monge--Amp\`ere operator induces a 1--1 map 
$$
\MA\colon\cN^\div_\R/\R\simto\cM^\div,
$$
where $\cN^\div_\R\subset\cN_\R$ denotes the set of \emph{divisorial norms}, of the form $\n=\min_i\{\n_{v_i}+c_i\}$ for a finite set $(v_i)$ in $X^\div$ and $c_i\in\R$. We then have $\MA(\n)=\sum_i m_i\d_{v_i}$ for the same set of valuations $(v_i)$, where $m_i=m_i(c)$ is a certain (non-linear) function of $c=(c_i)$. 

As a consequence, the infimum in Theorem D can be computed on the space $\cN^\div_\R$ of divisorial norms; we show that it can be further restricted to the subspace $\cN^\div_\Q$ of \emph{rational divisorial norms}, whose coefficients $c_i$ above can be chosen rational. By~\cite{nakstab1}, such norms arise from arbitrary (\ie not necessarily ample) test configuration for $(X,L)$, called \emph{models} in~\cite{Li22}, and it follows that divisorial stability is also equivalent to \emph{uniform K-stability for models} in the sense of C.~Li; see~\S\ref{sec:uKstabmod}.

A fortiori, divisorial stability implies uniform K-stability, as we already noted above. The two notions are equivalent if a certain \emph{entropy regularization conjecture} holds, as first formulated in~\cite{nakstabold}. A stronger, and more precise conjecture, goes as follows:
\begin{conjO}
  Let $\n\in\cN^\div_\Q$ be a rational divisorial norm, and let $(\n_d)$ be its sequence of canonical approximants, where $\n_d$ is generated in degree $1$ by the restriction of $\chi$ to $\Hnot(X,dL)$ for $d$ sufficiently divisible. Then $\lim_d\Ent_{X,B}(\MA(\n_d))=\Ent_{X,B}(\MA(\n))$. 
\end{conjO}
Granted this conjecture, divisorial stability is the same as uniform K-stability---and the uniform YTD conjecture thus holds for any polarized complex manifold, as discussed in the beginning of the introduction. See also~\cite{Li23}. 

In the Fano case, divisorial stability and valuative stability are equivalent. In the general case, we do not expect this to be true, but one can ask whether it is enough to consider divisorial measures with a fixed bound on the cardinality of their supports. For example,~\cite[Proposition~5.3.1]{Dontoric} shows that on a toric surface, it suffices to consider measures supported on a set of cardinality at most two.
%
%
\subsection*{Structure of the paper}
The article is organized as follows. 

\begin{itemize}
\item Section~\ref{sec:prelim} recalls some aspects of our previous work~\cite{trivval,nakstab1} of relevance for the present paper. 
\item Section~\ref{sec:diffen} is devoted to the proof of Theorem~A (cf.~Theorem~\ref{thm:diffen}), along with some key estimates that will lead to the proof of Theorem~B~(ii). 
\item Section~\ref{sec:entropy} studies the entropy functional of a pair, the associated $\d$-invariant, and its relation to Ding-stability. Theorem~\ref{thm:delta} is the main ingredient in the proof of Theorem~B~(i). 
\item Section~\ref{sec:Divstab} introduces the main concepts of this paper, the $\b$-invariant of a divisorial measure, and the associated notion of divisorial stability. It completes the proof of Theorems~A and B. 
\item Section~\ref{sec:DivKstab} compares divisorial stability and K-stability. Theorem~D is proved, and the entropy regularization conjecture is discussed. Along the way, we prove the Main Theorem above.
\item Finally, Appendix~\ref{sec:logdisc} reviews the properties of the log discrepancy function of a pair, and its extension to the Berkovich space in the subklt case. 
 \end{itemize}

%
%
\begin{ackn}
The authors would like to thank R.~Berman, H.~Blum, R.~Dervan, E.~di Nezza, A.~Ducros, V.~Guedj, M.~Hattori, E.~Inoue, E.~Legendre, C.~Li, Yaxiong Liu, Y.~Odaka, R.~Reboulet, L.~Sektnan, Z.~Sj\"ostr\"om Dyrefelt, A.~Zeriahi, and K.~Zhang for fruitful discussions and useful comments. The first author was partially supported by the ANR grant GRACK.\@ The second author was partially supported by NSF grants DMS-1600011, DMS-1900025, DMS-2154380, and the United States---Israel Binational Science Foundation.
\end{ackn}

%
%
%
\section{Preliminaries}\label{sec:prelim}
The main purpose of this preliminary section is to recall results from non-Archimedean pluripotential theory, as developed in~\cite{trivval}, which form the building blocks of our approach. 
%
%
%
\subsection{Notation and conventions}\label{sec:notation}

\begin{itemize}

\item We say that a function $f\colon Z\to\R$ defined on a set $Z$ endowed with an action of $\R_{>0}$ (or a subgroup thereof) is \emph{homogeneous} if it satisfies the equivariance property $f(t\cdot x)=tf(x)$ for $t\in\R_{>0}$ and $x\in Z$. 

\item A \emph{net} in a set $Z$ is a family $(x_i)_{i\in I}$ of elements of $Z$ indexed by a directed set, \ie a partially preordered set in which any two elements are dominated by a third one.   

\item If $Z$ is a Hausdorff topological space, and $\f\colon Z\to\R\cup\{\pm\infty\}$ is any function, then the \emph{usc regularization} $\f^\star$ of $\f$ is the smallest usc function with $\f^\star\ge\f$. Concretely, $\f^\star(x)=\limsup_{y\to x}\f(y)$. The \emph{lsc regularization} is defined by $\f_\star=-(-\f)^\star$. 
   
\item For $x,y\in\R_+$, $x\lesssim y$ means $x\le C_n  y$ for a constant $C_n>0$ only depending on $n$, and $x\approx y$ if $x\lesssim y$ and $y\lesssim x$. Here $n$ will be the dimension of a fixed variety $X$ over $k$.

\end{itemize}
%
\subsection{Quasi-metric spaces}

A \emph{quasi-metric} on a set $Z$ is a function $d\colon Z\times Z\to\R_{\ge 0}$ that is symmetric, separates points, and satisfies the \emph{quasi-triangle inequality}
$$
\e\,d(x,y)\le d(x,z)+d(z,y)
$$
for some constant $\e>0$. A quasi-metric space $(Z,d)$ comes with a Hausdorff topology, and even a uniform structure. In particular, Cauchy sequences and completeness make sense for $(Z,d)$. Such uniform structures have a countable basis of entourages, and are thus metrizable, by general theory. 

A continuous function $f\colon Z\to\R$ on a quasi-metric space is \emph{uniformly continuous} if, for each $\e>0$, there exists $\d>0$ such that $d(x,y)\le\d\Longrightarrow |f(x)-f(y)|\le\e$ for any $x,y\in Z$. The following standard result will be used several times in the paper:
\begin{lem}\label{lem:unifext} Let $(Z,d)$ be a quasi-metric space, $D\subset Z$ a dense subset, and $f\colon D\to\R$ a uniformly continuous function. Then $f$ admits a unique uniformly continuous extension $f\colon Z\to\R$. 
\end{lem}
\begin{proof} Uniqueness is clear, by density of $D$. Pick $x\in Z$, and a sequence $(x_i)$ in $D$ such that $x_i\to x$. Then $\e d(x_i,x_j)\le d(x_i,x)+d(x,x_j)$ tends to $0$ as $i,j\to\infty$. By uniform continuity of $f$, it follows that $(f(x_i))$ is a Cauchy sequence, which thus admits a limit $f(x)\in\R$. Using again uniform continuity, it is further easy to check that the limit is independent of the choice of $(x_i)$, and that the extension constructed this way is uniformly continuous. 
\end{proof}

%
\subsection{Positive numerical classes}\label{sec:voltr}
In the entire paper, we work over an algebraically closed field $k$ of characteristic $0$, and $X$ denotes an irreducible projective variety over $k$, \ie an integral projective $k$-scheme (not necessarily normal for now). We set $n:=\dim X$. 

Denote by $\Num(X)$ the finite dimensional $\R$-vector space of numerical classes of $\R$-Cartier divisors on $X$. Ample classes form a nonempty open convex cone $\Amp(X)\subset\Num(X)$. We generally denote by $\theta$ an element of $\Num(X)$, and by $\om$ an element of $\Amp(X)$. The closure $\Nef(X)\subset\Num(X)$ of $\Amp(X)$ is the closed convex cone of nef classes. We denote by $\ge$ the corresponding partial order on $\Num(X)$, \ie
\begin{equation}\label{equ:neforder}
\theta\ge\theta'\Longleftrightarrow\theta-\theta'\in\Nef(X).
\end{equation}
Each $\om\in\Amp(X)$ induces a norm $\|\cdot\|_\om$ on $\Num(X)$, defined by 
\begin{equation}\label{equ:normom}
\|\theta\|_\om:=\sup\{s\ge 0\mid -s\om\le\theta\le s\om\}.
\end{equation}
We will occasionally use the \emph{Thompson metric} of the open convex cone $\Amp(X)$, defined by
\begin{equation}\label{equ:Thompson}
\d(\om,\om'):=\max\{\d\ge 0\mid e^{-\d}\om\le\om'\le e^\d\om\}.
\end{equation}
It is locally equivalent to the metric on $\Amp(X)$ induced by any norm on $\Num(X)$ (see~\cite[Lemma~3]{Tho}). 

 The \emph{volume} of $\om\in\Amp(X)$ is defined as
\begin{equation}\label{equ:vol}
V_\om=\vol(\om)=(\om^n). 
\end{equation}
We define the \emph{trace} of $\theta\in\Num(X)$ with respect to $\om\in\Amp(X)$ as 
\begin{equation}\label{equ:traceom}
\tr_\om(\theta):=n\frac{(\om^{n-1}\cdot\theta)}{(\om^n)}.
\end{equation}
Note that $\tr_\om(\theta)$ is linear with respect to $\theta$, with
\begin{equation}\label{equ:trbound}
|\tr_\om(\theta)|\le n\|\theta\|_\om.
\end{equation}
The trace computes the logarithmic derivative of the volume, \ie 
\begin{equation}\label{equ:logvol}
V_\om^{-1}V_{\om+t\theta}=1+\tr_\om(\theta)t+O(t^2),\quad t\to 0. 
\end{equation}

\begin{exam} When $k=\C$ and $X$ is smooth, $\tr_\om(c_1(X))=-\tr_\om(K_X)$ computes the mean value $\bar S$ of the scalar curvature  of any K\"ahler form representing $\om$. 
\end{exam}

The \emph{volume function} 
$$
\vol=\vol_X\colon\Num(X)\to\R_{\ge 0}
$$ 
is continuous, positive precisely on the open convex cone $\mathrm{Big}(X)\subset\Num(X)$ of big classes, and coincides with~\eqref{equ:vol} on the ample cone. It is further homogeneous, log concave on the pseudoeffective cone $\Psef(X)$, and of class $C^1$ on $\mathrm{Big}(X)$, with derivative at $\a\in\mathrm{Big}(X)$ given by
\begin{equation}\label{equ:vol'}
\nabla_\theta\vol(\a):=\frac{d}{dt}\bigg|_{t=0}\vol(\a+t\theta)=n\langle\a^{n-1}\rangle\cdot\theta
\end{equation}
for all $\theta\in\Num(X)$, where $\langle\a^{n-1}\rangle\in\mathrm{N}_1(X)$ is a positive intersection product (see~\cite{diskant}). The latter is in fact defined for any $\a\in\Psef(X)$, with $\langle\a^{n-1}\rangle=\lim_{\e\to 0_+}\langle(\a+\e\om)^{n-1}\rangle$, and we set $\langle\a^{n-1}\rangle=0$ for $\a\in\Num(X)\setminus\Psef(X)$. This makes sense of $\nabla_\theta\vol(\a):=n\langle\a^{n-1}\rangle\cdot\theta$ for any $\a\in\Num(X)$. Note, however, that this might not compute the directional derivatives of $\vol$ when $\a$ lies on the boundary of the big cone. 
%
%
%
\subsection{Berkovich analytification and psh functions} 
The \emph{Berkovich analytification} $X^\an$ of $X$ (with respect to the trivial absolute value on $k$) is a compact Hausdorff topological space, whose points are \emph{semivaluations} on $X$, \ie valuations $v\colon k(Y)^\times\to\R$ for a subvariety $Y\subset X$. It contains as a dense subset the space $X^\val$ of actual valuations on $X$, endowed with the topology of pointwise convergence as maps $k(X)^\times\to\R$. 

The space $X^\an$ comes with a continuous scaling action $\R_{>0}\times X^\an\to X^\an$ $(t,v)\mapsto tv$, which fixes the \emph{trivial valuation} $v_\triv\in X^\val$, defined by $v_\triv\equiv 0$ on $k(X)^\times$. 

The subspace $X^\div\subset X^\val$ of \emph{divisorial valuations} is already dense in $X^\an$. Each $v\in X^\div$ is of the form $v=t\ord_F$ for a prime divisor $F$ on a smooth birational model $Y\to X$ and $t\in\Q_{\ge 0}$ (the case $t=0$ corresponding to $v_\triv$). 

We denote by $\Cz(X)$ the Banach space of continuous functions $\f\colon X^\an\to\R$, endowed with the supnorm, and by $\Cz(X)^\vee$ its topological dual, \ie the space of (signed) Radon measures on $X^\an$. It contains the subspace $\cM=\cM(X)\subset\Cz(X)^\vee$ of Radon probability measures, which is convex and compact for the weak topology. The scaling action of $\R_{>0}$ on $X^\an$ induces an action $(t,\mu)\mapsto t_\star\mu$ on $\Cz(X)^\vee$ preserving $\cM$. 

The space $\Cz(X)$ contains a dense subspace $\PL(X)$ of \emph{piecewise linear (PL) functions}, see~\cite[\S2]{trivval}. Among various possible descriptions, each such function is of the form $\f_D\in\PL(X)$ for a vertical $\Q$-Cartier divisor $D$ on some test configuration $\cX\to\A^1$ for $X$, where $\f_D(v)=\sigma(v)(D)$ for $v\in X^\an$, with $\sigma\colon X^\an\to\cX^\an$ denoting Gauss extension. This construction is invariant under pullback to a higher test configuration, and one can thus always assume that $\cX$ dominates the trivial test configuration $X\times\A^1$.

To each $\om\in\Amp(X)$, one associates a class $\PSH(\om)$ of \emph{$\om$-psh functions} $\f\colon X^\an\to\R\cup\{-\infty\}$, $\f\not\equiv-\infty$, defined in~\cite[\S4]{trivval}, and characterized as follows:
\begin{itemize}
\item a PL function $\f$, written $\f=\f_D$ as above, is $\om$-psh iff $\om_\cX+D$ is relatively nef on $\cX$, with $\om_\cX\in\Num(\cX/\A^1)$ the pullback of $\om$ to $\cX$;
\item each $\f\in\PSH(\om)$ can be written as the pointwise limit of a decreasing net\footnote{In fact, sequences turn out to be enough, but this will not be used in this paper.} $(\f_i)$ in $\PL(X)\cap\PSH(\om)$. 
\end{itemize}
Each $\f\in\PSH(\om)$ is usc, and hence bounded above. Further, $\sup\f=\f(v_\triv)$. Define $\tee=\tee_\om\colon X^\an\to [0,+\infty]$ by 
$$
\tee(v):=\sup_{\f\in\PSH(\om)}\{\sup\f-\f(v)\}. 
$$
A simple approximation argument yields $\te(v)=\sup_{\f\in\PL\cap\PSH(\om)}\{\sup\f-\f(v)\}$, which shows that $\te$ is lsc; further, 
$$
\te(v)=0\Leftrightarrow v=v_\triv,\text{ and  }\te(tv)=t\te(v),\, t\in\R_{>0}.
$$
The (Borel) set
\begin{equation}\label{equ:Xlin}
X^\lin:=\{v\in X^\an\mid\te(v)<\infty\}
\end{equation}
is independent of $\om$. It is contained in $X^\val$, and a point $v\in X^\an$ lies in $X^\lin$ iff $\f(v)>-\infty$ for all $\f\in\PSH(\om)$, \ie iff $\{v\}$ is \emph{nonpluripolar}.  

\begin{exam} If $\om=c_1(L)$ with $L\in\Pic(X)_\Q$ ample, then
\begin{equation}\label{equ:Tsec}
\te(v)=\sup\left\{m^{-1}v(s)\mid s\in\Hnot(X,mL)\setminus\{0\}\right\}=\sup\{v(D)\mid D\in |L|_\Q\},
\end{equation}
and hence $v\in X^\lin$ iff $v$ is a \emph{valuation of linear growth} in the sense of~\cite{BKMS}. 
\end{exam}

\begin{exam} If $v\in X^\div$ is written $v=t\ord_F$ for a prime divisor $F$ on a smooth model $\pi\colon Y\to X$, then  
\begin{equation}\label{equ:Tthresh}
\te(v)=t\sup\left\{\la\in\R_{>0}\mid \pi^\star\om-\la F\in\Psef(Y)\right\}<\infty. 
\end{equation}
\end{exam}
In particular, every $\om$-psh function is finite-valued on $X^\div\subset X^\lin$; it is further determined by its restriction to $X^\div$, and we equip the space $\PSH(\om)$ with the (Hausdorff) topology of pointwise convergence on $X^\div$. The scaling action of $\R_{>0}$ on $X^\an$ induces an action on the topological space $\PSH(\om)$, denoted by $(t,\f)\mapsto t\cdot\f$ where
\begin{equation}\label{equ:scalefunc}
(t\cdot\f)(v):=t\f(t^{-1}v). 
\end{equation}

%
\subsection{Energy pairing and functions of finite energy}\label{sec:enpairing}
The \emph{energy pairing} is first defined as a symmetric $(n+1)$-linear map on $\Num(X)\times\PL(X)$, that takes a tuple $(\theta_i,\f_i)\in\Num(X)\times\PL(X)$, $i=0,\dots,n$, to 
\begin{equation}\label{equ:pairing}
(\theta_0,\f_0)\inter(\theta_n,\f_n):=(\theta_{0,\bar\cX}+D_0)\inter(\theta_{n,\bar\cX}+D_n)\in\R,
\end{equation}
where $\cX\to\A^1$ is a high enough test configuration such that $\f_i=\f_{D_i}$ for some vertical $\Q$-Cartier divisor $D_i$ on $\cX$, $\theta_{i,\bar\cX}\in\Num(\bar\cX)$ is the pull-back of $\theta_i$ to the canonical compactification $\bar\cX\to\P^1$, and the right-hand side is an intersection number computed on $\bar\cX$. See~\cite[\S3.2]{trivval}. The energy pairing (which is simply an extension to numerical classes of~\cite[Definition~6.6]{BHJ1}) satisfies
\begin{equation}\label{equ:pairing0}
(\theta_0,0)\inter(\theta_n,0)=0,
\end{equation}
\begin{equation}\label{equ:pairingmass}
(0,1)\cdot(\theta_1,\f_1)\inter(\theta_n,\f_n)=(\theta_1\inter\theta_n).
\end{equation}
Further, if $\f_i,\p_i\in\PL\cap\PSH(\om_i)$ with $\om_i\in\Amp(X)$, $i=0,\dots,n$, then 
\begin{equation}\label{equ:pairingmono}
\f_i\le\p_i\ \text{for all }i\Longrightarrow(\om_0,\f_0)\inter(\om_n,\f_n)\le(\om_0,\p_0)\inter(\om_n,\p_n).
\end{equation}
For any $\om\in\Amp(X)$, the \emph{Monge--Amp\`ere energy}\footnote{This should not be confused with the \emph{extended} Monge--Amp\`ere energy of $\f$ as in~\cite[\S 8]{trivval}.} of $\f\in\PL(X)$ with respect to $\om$ is defined as
\begin{equation}\label{equ:EMA}
\en_\om(\f):=\frac{(\om,\f)^{n+1}}{(n+1)(\om^n)}. 
\end{equation}

This normalization guarantees the equivariance property 
\begin{equation}\label{equ:MAenequiv}
\en_\om(\f+c)=\en_\om(\f)+c
\end{equation}
for $c\in\Q$ (see~\eqref{equ:pairingmass}). By~\eqref{equ:pairingmono}, the restriction $\en_\om\colon\PL\cap\PSH(\om)\to\R$ is increasing; it admits a unique usc, increasing extension $\en_\om\colon\PSH(\om)\to\R\cup\{-\infty\}$, given by  
$$
\en_\om(\f)=\inf\{\en_\om(\p)\mid\p\in\PL\cap\PSH(\om),\,\p\ge\f\}. 
$$
The space of \emph{$\om$-psh functions of finite energy} is defined as 
$$
\cE^1(\om)=\{\f\in\PSH(\om)\mid\en_\om(\f)>-\infty\},
$$
and the \emph{strong topology} of $\cE^1(\om)$ is defined as the coarsest refinement of the subset topology from $\PSH(\om)$ (\ie the topology of pointwise convergence on $X^\div$) in which $\en_\om\colon\cE^1(\om)\to\R$ becomes continuous. 

The vector space $\vec\cE^1$ generated by $\cE^1(\om)$ (interpreted as functions $X^\lin\to\R$) turns out to be independent of $\om\in\Amp(X)$. It contains $\PL(X)$, and we have $\cE^1(\om)=\vec\cE^1\cap\PSH(\om)$ for any $\om\in\Amp(X)$. See~\cite[\S7]{trivval} for details on this and the remainder of~\S\ref{sec:enpairing}. 

\begin{thm}\label{thm:enpairing} The energy pairing $(\theta_0,\f_0)\inter(\theta_n,\f_n)$, previously defined as a multilinear function of tuples $(\theta_i,\f_i)\in\Num(X)\times\PL(X)$, admits a unique extension to a symmetric, $(n+1)$-linear function of tuples $(\theta_i,\f_i)\in\Num(X)\times\vec\cE^1$ such that, for any $\om\in\Amp(X)$, its restriction to tuples in $\Num(X)\times\cE^1(\om)$ is continuous with respect to the strong topology of $\cE^1(\om)$. 
\end{thm}
The energy pairing is further homogeneous with respect to the scaling action~\eqref{equ:scalefunc}, \ie
\begin{equation}\label{equ:enpairinghomog}
(\theta_0,t\cdot\f_0)\inter(\theta_n,t\cdot\f_n)=t(\theta_0,\f_0)\inter(\theta_n,\f_n)
\end{equation}
for $t\in\R_{>0}$. 

Since $\PL(X)$ is a dense subspace of $\Cz(X)$, we can now associate to any tuple $(\theta_i,\f_i)\in\Num(X)\times\vec\cE^1$, $i=0,\dots,n$, a \emph{mixed Monge--Amp\`ere measure} 
\begin{equation}\label{equ:mixedMA}
(\theta_1+\ddc\f_1)\wedge\dots\wedge(\theta_n+\ddc\f_n)\in\Cz(X)^\vee
\end{equation}
by requiring that 
$$
\int\f_0\,(\theta_1+\ddc\f_1)\wedge\dots\wedge(\theta_n+\ddc\f_n)=(0,\f_0)\cdot(\theta_1,\f_1)\inter(\theta_n,\f_n)
$$
for all $\f_0\in\vec\cE^1$. The Radon measure~\eqref{equ:mixedMA} has total mass $(\theta_1\inter\theta_n)$ (see~\eqref{equ:pairingmass}), and it is positive when $\theta_i$ is ample and $\f_i$ is $\theta_i$-psh. 

Fix now $\om\in\Amp(X)$, with volume $V_\om=(\om^n)$. By Theorem~\ref{thm:enpairing}, the Monge--Amp\`ere energy~\eqref{equ:EMA} makes sense for any $\f\in\vec\cE^1$, and~\eqref{equ:pairing0} yields the more standard expression
\begin{equation}\label{equ:Emixed}
\en_\om(\f)=\frac{1}{n+1}\sum_{j=0}^n V_\om^{-1}\int\f (\om+\ddc\f)^j\wedge\om^{n-j}. 
\end{equation}
By~\eqref{equ:enpairinghomog}, we further have 
\begin{equation}\label{equ:enhomog}
\en_\om(t\cdot\f)=\en_{t\om}(t\f)=t\en_\om(\f)
\end{equation}
for $t\in\R_{>0}$. 

We also get a \emph{Monge--Amp\`ere operator} $\MA_\om\colon\vec\cE^1\to\Cz(X)^\vee$, that takes $\f\in\vec\cE^1$ to the Radon measure 
$$
\MA_\om(\f):=V_\om^{-1}(\om+\ddc\f)^n.
$$
If $\f\in\cE^1(\om)=\vec\cE^1\cap\PSH(\om)$, then $\MA_\om(\f)$ is a probability measure, and the induced map $\MA_\om\colon\cE^1(\om)\to\cM$ is further strongly continuous. The energy $\en_\om\colon\vec\cE^1\to\R$ can be understood as the anti-derivative of the Monge--Amp\`ere operator; indeed, a simple computation yields
\begin{equation}\label{equ:Eder}
\frac{d}{dt}\bigg|_{t=0}\en_\om(\f+t\p)=\int\p\,\MA_\om(\f)
\end{equation}
for all $\f,\p\in\vec\cE^1$. The scaling action~\eqref{equ:scalefunc} preserves $\vec\cE^1$, and 
\begin{equation}\label{equ:MAhomog}
\MA_\om(t\cdot\f)=t_\star\MA_\om(\f)
\end{equation}
for any $\f\in\vec\cE^1$ and $t\in\R_{>0}$.

%
\subsection{Measures of finite energy}
Fix $\om\in\Amp(X)$, and set for simplicity $\cE^1:=\cE^1(\om)$, $\en=\en_\om$, $\MA=\MA_\om$. The \emph{energy}\footnote{This corresponds to $\en^\vee(\mu)$ in~\cite{trivval}; the change of notation is intended to make the formalism of the present paper easier to digest.} of a Radon probability measure $\mu\in\cM$ is defined as 
\begin{equation}\label{equ:envee}
\|\mu\|=\|\mu\|_\om:=\sup_{\f\in\cE^1}\left\{\en(\f)-\int\f\,\mu\right\}\in[0,+\infty]. 
\end{equation}
A \emph{maximizing net} for $\mu$ is defined as a net $(\f_i)$ in $\cE^1$ that computes $\|\mu\|$, \ie $\jj_\mu(\f_i)\to 0$, where
\begin{equation}\label{equ:Jmu}
\jj_\mu(\f):=\|\mu\|-\en(\f)+\int\f\,\mu\ge 0
\end{equation} 
for any $\f\in\cE^1$. See~\cite[\S\S9-10]{trivval} for details on this and on what follows. 

The energy functional $\|\cdot\|\colon\cM\to[0,+\infty]$ is convex; it is also lsc for the weak topology of $\cM$, as a simple regularization argument shows that $\cE^1$ can be replaced by $\PL\cap\PSH(\om)$ in the right-hand side of~\eqref{equ:envee}. Despite the chosen notation, the energy is not actually a norm; however, it satisfies
\begin{equation}\label{equ:homogenvee}
\|\mu\|=0\Leftrightarrow\mu=\mu_\triv,\quad\|t_\star\mu\|=t\|\mu\|,\quad\|\mu\|_{t\om}=t\|\mu\|_\om,\quad t\in\R_{>0}. 
\end{equation}
The space $\cM^1\subset\cM$ of \emph{measures of finite energy} is defined as the domain 
$$
\cM^1=\{\mu\in\cM\mid \|\mu\|<+\infty\}. 
$$
of the energy functional. By definition, any $\mu\in\cM^1$ satisfies $\cE^1\subset L^1(\mu)$, and the converse holds as well. 

In contrast with $\cE^1=\cE^1(\om)$, the space $\cM^1$ turns out to be independent of $\om\in\Amp(X)$, a property that plays a key role in the present paper. More precisely, for any $\om'\in\Amp(X)$ and $s\ge 1$ such that $s^{-1}\om\le\om'\le s\om$, we have 
\begin{equation}\label{equ:Eveelip}
s^{-C_n}\|\mu\|_\om\le\|\mu\|_{\om'}\le s^{C_n}\|\mu\|_\om
\end{equation}
for all $\mu\in\cM$, where $C_n:=2n^2+1$. See~\cite[Theorem~9.24]{trivval}.

\begin{exam}\label{exam:enMA} For any $\f\in\cE^1$, the Monge--Amp\`ere measure $\MA(\f)$ lies in $\cM^1$, and $\f$ computes its energy~\eqref{equ:envee}, \ie
\begin{equation}\label{equ:enMA}
\|\MA(\f)\|=\en(\f)-\int\f\,\MA(\f). 
\end{equation}
\end{exam}
When $n=1$, this can be rewritten as $\|\MA(\f)\|=\tfrac 12\int(-\f)\ddc\f$, whose analogue in the Archimedean case coincides with the classical Dirichlet functional. 

\medskip

As in the case of $\cE^1$, the \emph{strong topology} of $\cM^1$ is defined as the coarsest refinement of the (weak) subspace topology induced by $\cM$ such that the energy functional $\|\cdot\|\colon\cM^1\to\R$ becomes continuous. The strong topology of $\cM^1$ is also independent of $\om\in\Amp(X)$. 

By Example~\ref{exam:enMA}, the Monge--Amp\`ere operator defines a map $\MA\colon\cE^1\to\cM^1$. By~\cite[Theorem~A]{trivval}, this induces a topological embedding with dense image 
$$
\MA\colon\cE^1/\R\hto\cM^1,
$$ 
with respect to the strong topology on both sides. Further, a net $(\f_i)$ in $\cE^1$ and $\mu\in\cM^1$ satisfy $\MA(\f_i)\to\mu$ strongly in $\cM^1$ iff $(\f_i)$ is a maximizing net for $\mu$, \ie $\jj_\mu(\f_i)\to 0$. 

The Monge--Amp\`ere operator maps $\cE^1$ onto $\cM^1$ (and hence induces a homeomorphism $\cE^1/\R\simto\cM^1$) iff the \emph{envelope property} (aka \emph{continuity of envelopes}) holds for $\om$, which is the case when $X$ is smooth. 

We define the \emph{normalized potential} of a measure $\mu$ in the image of $\MA$ as the unique function $\f_\mu\in\cE^1$ such that 
\begin{equation}\label{equ:normpot}
\left\{
\begin{array}{ll}
\mu=\MA(\f_\mu) \\
\int\f_\mu\,\mu=0. 
\end{array}
\right. 
\end{equation}
By~\eqref{equ:enMA}, we have 
\begin{equation}\label{equ:ennormpot}
\|\mu\|=\en(\f_\mu), 
\end{equation}
and $\f_\mu$ is characterized as the unique function that achieves
$$
\|\mu\|=\sup\left\{\en(\f)\mid\f\in\cE^1,\,\int\f\,\mu=0\right\}.
$$
For any $\p\in\cE^1$, the normalized potential of $\MA(\p)\in\cM^1$ is given by
\begin{equation}\label{equ:normpotMA}
\f_{\MA(\p)}=\p-\int\p\,\MA(\p).
\end{equation}

\begin{exam}\label{exam:endirac} For any $v\in X^\an$, set 
$$
\|v\|:=\|v\|_\om:=\|\d_v\|\in[0,+\infty].
$$
Then $\|v\|\approx\te(v)$ (see~\cite[Theorem~11.1]{trivval}), and hence $\|v\|<\infty\Leftrightarrow v\in X^\lin$, see~\eqref{equ:Xlin}. By~\cite[Theorem~7.22]{nakstab1}, if $\om=c_1(L)$ with $L\in\Pic(X)_\Q$ ample, then $\|v\|=\|v\|_L$ coincides with the expected vanishing order $\esse(v)=\esse_L(v)$~\cite{FO18,BlJ}. Thus
$$
\|v\|=\lim_m \max\{v(D)\mid D\in |L|_\Q\text{ of $m$-basis type}\}
$$
where $D\in|L|_\Q$ is of \emph{$m$-basis type} if $D=\tfrac 1m\sum_i\div(s_i)$ for a basis $(s_i)$ of $\Hnot(X,mL)$ (see~\cite[Definition~0.1]{FO18}). For a divisorial valuation $v\in X^\div$, written $v=t\ord_F$ for a prime divisor $F$ on a smooth model $\pi\colon Y\to X$, we further have
\begin{equation}\label{equ:envdiv}
\|v\|=tV_\om^{-1}\int_0^{\te(v)}\vol(\pi^\star\om-\la F)\,d\la=tV_\om^{-1}\int_0^{+\infty}\vol(\pi^\star\om-\la F)\,d\la,
\end{equation}
see~\eqref{equ:Tthresh}. 
\end{exam}

In~\cite{trivval} and~\cite{nakstab1}, it was respectively shown that the strong topology of $\cM^1$ can be defined by a certain quasi-metric $\ii^\vee=\ii_\om^\vee$, and also by a (Darvas-type) metric $\dd_1=\dd_{1,\om}$, these two (quasi-)metrics being further H\"older equivalent on bounded sets. The metric $\dd_1$ on $\cM^1$ will not be used in this paper; as to the quasi-metric $\ii^\vee$, it will be advantageously replaced by the following equivalent, more natural quasi-metric, directly induced by Aubin's $\ii$-functional on $\cE^1$. 

\begin{thm}\label{thm:I} There exists a unique strongly continuous functional $\ii=\ii_\om\colon\cM^1\times\cM^1\to\R_{\ge 0}$ such that 
\begin{equation}\label{equ:IMA}
\ii(\MA(\f),\MA(\f))=\ii(\f,\p):=\int(\f-\p)(\MA(\p)-\MA(\f))
\end{equation}
for all $\f,\p\in\cE^1$. For all $\mu_1,\mu_2,\mu_3,\mu\in\cM^1$ and $\f\in\cE^1$, we further have: 
\begin{itemize}
\item[(i)] $\ii(\mu_1,\mu_2)=\ii(\mu_2,\mu_1)$, and $\ii(\mu_1,\mu_2)=0\Leftrightarrow\mu_1=\mu_2$; 
\item[(ii)] $\ii(\mu_1,\mu_2)\lesssim \ii(\mu_1,\mu_3)+\ii(\mu_3,\mu_2)$;  
\item[(iii)] $\ii$ is a quasi-metric on $\cM^1$ that defines the strong topology, and the quasi-metric space $(\cM^1,\ii)$ is complete;
\item[(iv)] $\ii(\mu,\MA(\f))\approx\jj_\mu(\f)$ (see~\eqref{equ:Jmu}); in particular, $\ii(\mu):=\ii(\mu,\mu_\triv)\approx\|\mu\|$.
\end{itemize}
Finally, for all $\f,\p\in\cE^1$ and $\mu,\nu\in\cM^1$ we have 
\begin{equation}\label{equ:BBGZ}
\left|\int(\f-\f')(\mu-\mu')\right|\lesssim\ii(\f,\f')^\a\ii(\mu,\mu')^{\frac 12} \max\{\ii(\f),\ii(\f'),\|\mu\|,\|\mu'\|\}^{\frac 12-\a}
\end{equation}
with $\a:=2^{-n}$.
\end{thm}

\begin{proof}[Proof of Theorem~\ref{thm:I}] By~\cite[\S10]{trivval}, Theorem~\ref{thm:I} is valid for a certain quasi-metric $\ii^\vee$ on $\cM^1$ in place of $\ii$, except for~\eqref{equ:IMA}, which is replaced by
\begin{equation}\label{equ:IveeMA}
\ii^\vee(\MA(\f),\MA(\p))\approx\ii(\f,\p).
\end{equation} 
Next, note that uniqueness is clear, by density of the image of $\MA$, and that~\eqref{equ:IMA} is equivalent to 
$$
\ii(\mu,\nu)=\ii(\f_\mu,\f_\nu)=\int\f_\mu\,\nu+\int\f_\nu\,\mu
$$
for all $\mu,\nu\in\cM^1$ in the image of $\MA$, by~\eqref{equ:normpotMA}. Pick $\mu,\mu',\nu,\nu'$ in the image of $\MA$, and write
$$
\ii(\mu,\nu)-\ii(\mu',\nu')=\int\f_\mu\,\nu+\int\f_\nu\,\mu-\int\f_{\mu'}\,\nu'-\int\f_{\nu'}\,\mu'
$$
$$
=\int(\f_\mu-\f_{\mu'})\,\nu+\int\f_{\mu'}(\nu-\nu')+\int(\f_\nu-\f_{\nu'})\,\mu+\int\f_{\nu'}(\mu-\mu'). 
$$
Using $\ii(\f_\mu,\f_{\mu'})\approx\ii^\vee(\mu,\mu')$ and $\ii(\f_\nu,\f_{\nu'})\approx\ii^\vee(\nu,\nu')$, the `$\ii^\vee$-version' of~\eqref{equ:BBGZ} yields a H\"older estimate
\begin{equation}\label{equ:Ihold}
\left|\ii(\mu,\nu)-\ii(\mu',\nu')\right|\lesssim\ii^\vee(\mu,\mu')^\a M^{1-\a}+\max\{\ii^\vee(\nu,\nu'),\ii^\vee(\mu,\mu')\}^{\frac 12}M^{\frac 12}
\end{equation}
with $M:=\max\{\|\mu\|,\|\mu'\|,\|\nu\|,\|\nu'\|\}$. This shows that $\ii$ is uniformly continuous on a dense subspace of the quasi-metric space $\cM^1\times\cM^1$, and hence that it admits a unique (uniformly) continuous extension $\cM^1\times\cM^1\to\R_{\ge 0}$ (see Lemma~\ref{lem:unifext}). Finally, \eqref{equ:IMA} and~\eqref{equ:IveeMA} imply $\ii(\mu,\nu)\approx\ii^\vee(\mu,\nu)$ for all $\mu,\nu\in\cM^1$, by continuity, and it is now immediate to see that (i)--(iv) and~\eqref{equ:BBGZ} follow from their version for $\ii^\vee$. 
\end{proof}

%
%
%
\section{Differentiability of the energy of a measure}\label{sec:diffen}
The main purpose of this section is to establish the differentiability of the energy of a measure with respect to the ample class (Theorem~\ref{thm:diffen} below, which corresponds to Theorem~A in the introduction), along with H\"older estimates for the derivative (Theorem~\ref{thm:twistedhold}) that will be the main ingredient for the continuity of the divisorial stability threshold (Theorem~B~(ii) in the introduction). 

As before, $X$ denotes an arbitrary (possibly non normal) projective variety, of dimension $n$, defined over an algebraically closed field $k$ (whose characteristic can in fact be arbitrary in this section). 

%
%
\subsection{The twisted Monge--Amp\`ere energy}\label{sec:twisted}
Fix for the moment $\om\in\Amp(X)$, with volume $V_\om=(\om^n)$, and pick $\theta\in\Num(X)$. 

\begin{defi}\label{defi:twisteden} We define the \emph{$\theta$-twisted Monge--Amp\`ere energy} $\en^\theta_\om\colon\vec\cE^1\to\R$ by setting 
$$
\en_\om^\theta(\f):=V_\om^{-1}(\om,\f)^n\cdot(\theta,0). 
$$
\end{defi} 
The right-hand side involves the energy pairing, see Theorem~\ref{thm:enpairing}. By~\eqref{equ:pairingmass}, we have
\begin{equation}\label{equ:twistedtrans}
\en_\om^\theta(\f+c)=\en_\om^\theta(\f)+c\tr_\om(\theta)\ \text{for }c\in\R,
\end{equation}
where $\tr_\om(\theta)=nV_\om^{-1}(\om^{n-1}\cdot\theta)$ (see~\S\ref{sec:voltr}). Further, $\en_\om^\theta(\f)$ is a linear function of $\theta\in\Num(X)$, and 
\begin{equation}\label{equ:twEhomog}
\en_\om^\theta(t\cdot\f)=t\en_\om^\theta(\f),\quad\en_{t\om}^\theta(t\f)=\en_\om^\theta(\f)
\end{equation}
for $t\in\R_{>0}$. By Theorem~\ref{thm:enpairing}, the restriction $\en_\om^\theta\colon \cE^1\to\R$ is continuous in the strong topology. 

As with the Monge--Amp\`ere energy $\en_\om$ (see~\eqref{equ:Emixed}, \eqref{equ:Eder}), straightforward computations yield a more standard description of the twisted Monge--Amp\`ere energy $\en_\om^\theta$ in terms of mixed Monge--Amp\`ere integrals, and also as the anti-derivative of $\f\mapsto(\om+\ddc\f)^{n-1}\wedge\theta$ (compare~\cite[\S3]{Che}\footnote{Note that $\en$ and $\en^\theta$ are respectively denoted by $\ii$ and $\jj$ in Chen's paper, and in much of the ensuing literature on the J-equation and the J-flow.} and~\cite[\S1]{BDL17} in the usual K\"ahler setting). 

\begin{prop}\label{prop:twisteden} For all $\f\in\vec\cE^1$ we have 
\begin{equation}\label{equ:twisteden} 
\en_\om^\theta(\f)=\sum_{j=0}^{n-1}V_\om^{-1}\int\f\,(\om+\ddc\f)^j\wedge\om^{n-1-j}\wedge\theta, 
\end{equation}
and
\begin{equation}\label{equ:twistedder}
\frac{d}{dt}\bigg|_{t=0}\en_\om^\theta(\f+t\p)=n V_\om^{-1}\int\p\,(\om+\ddc\f)^{n-1}\wedge\theta
\end{equation}
for all $\p\in\vec\cE^1$. 
\end{prop}

As a special case of~\cite[Theorem 3.25]{trivval}, we also have: 

\begin{lem}\label{lem:twistedconcave} If $\theta$ is nef, then the restriction of $\en_\om^\theta$ to $\cE^1(\om)=\vec\cE^1\cap\PSH(\om)$ is concave. 
\end{lem} 

The twisted energy is closely related to the differential of $\en_\om$ with respect to $\om$: 

\begin{lem}\label{lem:deren} For each $\f\in\vec\cE^1$, $\en_{\om}(\f)$ is a smooth function of $\om\in\Amp(X)$, with directional derivative
\begin{equation}\label{equ:twEder}
\nabla_\theta\en_\om(\f):=\frac{d}{dt}\bigg|_{t=0}\en_{\om+t\theta}(\f)
\end{equation}
given by 
\begin{equation}\label{equ:nablaE}
\nabla_\theta\en_\om(\f)=\en_\om^\theta(\f)-\tr_\om(\theta)\en_\om(\f).
\end{equation}
\end{lem}
Note that~\eqref{equ:nablaE} is translation invariant as a function of $\f$ (see~\eqref{equ:twistedtrans} and~\eqref{equ:MAenequiv}). 

\begin{proof} Smoothness follows from the fact that $V_\om=(\om^n)$ and $V_\om\en_\om(\f)=(\om,\f)^n\cdot(\theta,0)
$ are both polynomial functions of $\om$. Next, pick $\theta\in\Num(X)$. Then 
$$
(\om+t\theta,\f)^{n+1}=(\om,\f)^{n+1}+(n+1)t (\om,\f)^n\cdot(\theta,0)+O(t^2). 
$$
Combined with~\eqref{equ:logvol}, this yields
\begin{align*}
\en_{\om+t\theta}(\f) & =\left(1-t\tr_\om(\theta)+O(t^2)\right)\left(\en_\om(\f)+t \en_\om^\theta(\f)+O(t^2)\right)\\
& =\en_\om(\f)+t \left(\en_\om^\theta(\f)-\tr_\om(\theta)\en_\om(\f)\right)+O(t^2),
\end{align*}
which proves~\eqref{equ:nablaE}. 
\end{proof}
By~\eqref{equ:twEhomog} (or by differentiating~\eqref{equ:enhomog} with respect to $\om$), we have 
\begin{equation}\label{equ:nablaEhomog}
\nabla_\theta\en_\om(t\cdot\f)=t\nabla_\theta\en_\om(\f)\,\,\ \ \ \text{and   }\,\,\nabla_\theta\en_{t\om}(t\f)=\nabla_\theta\en_\om(\f)
\end{equation}
for $t\in\R_{>0}$. We also note:
\begin{lem}\label{lem:nablaenvee} For any $\f\in\cE^1$ we have 
$$
\nabla_\om\en_\om(\f)=\en_\om(\f)-\int\f\,\MA_\om(\f)=\|\MA_\om(\f)\|_\om. 
$$
\end{lem}
\begin{proof} For each $t>0$ we have $\en_{t\om}(t\f)=t\en_\om(\f)$, see~\eqref{equ:enhomog}. Differentiating this at $t=1$ yields 
$$
\nabla_\om\en_\om(\f)+\int\f\,\MA_\om(\f)=\en_\om(\f),
$$
see~\eqref{equ:Eder} and~\eqref{equ:twEder}. The rest follows from Example~\ref{exam:enMA}. 
\end{proof}
%
%
\subsection{The twisted energy of a measure}
Fix $\om\in\Amp(X)$, and recall that the Monge--Amp\`ere operator induces a topological embedding $\MA_\om\colon\cE^1(\om)/\R\hto\cM^1$ with dense image. 

\begin{prop-def}\label{prop:effehold} For any $\theta\in\Num(X)$, there exists a unique strongly continuous functional 
$\nabla_\theta\|\cdot\|_\om\colon\cM^1\to\R$ such that 
\begin{equation}\label{equ:nablavee}
\nabla_\theta\|\MA_\om(\f)\|_\om=\nabla_\theta\en_\om(\f)
\end{equation}
for all $\f\in\cE^1(\om)$. Further, 
\begin{equation}\label{equ:twisteddualholder}
\left|\nabla_\theta\|\mu\|_\om-\nabla_{\theta}\|\nu\|_\om\right|\lesssim\ii_\om(\mu,\nu)^\a \max\{\|\mu\|_\om,\|\nu\|_\om\}^{1-\a}\|\theta\|_\om
\end{equation}
for all $\mu,\nu\in\cM^1$, where $\a=\a_n\in (0,1)$ only depends on $n$. In particular,
\begin{equation}\label{equ:twisteddualhold2}
\left|\nabla_\theta\|\mu\|_\om\right|\lesssim\|\mu\|_\om\|\theta\|_\om. 
\end{equation} 
\end{prop-def}
By uniqueness, $\nabla_\theta\|\mu\|_\om$ is a linear function of $\theta\in\Num(X)$. The choice of notation will be justified by Theorem~\ref{thm:diffen} below. 

 \begin{rmk}\label{rmk:contenv} When $\om$ has the envelope property (\eg when $X$ is smooth), the Monge--Amp\`ere operator induces a homeomorphism $\MA_\om\colon\cE^1(\om)/\R\simto\cM^1$ (see~\cite{nama,trivval}), and $\nabla_{\theta}\|\cdot\|_\om$ is then simply obtained by transporting the translation invariant functional $\nabla_\theta\en_\om$. Equivalently, for any $\mu\in\cM^1$ with normalized potential $\f_\mu\in\cE^1(\om)$ (with respect to $\om$), we have $\nabla_\theta\|\mu\|_\om=\nabla_\theta\en_\om(\f_\mu)$. 
\end{rmk}

\begin{lem}\label{lem:nablaholder} The exists $\a=\a_n\in(0,1)$ such that 
\begin{equation}\label{equ:nablaholder}
|\nabla_\theta\en_\om(\f)-\nabla_\theta\en_\om(\p)|\lesssim\ii_\om(\f,\p)^{\a}\max\{\ii_\om(\f),\ii_\om(\p)\}^{1-\a}\|\theta\|_\om
\end{equation}
for all $\f,\p\in\cE^1(\om)$. 
\end{lem}

\begin{proof} By homogeneity, we may first assume $\|\theta\|_\om=1$, and hence $-\om\le\theta\le\om$. Set $\theta_+:=\theta+2\om$, $\theta_-:=2\om$. Then $\theta=\theta_+-\theta_-$ and $\om\le\theta_\pm\le 3\om$. By linearity, we may thus assume without loss that $\om\le\theta\le 3\om$, and hence $|\tr_\om(\theta)|\le 3$.  By translation invariance of $\nabla_\theta\en_\om$, we may further normalize $\f,\p$ by $\sup\f=\sup\p=0$. By Theorem~7.34 and~(7.29) of~\cite{trivval}, we then have 
$$
\left|\en_\om^\theta(\f)-\en_\om^\theta(\p)\right|=V_\om^{-1}\left|(\theta,0)\cdot(\om,\f)^n-(\theta,0)\cdot(\om,\p)^n\right|\lesssim\ii_\om(\f,\p)^\a\max\{\ii_\om(\f),\ii_\om(\p)\}^{1-\a}
$$
and
\begin{equation}\label{equ:enholder}
|\en_\om(\f)-\en_\om(\p)|=((n+1)V_\om)^{-1}\left|(\om,\f)^{n+1}-(\om,\p)^{n+1}\right|\lesssim\ii_\om(\f,\p)^\a\max\{\ii_\om(\f),\ii_\om(\p)\}^{1-\a}. 
\end{equation}
with $\a=\a_n\in(0,1)$. The result follows, by~\eqref{equ:nablaE}. 
\end{proof}

\begin{proof}[Proof of Proposition~\ref{prop:effehold}] Pick $\mu,\nu\in\cM^1$ in the image of $\MA_\om$, and denote by $\f_\mu,\f_\nu\in\cE^1(\om)$ their normalized potentials (with respect to $\om$), cf.~\eqref{equ:normpot}. By~\eqref{equ:nablavee} and~\eqref{equ:nablaholder}, we have 
$$
\left|\nabla_\theta\|\mu\|_\om-\nabla_\theta\|\nu\|_\om\right|=\left|\nabla_\theta\en_\om(\f_\mu)-\nabla_\theta\en_\om(\f_\nu)\right|
$$
$$
\lesssim\ii_\om(\f_\mu,\f_\nu)^\a\max\{\ii_\om(\f_\mu),\ii_\om(\f_\nu)\}^{1-\a}\approx\ii_\om(\mu,\nu)^\a\max\{\|\mu\|_\om,\|\nu\|_\om\}^{1-\a},
$$
see Theorem~\ref{thm:I}. This proves~\eqref{equ:twisteddualholder} when $\mu,\nu$ lie in the image of $\MA_\om$. As a result, $\nabla_{\theta}\|\cdot\|_\om$ is uniformly continuous on a dense subspace of the quasi-metric space $(\cM^1,\ii_\om)$. By Lemma~\ref{lem:unifext}, it thus admits a unique continuous extension $\cM^1\to\R$, which clearly still satisfies~\eqref{equ:twisteddualholder}. 
\end{proof}

\begin{prop}\label{prop:homogeffe} 
For all $\mu\in\cM^1$ and $t\in\R_{>0}$, we have 
\begin{equation}\label{equ:homogeffe}
\nabla_\theta\|\mu\|_{t\om}=\nabla_\theta\|\mu\|_\om,\quad\nabla_\theta\|t_\star\mu\|_\om=t\nabla_\theta\|\mu\|_\om. 
\end{equation}
\end{prop}
\begin{proof} By continuity, we may assume that $\mu=\MA_\om(\f)=V_\om^{-1}(\om+\ddc\f)^n$ with $\f\in\cE^1(\om)$. Then $\mu=\MA_{t\om}(t\f)$, $t_\star\mu=\MA_\om(t\cdot\f)$ (see~\eqref{equ:MAhomog}), and hence 
$$
\nabla_\theta\|\mu\|_{t\om}=\nabla_\theta\en_{t\om}(t\f)=\nabla_\theta\en(\f)=\nabla_\theta\|\mu\|_\om,
$$
$$
\nabla_\theta\|t_\star\mu\|_\om=\nabla_\theta\en_\om(t\cdot\f)=t\nabla_\theta\en_\om(\f)=t\nabla_\theta\|\mu\|_\om,
$$
see~\eqref{equ:nablaEhomog}. 
\end{proof}

Note also that
\begin{equation}\label{equ:Euler}
\nabla_\om\|\mu\|_\om=\|\mu\|_\om.
\end{equation}
for each $\mu\in\cM^1$,  by Lemma~\ref{lem:nablaenvee} and continuity. 

\begin{rmk} If we take for granted Theorem~\ref{thm:diffen}, then~\eqref{equ:homogeffe} can alternatively be obtained by differentiating~\eqref{equ:homogenvee}, while~\eqref{equ:Euler} becomes the Euler equation reflecting the homogeneity property $\|\mu\|_{t\om}=t\|\mu\|_\om$, $t\in\R_{>0}$. 
\end{rmk}

For later use, we also note:
\begin{lem}\label{lem:boundentwisted}
If $\mu$ lies in the image of $\MA_\om$, with normalized potential $\f_\mu$, then $|\en_\om^\theta(\f_\mu)|\lesssim\|\theta\|_\om\|\mu\|_\om$. 
\end{lem}
\begin{proof} By~\eqref{equ:nablaE}, \eqref{equ:nablavee} and~\eqref{equ:ennormpot}, we have 
$\en_\om^\theta(\f_\mu)=\nabla_\theta\|\mu\|_\om+\tr_\om(\theta)\|\mu\|_\om$. The desired estimate now follows from~\eqref{equ:twisteddualhold2} and~\eqref{equ:trbound}.
\end{proof}

%
%
\subsection{H\"older continuity of the energy of a measure}\label{sec:holdertwist}

The next result will be the main ingredient leading to the continuity of the divisorial stability threshold (see Theorem~\ref{thm:stabopen}). 

\begin{thm}\label{thm:twistedhold} Suppose $\om,\om'\in\Amp(X)$ and $\d\in [0,1]$ satisfy $e^{-\d}\om\le\om'\le e^\d\om$. Then
\begin{equation}\label{equ:twistedlip}
|\nabla_\theta\|\mu\|_\om-\nabla_\theta\|\mu\|_{\om'}|\lesssim\d^\a\|\mu\|_\om\|\theta\|_\om
\end{equation}
for all $\theta\in\Num(X)$ and $\mu\in\cM^1$, where $\a\in(0,1)$ only depends on $n$. 
\end{thm}
Combined with~\eqref{equ:Eveelip}, this implies that $\om\mapsto\nabla_\theta\|\mu\|_\om$ is locally H\"older continuous on $\Amp(X)$.

\begin{lem}\label{lem:Lipen} Suppose $\om,\om'\in\Amp(X)$ satisfy $\om\le\om'\le e^{\d}\om$ with $\d\in [0,2]$. Pick $\f\in\cE^1(\om)\subset\cE^1(\om')$ and $\theta\in\Num(X)$, and set $\mu:=\MA_\om(\f)$, $\mu':=\MA_{\om'}(\f)$. Then:
\begin{equation}\label{equ:Lipen}
|\en_\om(\f)-\en_{\om'}(\f)|\lesssim \d\|\mu\|_\om,
\end{equation}
\begin{equation}\label{equ:Lipentwist}
|\en_\om^\theta(\f)-\en_{\om'}^\theta(\f)|\lesssim \d\|\theta\|_\om\|\mu\|_\om, 
\end{equation}
\begin{equation}\label{equ:Lipennabla}
|\nabla_\theta\en_\om(\f)-\nabla_\theta\en_{\om'}(\f)|\lesssim \d\|\theta\|_\om\|\mu\|_\om, 
\end{equation}
\begin{equation}\label{equ:enprime}
\max\{\|\mu'\|_{\om},\|\mu\|_{\om'},\|\mu'\|_{\om'}\}\lesssim\|\mu\|_\om,
\end{equation}
\begin{equation}\label{equ:disthold} 
\ii_{\om'}(\mu,\mu')\lesssim\d\|\mu\|_\om. 
\end{equation}
\end{lem}

The proof relies on the following variant of~\cite[Lemma~9.25]{trivval}. 
 
\begin{lem}\label{lem:twistedconc} Assume given $\om,\om'\in\Amp(X)$, $s\ge 1$ such that $\om\le\om'\le s\om$, and a nef class $\theta\in\Nef(X)$. For any $0\ge \f\in\cE^1(\om)\subset\cE^1(\om')$, we then have 
 $$
 0\ge s^{-n}\en_\om(\f)\ge\en_{\om'}(\f)\ge s^n\en_\om(\f),\quad 0\ge s^{-n}\en_{\om}^\theta(\f)\ge \en_{\om'}^\theta(\f)\ge s^{n-1}\en_\om^\theta(\f).
 $$
 \end{lem}
 \begin{proof} The first estimates are already proved in~\cite[Lemma~9.25]{trivval}, and we simply repeat the argument to get the remaining ones. Since $\f\le 0$ and $\om\le\om'\le s\om$, \cite[Proposition~7.3]{trivval} implies 
 $$
0\ge (\om,\f)^n\cdot(\theta,0)\ge (\om',\f)^n\cdot(\theta,0)\ge (s\om,\f)^n\cdot(\theta,0)=s^n(\om,s^{-1}\f)^n\cdot(\theta,0).
$$
By concavity of $\en_\om^\theta$ (see Lemma~\ref{lem:twistedconcave}), we further have 
$(\om,s^{-1}\f)^n\cdot(\theta,0)\ge s^{-1}(\om,\f)^n\cdot(\theta,0)$, 
and hence 
$$
0\ge (\om,\f)^n\cdot(\theta,0)\ge (\om',\f)^n\cdot(\theta,0)\ge s^{n-1}(\om,\f)^n\cdot(\theta,0). 
$$
The result follows after dividing by $0\le V_\om\le V_{\om'}\le s^n V_\om$. 
\end{proof}

\begin{proof}[Proof of Lemma~\ref{lem:Lipen}] Note that the desired estimates are invariant under translation of $\f$. To prove~\eqref{equ:Lipen}, \eqref{equ:Lipentwist} and~\eqref{equ:Lipennabla}, we first normalize $\f$ by $\sup\f=0$. By Lemma~\ref{lem:twistedconc}, we then have $e^{-n\d}\en_\om(\f)\ge \en_{\om'}(\f)\ge e^{n\d}\en_\om(\f)$, and hence 
$$
|\en_\om(\f)-\en_{\om'}(\f)|\lesssim\d|\en_\om(\f)|,
$$
Further, $\f(v_\triv)=\sup\f=0$ yields $-\en_\om(\f)=\jj_{\om}(\f)\approx\ii_\om(\f)=\ii_\om(\mu)\approx\|\mu\|_\om$, 
see Theorem~\ref{thm:I}~(iv); this proves~\eqref{equ:Lipen}. 

Similarly, Lemma~\ref{lem:twistedconc} yields $e^{-n\d}\en_{\om}^\theta(\f)\ge \en_{\om'}^\theta(\f)\ge e^{(n-1)\d}\en_\om^\theta(\f)$, and hence
$$
|\en_\om^\theta(\f)-\en_{\om'}^\theta(\f)|\lesssim\d|\en_\om^\theta(\f)|. 
$$
By~\eqref{equ:nablaholder} and~\eqref{equ:trbound}, we further have
$$
|\nabla_\theta\en_\om(\f)|\lesssim\ii_\om(\f)\|\theta\|_\om\approx\|\mu\|_\om\|\theta\|_\om,\quad|\tr_\om(\theta)|\lesssim\|\theta\|_\om. 
$$
By~\eqref{equ:nablaE}, this yields 
$$
|\en^\theta_\om(\f)|\le|\tr_\om(\theta)||\en_\om(\f)|+|\nabla_\theta\en_\om(\f)|\lesssim\|\mu\|_\om\|\theta\|_\om,
$$
and~\eqref{equ:Lipentwist} follows. 

Finally, \eqref{equ:nablaE} implies
$$
\nabla_\theta\en_\om(\f)-\nabla_\theta\en_{\om'}(\f)=\en_\om^\theta(\f)-\en_{\om'}^\theta(\f)+\left(\tr_\om(\theta)-\tr_{\om'}(\theta)\right)\en_\om(\f)+\tr_{\om'}(\theta)(\en_{\om}(\f)-\en_{\om'}(\f)), 
$$
and~\eqref{equ:Lipennabla} thus follows from~\eqref{equ:Lipen} and~\eqref{equ:Lipentwist} together with $|\tr_\om(\theta)-\tr_{\om'}(\theta)|\lesssim\d\|\theta\|_\om$, which is readily checked. 

Next we turn to~\eqref{equ:enprime} and~\eqref{equ:disthold}, for which we now normalize $\f$ by $\int\f\,\mu=0$, so that $\f=\f_\mu$ is the normalized potential of $\mu$ with respect to $\om$. Then 
\begin{equation}\label{equ:estim}
\|\mu\|_{\om}=\en_{\om}(\f),\quad\sup\f=\ii_\om(\f)=\ii_\om(\mu)\approx\|\mu\|_\om,
\end{equation} 
see Theorem~\ref{thm:I}~(iv). By~\eqref{equ:Lipen} we further have $|\en_{\om'}(\f)|\lesssim\|\mu\|_\om$, and hence 
$$
\|\mu'\|_{\om'}\approx\jj_{\om'}(\f)=\sup\f-\en_{\om'}(\f)\lesssim\|\mu\|_\om,
$$
which yields~\eqref{equ:enprime} thanks to \eqref{equ:estim} and~\eqref{equ:Eveelip}. Finally, 
$$
\ii_{\om'}(\mu,\mu')\approx\jj_{\om',\mu}(\f)=\|\mu\|_{\om'}-\en_{\om'}(\f),
$$
see Theorem~\ref{thm:I}~(iv). By~\eqref{equ:Eveelip}, \eqref{equ:Lipen} and~\eqref{equ:estim}, this yields
$$
\ii_{\om'}(\mu,\mu')=(1+O(\d))\|\mu\|_\om-\en_{\om}(\f)+O(\d)\|\mu\|_\om=O(\d)\|\mu\|_\om,
$$
where the implicit constant in $O$ only depends on $n$. This proves~\eqref{equ:disthold}. 

\end{proof}

\begin{proof}[Proof of Theorem~\ref{thm:twistedhold}] Note first that $\om'':=e^{-\d}\om\le\om'\le e^\d\om$, and hence
$$
\om''\le\om'\le e^{2\d}\om'',\quad\om''\le\om\le e^\d\om''.
$$
Arguing successively with $\om'',\om'$, and with $\om'',\om$, and relying on~\eqref{equ:Eveelip}, it is thus enough to prove the result when $\om\le\om'\le e^\d\om$, which we assume from now on. Next, pick $\f\in\cE^1(\om)\subset\cE^1(\om')$ and set 
$$
\mu:=\MA_\om(\f),\quad\mu':=\MA_{\om'}(\f),
$$
so that 
$$
\nabla_\theta\|\mu\|_\om=\nabla_\theta\en_\om(\f),\quad\nabla_\theta\|\mu'\|_{\om'}=\nabla_\theta\en_{\om'}(\f),
$$
see~\eqref{equ:nablavee}. By~\eqref{equ:Lipennabla}, we have 
$$
|\nabla_\theta\|\mu\|_\om-\nabla_\theta\|\mu'\|_{\om'}|\lesssim\d\|\theta\|_\om\|\mu\|_\om. 
$$
On the other hand, \eqref{equ:twisteddualholder}, \eqref{equ:enprime} and~\eqref{equ:disthold} yield
$$
|\nabla_\theta\|\mu\|_{\om'}-\nabla_\theta\|\mu'\|_{\om'}|\lesssim\d^{\a}\|\mu\|_\om\|\theta\|_\om
$$
with $\a\in(0,1)$ only depending on $n$. Summing up these estimates yields~\eqref{equ:twistedlip} when $\mu$ lies in the image of $\MA_\om\colon\cE^1(\om)\to\cM^1$, and hence in general, by strong continuity of $\nabla_\theta\|\cdot\|_\om$ and $\nabla_\theta\|\cdot\|_{\om'}$ (recall that the strong topology of $\cM^1$ is independent of $\om$). 
\end{proof}

%
%
%
\subsection{Differentiability of the energy with respect to the class}

\begin{thm}\label{thm:diffen} For any $\mu\in\cM^1$, the function $\om\mapsto\|\mu\|_\om$ is locally $C^{1,\a}$ on $\Amp(X)$ for some $\a\in(0,1)$ only depending on $n$. Further, 
$$
\frac{d}{dt}\bigg|_{t=0}\|\mu\|_{\om+t\theta}=\nabla_\theta\|\mu\|_\om
$$
for all $\om\in\Amp(X)$ and $\theta\in\Num(X)$. 
\end{thm}

\begin{lem}\label{lem:diffen} Pick $\om\in\Amp(X)$ and $\theta\in\Num(X)$ such that $\|\theta\|_\om<1$, and hence $\om+\theta\in\Amp(X)$. Assume also given $A>0$ and $\mu\in\cM^1$ such that 
$$
A^{-1}\le V_\om\le A,\quad \|\mu\|_\om\le A.
$$
Then 
$$
\|\mu\|_{\om+\theta}\ge \|\mu\|_\om+\nabla_\theta\|\mu\|_\om-C\|\theta\|_\om^2
$$
for a constant $C>0$ only depending on $A$ and $n$. 
\end{lem}
\begin{proof} We may assume that $\e:=\|\theta\|_\om>0$, and write $\theta=\e\tilde\theta$ with $\|\tilde\theta\|_\om=1$. 
Note that 
\begin{equation}\label{equ:compom}
(1-\e)\om\le\om+\theta\le(1+\e)\om\le(1-\e)^{-1}\om.
\end{equation}
By density of the image of $\MA_\om\colon\cE^1(\om)\to\cM^1$, it is enough to prove the result when $\mu=\MA_\om(\f)$ with $\f\in\cE^1(\om)$. We further normalize $\f$ by $\int\f\,\mu=0$, so that $\f=\f_\mu$ is the normalized potential of $\mu$ with respect to $\om$. Then 
\begin{equation}\label{equ:supfi}
0\le\sup\f=\ii_\om(\f)\approx\en_\om(\f)=\|\mu\|_\om\le A. 
\end{equation}
On the other hand, \eqref{equ:compom} implies
$$
(1-\e)\f\in\cE^1((1-\e)\om)\subset\cE^1(\om+\theta),
$$ 
and hence  
$\|\mu\|_{\om+\theta}\ge\en_{\om+\theta}\left((1-\e)\f\right)$, see~\eqref{equ:envee}. Now
$$
(n+1)V_{\om+\theta}\en_{\om+\theta}\left((1-\e)\f)\right)=(\om+\theta,(1-\e)\f)^{n+1}
$$
$$
=\left((\om,\f)+\e(\tilde\theta,-\f)\right)^{n+1}=(\om,\f)^{n+1}+(n+1)(\om,\f)^n\cdot(\theta,-\e\f)+\e^2 a(\e)
$$
with 
$$
a(\e):=\sum_{j=2}^{n+1}{n+1\choose j}\e^{j-2}(\om,\f)^{n+1-j}\cdot(\tilde\theta,-\f)^j
$$
On the one hand, we have $(\om,\f)^n\cdot(\theta,-\e\f)=(\om,\f)^n\cdot(\theta,0)$, since 
$$
(\om,\f)^n\cdot (0,\f)=V_\om\int\f\,\MA_\om(\f)=0.
$$
On the other hand, injecting~\eqref{equ:supfi} in~\cite[Corollary~7.35]{trivval} yields, for $j=0,\dots,n+1$
$$
\left|(\om,\f)^{n+1-j}\cdot(\tilde\theta,-\f)^j\right|\le C
$$
where $C=C(A,n)>0$ only depends on $A$ and $n$. This shows
$$
V_\om^{-1}V_{\om+\theta}\en_{\om+\theta}((1-\e)\f))=\en_\om(\f)+\en_\om^\theta(\f) +O(\e^2)
$$
where the implicit constant in $O$ only depends on $A$ and $n$. Since
$$
V_\om^{-1}V_{\om+\theta}=1+\tr_\om(\theta)+O(\e^2), 
$$
see~\eqref{equ:logvol}, we infer
$$
\|\mu\|_{\om+\theta}\ge\en_{\om+\theta}\left((1-\e)\f\right)=\left(1-\tr_\om(\theta)+O(\e^2)\right)\left(\en_\om(\f)+\en_\om^\theta(\f)+O(\e^2)\right) 
$$
$$
=\en_\om(\f)+\nabla_\theta\en_\om(\f)+O(\e^2) =\|\mu\|_\om+\nabla_\theta\|\mu\|_\om+O(\e^2) , 
$$
using~\eqref{equ:nablaE}, \eqref{equ:nablavee}, \eqref{equ:supfi} as well as $\tr_\om(\theta)=O(\e)$ and $\en_\om^\theta(\f)=O(\e)$, where the implicit constant only depends on $A$ and $n$, see~\eqref{equ:trbound} and Lemma~\ref{lem:boundentwisted}. The result follows. 

\end{proof}

\begin{proof}[Proof of Theorem~\ref{thm:diffen}] Pick $\om\in\Amp(X)$ and $\theta\in\Num(X)$ such that $\|\theta\|_\om\le 1/2$, and note that $\om+\theta\in\Amp(X)$ and $\|\theta\|_{\om+\theta}\le 2\|\theta\|_\om$. By Lemma~\ref{lem:diffen} we thus have 
$$
\|\mu\|_{\om+\theta}\ge\|\mu\|_\om+\nabla_\theta\|\mu\|_\om-C\|\theta\|_\om^2
$$
and 
$$
\|\mu\|_{\om}\ge\|\mu\|_{\om+\theta}-\nabla_\theta\|\mu\|_{\om+\theta}-C\|\theta\|_\om^2,
$$
where $C>0$ is independent of $\theta$. On the other hand, Theorem~\ref{thm:twistedhold} yields
$\nabla_\theta\|\mu\|_{\om+\theta}=\nabla_\theta\|\mu\|_\om+o(\|\theta\|)$ as $\theta\to 0$, and hence 
$$
\|\mu\|_{\om+\theta}=\|\mu\|_\om+\nabla_\theta\|\mu\|_\om+o(\|\theta\|_\om). 
$$
This proves that $\om\mapsto\|\mu\|_\om$ is differentiable, with differential equal to $\nabla_\theta\|\mu\|_\om$. By Theorem~\ref{thm:twistedhold}, the latter is locally $C^\a$ with respect to $\om$, and we conclude that $\om\mapsto\|\mu\|_\om$ is locally $C^{1,\a}$. 
\end{proof}

\begin{rmk} Assume that all ample classes on $X$ have the envelope property (\eg $X$ is smooth). Then any $\mu\in\cM^1$ lies in the image of $\MA_\om$, and hence admits a normalized potential $\f_\om=\f_{\om,\mu}$, characterized as the unique maximizer of the concave optimization problem 
\begin{equation}\label{equ:concopt}
\|\mu\|_\om=\sup\{\en_{\om}(\f)\mid\f\in\cE^1(\om),\,\int\f\,\mu=0\}.
\end{equation}
Since $\nabla_\theta\|\mu\|_\om=\nabla_\theta\en_\om(\f_\om)$, Theorem~\ref{thm:diffen} asserts that the differential of~\eqref{equ:concopt} with respect to the parameter $\om$ coincides with the differential of $\en_\om$ evaluated at the unique maximizer $\f_\om$. This is analogous, for instance, to the well-known differentiability of the distance to a closed convex subset in a Hilbert space. 
\end{rmk}

%
%
\subsection{The case of a Dirac mass}\label{sec:endiv}
In this section, we fix $\om\in\Amp(X)$ and $\theta\in\Num(X)$. For any $v\in X^\an$, recall that $v\in X^\lin\Leftrightarrow\d_v\in\cM^1$, and that we then write, for simplicity, $\|v\|_\om:=\|\d_v\|_\om$; we similarly set
\begin{equation}\label{equ:enshort}
\nabla_\theta\|v\|_\om:=\nabla_\theta\|\d_v\|_\om
\end{equation}
By~\eqref{equ:homogenvee} and~\eqref{equ:homogeffe}, we have
\begin{equation}\label{equ:homogdirac}
\|tv\|_\om=t\|v\|_\om,\quad\nabla_\theta\|tv\|_\om=t\nabla_\theta\|v\|_\om
\end{equation}
for $t\in\R_{>0}$. 

In what follows, we provide explicit formulas for these invariants when $v\in X^\div$ is a divisorial valuation, \ie $v=t\ord_F$ for a prime divisor $F\subset Y$ on a smooth birational model $\pi\colon Y\to X$ and $t\in\Q_{>0}$. By~\eqref{equ:homogdirac}, we assume $t=1$, to simplify notation.

\begin{thm}\label{thm:endiv} Pick a prime divisor $F\subset Y$ on a smooth birational model $\pi\colon Y\to X$, and set $v:= \ord_F$. Then: 
\begin{equation}\label{equ:energydirac}
\|v\|_\om=V_\om^{-1}\int_0^{+\infty}\vol(\om-\la F)\,d\la; 
\end{equation}
\begin{equation}\label{equ:DL}
\nabla_\theta\|v\|_\om=V_\om^{-1}\int_0^{+\infty}\nabla_\theta\vol(\om-\la F)\,d\la-\tr_\om(\theta)\|v\|_\om. 
\end{equation}
\end{thm}
For simplicity, we have set 
$$
\vol(\om-\la F):=\vol(\pi^\star\om-\la F),\quad\nabla_\theta\vol(\om-\la F):=\nabla_{\pi^\star\theta}\vol_Y(\pi^\star\om-\la F),
$$
see~\S\ref{sec:voltr} for the notation. Recall from~\eqref{equ:Tthresh} that $\vol(\om-\la F)>0\Longleftrightarrow\la<\tee_\om(v)$, so that the above integrals actually take place over $[0,\tee_\om(v)]$. 

Up to the factor $V_\om$, the right-hand side of~\eqref{equ:DL} is precisely the quantity that features in the Dervan--Legendre $\b$-invariant~\cite{DL}, see Lemma~\ref{lem:DL} below. 

\begin{proof}[Proof of Theorem~\ref{thm:endiv}] By~\eqref{equ:Eveelip}, the left-hand side of~\eqref{equ:energydirac} is continuous with respect to $\om$. The same holds for the right-hand side, using $\int_0^{+\infty}\vol(\om-\la F)\,d\la=\int_0^{\tee_\om(v)}\vol(\om-\la F)\,d\la$ and the continuity $\om\mapsto\tee_\om(v)$ and the volume function. 
By~\eqref{equ:envdiv}, \eqref{equ:energydirac} holds when $\om$ is rational, and the general case follows. As a consequence, for any $t\in\R$ small enough we get
$$
\|v\|_{\om+t\theta}=V_{\om+t\theta}^{-1}\int_0^{+\infty} f(t,\la)d\la
$$
with $f(t,\la):=\vol(\om+t\theta-\la F)$. By~\eqref{equ:vol'}, we further have 
$$
\frac{\partial}{\partial t}f(t,\la)=g(t,\la):=\nabla_\theta\vol(\om+t\theta-\la F)
$$
for $\la\ne\la(t):=\te_{\om+t\theta}(v)$. On the other hand, Theorem~\ref{thm:diffen} yields 
$$
\nabla_\theta\|v\|_\om=\frac{d}{dt}\bigg|_{t=0}\|v\|_{\om+t\theta},
$$
and a simple computation thus shows that~\eqref{equ:DL} is equivalent to  
$$
\frac{d}{dt}\bigg|_{t=0}\int_0^{+\infty} f(t,\la)d\la=\int_0^{+\infty} g(0,\la)d\la.
$$
While this appears to be a simple differentiation under the integral sign, the slight twist here is that $f(t,\la)$ might fail to be differentiable at $(t,\la(t))$. To circumvent this, note first that $t\mapsto\la(t)$ is locally Lipschitz continuous. Indeed, pick $C>0$ such that $-C\om\le\theta\le C\om$. For any $t,s\in\R$ we then have 
$$
\pi^\star(\om+(t+s)\theta)-\la(t+s)F\in\Psef(Y)\Longrightarrow\pi^\star((1+|s|C)\om+t\theta)-\la(t+s)F\in\Psef(Y)
$$
$$
\Longrightarrow \la(t+s)\le(1+C|s|)\la(t),
$$
which proves that $t\mapsto\la(t)$ is locally Lipschitz, since it is locally bounded. 

Next, we may assume $\theta\in\Amp(X)$, by linearity of the desired formula with respect to $\theta$. In that case, $t\mapsto\la(t)$ is further strictly increasing, and $f(t,\la)$ is thus $C^1$ on $\{(t,\la)\mid t>0,\quad\la<\la(0)\}$. By usual differentiation under the integral sign, we thus have 
$$
\frac{d}{dt}\bigg|_{t=0_+}\int_0^{\la(0)} f(t,\la)d\la=\int_0^{\la(0)}g(0,\la)d\la,
$$
and it remains to see 
$$
\int_{\la(0)}^{+\infty}f(t,\la)\,d\la=\int_{\la(0)}^{\la(t)} f(t,\la)\,d\la=o(t)
$$ 
as $t\to 0_+$. But $|\la(t)-\la(0)|\le Ct$, and $f(t,\la(t))=0$, which yields the desired estimate by (uniform) continuity of $f$, thereby finishing the proof. 
\end{proof}
%
%
%
%
\section{Entropy and the $\d$-invariant}\label{sec:entropy}
In what follows, $(X,B;\om)$ denotes a \emph{polarized pair}, \ie $X$ is a normal projective variety, $B$ is any $\Q$-Weil divisor such that $K_{X,B}:=K_X+B$ is $\Q$-Cartier, and $\om\in\Amp(X)$ is an ample numerical class. We introduce the entropy functional of $(X,B)$, defined on the space of divisorial measures, and study the associated $\d$-invariant. When $(X,B)$ is subklt, we prove that the entropy admits a natural extension to all measures of finite energy, and show that the $\d$-invariant is a threshold for Ding-stability. 

%
%
%
\subsection{Divisorial measures}

\begin{defi}\label{defi:divmes} We define a \emph{divisorial measure} on $X^\an$ as a Radon probability measure with support a finite subset of $X^\div$. We denote by $\cM^\div$ the set of such measures. 
\end{defi}
A divisorial measure $\mu\in\cM^\div$ is thus a measure of the form 
\begin{equation}\label{equ:mudiv}
\mu=\sum_i m_i\d_{v_i}
\end{equation}
for a finite set $(v_i)$ of divisorial valuations and $m_i\in\R_{\ge 0}$ such that $\sum_i m_i=1$. In other words, $\cM^\div$ is the convex hull of $X^\div\hto\cM$. 

As $X^\div\subset X^\an$ is stable under the scaling action of $\Q_{>0}$, the same holds for $\cM^\div$, \ie 
$$
\mu\in\cM^\div,\,t\in\Q_{>0}\Longrightarrow t_\star\mu\in\cM^\div.
$$
 \begin{exam}\label{exam:MAPL} For any $\f\in\PL\cap\PSH(\om)$, the Monge--Amp\`ere measure $\MA(\f)$ lies in $\cM^\div$.
\end{exam}

\begin{lem}\label{lem:divdense} The space $\cM^\div$ of divisorial measures sits as a dense subset of $\cM^1$ for the strong topology. 
\end{lem}
\begin{proof} For any $v\in X^\div\subset X^\lin$, $\d_v$ lies in $\cM^1$, and hence $\cM^\div\subset\cM^1$, by convexity of $\|\cdot\|_\om\colon\cM\to [0,+\infty]$. Next, pick $\mu\in\cM^1$, and choose a maximizing sequence $(\f_j)$ in $\PL\cap\PSH(\om)$ for $\mu$. Then $\mu_j:=\MA(\f_j)$ lies in $\cM^\div$ (see Example~\ref{exam:MAPL}), and $\mu_j\to\mu$ strongly in $\cM^1$. Thus $\cM^\div$ is strongly dense in $\cM^1$.
\end{proof}

%
\subsection{The entropy functional of a pair}\label{sec:entfunc}
Following~\cite{BHJ1}, we define the \emph{entropy functional} 
$$
\Ent_{X,B}\colon\cM^\div\to\R
$$
of the pair $(X,B)$ by setting
$$
\Ent_{X,B}(\mu):=\int\ld_{X,B}\,\mu
$$
for all $\mu\in\cM^\div$. Here $\ld_{X,B}\colon X^\div\to\Q$ is the classical log discrepancy function of the pair $(X,B)$, see Appendix~\ref{sec:logdisc}. 
\begin{rmk} When $k=\C$, the above (non-Archimedean) entropy computes the `slope at infinity' of the usual (relative) entropy functional along certain rays of smooth volume forms (see~\cite[Theorem~3.6]{BHJ2}); this explains the chosen terminology. 
\end{rmk}
The entropy functional is clearly affine on the convex set $\cM^\div$, and simply given by 
$$
\Ent_{X,B}(\mu)=\sum_i m_i\ld_{X,B}(v_i)
$$
if $\mu=\sum_i m_i\d_{v_i}$ as in~\eqref{equ:mudiv}. Note also that
\begin{equation}\label{equ:homogent}
\Ent_{X,B}(t_\star\mu)=t\Ent_{X,B}(\mu) 
\end{equation}
for $\mu\in\cM^\div$ and $t\in\Q_{>0}$, by homogeneity of $\ld_{X,B}$.

%
\subsection{The $\d$-invariant}
Consider as above an arbitrary polarized pair $(X,B;\om)$. 

\begin{defi} The \emph{$\d$-invariant} of $(X,B;\om)$ is defined as 
\begin{equation}\label{equ:delta}
\d(X,B;\om):=\inf_{\mu\in\cM^\div\setminus\{\mu_\triv\}}\frac{\Ent_{X,B}(\mu)}{\|\mu\|_\om}\in\R\cup\{-\infty\}. 
\end{equation}
\end{defi}
Note that 
$$
\d(X,B;s\om)=s^{-1}\d(X,B;\om)
$$
for $s\in\R_{>0}$, by~\eqref{equ:homogenvee}.

\begin{thm}\label{thm:delta} For any polarized pair $(X,B;\om)$, the following holds: 

\begin{itemize}
\item[(i)] $\d(X,B;\om)>-\infty\Longleftrightarrow (X,B)\text{ sublc }\Longleftrightarrow\d(X,B;\om)\ge 0$; 
\item[(ii)] $\d(X,B;\om)>0\Longleftrightarrow (X,B)\text{ subklt}$; 
\item[(iii)] if $(X,B)$ is sublc, then 
\begin{equation}\label{equ:deltaval}
\d(X,B;\om)=\inf_{v\in X^\div\setminus\{v_\triv\}}\frac{\ld_{X,B}(v)}{\|v\|_\om}.
\end{equation}
\item[(iv)] if $\om'\in\Amp(X)$ and $s\ge 1$ satisfy $s^{-1}\om\le\om'\le s\om$, then 
$$
s^{-C_n}\d(X,B;\om)\le\d(X,B;\om')\le s^{C_n}\d(X,B;\om)
$$
where $C_n:=2n^2+1$. 
\end{itemize}
\end{thm}
Note that (iv) implies in particular that $\d(X,B;\om)$ is continuous with respect to $\om\in\Amp(X)$ (see~\cite[Theorem~1.7]{Zha} for an extension to the big cone). 

When $\om=c_1(L)$ with $L\in\Pic(X)_\Q$ ample, we have $\|v\|=\esse(v)$ (see Example~\ref{exam:endirac}), and~\eqref{equ:deltaval} then shows that $\d(X,B;L)$ coincides with the usual $\d$-invariant~\cite{FO18,BlJ}. 

In what follows, we write for simplicity $\|\cdot\|=\|\cdot\|_\om$ when $\om$ is fixed. 

\begin{lem}\label{lem:distaff} For any $\mu\in\cM^1$ and $t\in[0,1]$ we have 
$$
\|(1-t)\mu_\triv+t\mu\|\lesssim t^{1+\e}\|\mu\|
$$
for a constant $\e>0$ only depending on $n$. 
\end{lem}
\begin{proof} Pick a maximizing sequence $(\f_{t,j})$ for $\mu_t:=(1-t)\mu_\triv+t\mu$, so that $\mu_{t,j}:=\MA(\f_{t,j})\to\mu_t$ strongly in $\cM^1$. Then
\begin{equation}\label{equ:aff0}
\|\mu_{t,j}\|\approx\ii(\mu_{t,j})=\ii(\f_{t,j})=\int\f_{t,j}(\mu_\triv-\mu_{t,j})=t\int\f_{t,j}(\mu_\triv-\mu)+\e_{t,j}
\end{equation}
with $\e_{t,j}:=\int\f_{t,j}(\mu_t-\mu_{t,j})$. By~\eqref{equ:BBGZ} and the quasi-triangle inequality for $\ii$, we have 
\begin{equation}\label{equ:aff1}
\left|\int\f_{t,j}(\mu_\triv-\mu)\right|\lesssim\|\mu_{t,j}\|^\a\max\{\|\mu_{t,j}\|,\|\mu\|\}^{1-\a}
\end{equation}
and
\begin{equation}\label{equ:aff2}
|\e_{t,j}|\lesssim\ii(\mu_t,\mu_{t,j})^{\frac 12}\max\{\|\mu_t\|,\|\mu_{t,j}\|\}^{\frac 12}, 
\end{equation}
with $\a:=2^{-n}$. Combining~\eqref{equ:aff0}, \eqref{equ:aff1} and~\eqref{equ:aff2}, we infer
$$
\|\mu_{t,j}\|\lesssim t\|\mu_{t,j}\|^\a\max\{\|\mu_{t,j}\|,\|\mu\|\}^{1-\a}+\ii(\mu_t,\mu_{t,j})^{\frac 12}\max\{\|\mu_t\|,\|\mu_{t,j}\|\}^{\frac 12}. 
$$
Letting $j\to\infty$, this yields
$$
\|\mu_t\|\lesssim t\|\mu_t\|^\a\max\{\|\mu_t\|,\|\mu\|\}^{1-\a}.
$$
By convexity of the energy functional $\|\cdot\|$, we further have 
$$
\|\mu_t\|\le t\|\mu\|+(1-t)\|\mu_\triv\|\le\|\mu\|. 
$$
We thus get $\|\mu_t\|\lesssim t\|\mu_t\|^\a\|\mu\|$, and hence $\|\mu_t\|\lesssim t^{(1-\a)^{-1}}\|\mu\|$, which yields the result with $\e>0$ such that $1+\e=(1-\a)^{-1}$ (\ie $\e=(2^n-1)^{-1}$). 
\end{proof}

\begin{proof}[Proof of Theorem~\ref{thm:delta}] Assume $\d:=\d(X,B;\om)>-\infty$. Pick $v\in X^\div$, and set as above
$$
\mu_t:=(1-t)\mu_\triv+t\d_v\in\cM^\div
$$ 
for $t\in [0,1]$. By Lemma~\ref{lem:distaff}, we have 
$$
t\ld_{X,B}(v)=\Ent_{X,B}(\mu_t)\ge \d\|\mu_t\|=O(t^{1+\e})
$$
with $\e>0$, and hence $\ld_{X,B}(v)\ge 0$. This shows that $(X,B)$ is sublc, which in turn trivially implies $\d(X,B;\om)\ge 0$. This proves (i). 

Assume now $(X,B)$ subklt. By Theorem~\ref{thm:ldan} and Example~\ref{exam:endirac}, there exists $\a>0$ such that $\ld_{X,B}(v)\ge\a\te(v)\approx\a\|v\|$ for all $v\in X^\div$, and hence $\d(X,B;\om)>0$. The converse trivially holds, and (ii) follows.  

Next, assume that $(X,B)$ is sublc. Suppose given $c\in\R_{\ge 0}$ such that $\ld_{X,B}(v)\ge c\|v\|$ for all $v\in X^\div$, and pick $\mu\in\cM^\div$. Write $\mu=\sum_i m_i\d_{v_i}$ for a finite set $(v_i)$ in $X^\div$ and $m_i\in\R_{>0}$, such that $\sum_i m_i=1$. Then 
\begin{equation}\label{equ:entmuen}
\Ent_{X,B}(\mu)=\sum_i m_i\ld_{X,B}(v)\ge\sum_i m_i c\|\d_{v_i}\|\ge c\|\mu\|,
\end{equation}
by convexity of $\|\cdot\|$. This proves (iii). Finally, (iv) is a direct consequence of~\eqref{equ:Eveelip}. 
\end{proof}

\begin{rmk}\label{rmk:deltaval} Note that~\eqref{equ:deltaval} fails when $(X,B)$ is not sublc. Indeed, using dual cone complexes as in Appendix~\ref{sec:logdisc}, one can show that the right-hand side of~\eqref{equ:deltaval} is finite for any polarized pair $(X,B;\om)$, whether $(X,B)$ is sublc or not. 
\end{rmk}

%
%
\subsection{Extending the entropy functional}\label{sec:extent}
We assume in this section that the pair $(X,B)$ is \textbf{subklt} (\ie $\ld_{X,B}>0$ on $X^\div\setminus\{v_\triv\}$). The log discrepancy function $\ld_{X,B}\colon X^\div\to\Q_{\ge 0}$ then admits a greatest lsc extension $\ld_{X,B}\colon X^\an\to[0,+\infty]$, which further satisfies 
\begin{equation}\label{equ:XBfld}
\left\{v\in X^\an\mid \ld_{X,B}(v)<+\infty\right\}=X^\fld\subset X^\val,
\end{equation}
see Appendix~\ref{sec:logdisc}. The entropy functional thus admits a natural extension
$$
\Ent_{X,B}\colon\cM\to[0,+\infty], 
$$
defined by
$$
\Ent_{X,B}(\mu):=\int\ld_{X,B}\,\mu
$$
for any $\mu\in\cM$. By~\eqref{equ:XBfld}, we have
\begin{equation}\label{equ:entfld}
\Ent_{X,B}(\mu)<+\infty\Longrightarrow\mu(X^\an\setminus X^\fld)=0
\end{equation}
for any $\mu\in\cM$. 

Note also that $\Ent_{X,B}\colon\cM\to[0,+\infty]$ is lsc in the weak topology, since $\ld_{X,B}$ is lsc on $X^\an$. Its restriction to $\cM^1$ is thus \emph{a fortiori} strongly lsc, but it is not strongly continuous in general. 

\begin{exam}\label{exam:Tcurve} Assume $X=(X,0)$ is a smooth curve, and normalize $\om\in\Amp(X)$ by $V_\om=1$. Pick any  finite subset $\Sigma\in X(k)$ of cardinality $N$, and set
$$
\mu_\Sigma:=\frac{1}{N}\sum_{p\in\Sigma}\d_{\ord_p}\in\cM^\div.
$$
By~\cite[Example~3.19]{trivval}, we have $\mu_\Sigma=\MA(\f_\Sigma)$ where $\f_\Sigma\in\PL\cap\PSH(\om)$ is given on each ray $(t\ord_q)_{t\in [0,+\infty]}$ of $X^\an$ by 
$$
\f_\Sigma(t\ord_q)=\left\{
\begin{array}{ll}
\tfrac 1N\max\{-t,-1\} & \text{ if } q\in \Sigma; \\
0 & \text{ otherwise}. 
\end{array}
\right. 
$$
As $N\to\infty$, we have $\f_\Sigma\to 0$ uniformly on $X^\an$; thus $\f_\Sigma\to 0$ strongly in $\cE^1$, and hence $\mu_\Sigma\to\mu_\triv$ strongly in $\cM^1$, by strong continuity of the Monge--Amp\`ere operator. However, $\Ent_{X,B}(\mu_\Sigma)=\frac 1N\sum_{p\in\Sigma}\ld_X(\ord_p)=1\ne 0=\Ent_{X,B}(\mu_\triv)$. 
\end{exam}

The following approximation result, which can be viewed as a non-Archimedean analogue of~\cite[Lemma~3.1]{BDL17}, plays a key role in what follows. 

\begin{thm}\label{thm:entapprox} Any $\mu\in\cM^1$ can be written as the strong limit of a sequence $(\mu_i)$ in $\cM^\div$ such that $\Ent_{X,B}(\mu_i)\to\Ent_{X,B}(\mu)$. Equivalently, $\Ent_{X,B}\colon\cM^1\to[0,+\infty]$ is the greatest (strongly) lsc extension of the entropy functional. 
\end{thm}

When $X$ is smooth and $B=0$, the result follows from~\cite[Proposition~6.3]{Li22} and its proof (itself based in part on the authors' preprint~\cite{trivvalold}). 

\begin{cor}\label{cor:sigM1} The $\d$-invariant of any polarized subklt pair $(X,B;\om)$ satisfies 
\begin{equation}\label{equ:deltames}
\d(X,B;\om)=\inf_{\mu\in\cM^1\setminus\{\mu_\triv\}}\frac{\Ent_{X,B}(\mu)}{\|\mu\|}=\inf_{\mu\in\cM^1,\,\|\mu\|=1}\Ent_{X,B}(\mu). 
\end{equation}
\end{cor} 

\begin{proof} By definition of $\d:=\d(X,B;\om)$, we have $\Ent_{X,B}\ge\d\|\cdot\|$ on $\cM^\div$. By Theorem~\ref{thm:entapprox} and the strong continuity of $\|\cdot\|$, this inequality extends to $\cM^1$; this proves~\eqref{equ:deltames}, the right-hand equality being a consequence of the homogeneity of $\Ent_{X,B}$ and $\|\cdot\|$ with respect to the scaling action of $\R_{>0}$ on $\cM^1$. 
\end{proof}

We will rely several times on the following simple observation:
\begin{lem}\label{lem:lsccv} Let $f:Z\to\R\cup\{+\infty\}$ be an lsc function on a topological space, and pick a convergent net $x_i\to x$ in $Z$ such that $f(x_i)\le f(x)$ for all $i$. Then $f(x_i)\to f(x)$. 
\end{lem}
\begin{proof} Since $f$ is lsc, we have $f(x)\le\liminf_i f(x_i)$, while the assumption yields $\limsup_i f(x_i)\le f(x)$. The result follows.
\end{proof}

The next two results are the key ingredients in the proof of Theorem~\ref{thm:entapprox}. 

\begin{lem}\label{lem:approx} Assume $X$ is smooth. Let $\cX$ be an snc test configuration for $X$, and denote by $p_\cX\colon X^\an\to\D_\cX$ the retraction onto the associated dual complex $\D_\cX\hto X^\an$ (see~\S\ref{sec:ldtc}). Then: 
\begin{itemize}
\item[(i)] any $\f\in\PSH(\om)$ satisfies $\f|_{\D_\cX}\in\Cz(\D_\cX)$ and $\f\le \f\circ p_\cX$; 
\item[(ii)] $\PSH_{\sup}(\om)|_{\D_{\cX}}$ has compact closure in $\Cz(\D_\cX)$ for the topology of uniform convergence; 
\item[(iii)] for any weakly convergent net $\mu_i\to\mu$ in $\cM(\D_\cX)\hto\cM$, $\int\f\,\mu_i\to\int\f\,\mu$ uniformly with respect to $\f\in\PSH(\om)$. 
\end{itemize}
\end{lem}
Here $\PSH_{\sup}(\om)=\{\f\in\PSH(\om)\mid\sup\f=0\}$. 

\begin{proof} (i) and (ii) follow from Theorems~A.1 and~A.4 of~\cite{trivval} (the latter being a consequence of the uniform Lipschitz estimates of~\cite[Theorem~C]{siminag}). Now consider a convergent net $\mu_i\to\mu$ as in (iii). Since $\int\f(\mu_i-\mu)$ is invariant under translation of $\f$ by a constant, it is enough to show $\int\f\,\mu_i\to\int\f\,\mu$ uniformly for $\f\in\PSH_{\sup}(\om)$. Denote by $K$ the (compact) closure of $\PSH_{\sup}(\om)|_{\D_{\cX}}$ in $\Cz(\D_\cX)$, equipped with the supnorm metric. By assumption, the functions $F_i\in\Cz(K)$ defined by $F_i(\f):=\int\f\,\mu_i$ converge pointwise to $F(\f):=\int\f\,\mu$. Since $\mu_i$ is a probability measures, $F_i$ is $1$-Lipschitz, and (iii) now follows from the `easy direction' of the Arzel\`a--Ascoli theorem. 
\end{proof}

\begin{lem}\label{lem:ldlisse} Assume $X$ is smooth. Pick $\mu\in\cM$, and set 
$$
\mu_\cX:=(p_\cX)_\star\mu\in\cM(\D_\cX)\hto\cM
$$ 
for each snc test configuration $\cX$. Then:
\begin{itemize}
\item[(i)] $\mu_\cX\to\mu$ weakly in $\cM$; 
\item[(ii)] $\int\f\,\mu_\cX\to\int\f\,\mu$ in $\R\cup\{-\infty\}$ for all $\f\in\PSH(\om)$; 
\item[(iii)] $\Ent_{X,B}(\mu_\cX)\to\Ent_{X,B}(\mu)$ in $[0,+\infty]$. 
\end{itemize}
\end{lem}
\begin{proof} Any $\f\in\PL(X)$ satisfies $\f=\f\circ p_\cX$, and hence $\int\f\,\mu_\cX=\int(\f\circ p_\cX)\,\mu=\int\f\,\mu$, for all $\cX$ high enough. This implies (i), by density of $\PL(X)$ in $\Cz(X)$. Pick now $\f\in\PSH(\om)$. Then $\f$ is usc, and $\f\le\f\circ p_\cX$ (see Lemma~\ref{lem:approx}). Thus $\int\f\,\mu\le\int(\f\circ p_\cX)\,\mu=\int\f\,\mu_\cX$, and (ii) follows, thanks to Lemma~\ref{lem:lsccv} (applies to the lsc function $\nu\mapsto-\int\f\,\nu$). 

Since $\ld_X$ is lsc and $\ld_X\circ p_\cX\le\ld_X$ (see Theorem~\ref{thm:lddual}), Lemma~\ref{lem:lsccv} similarly yields $\Ent_X(\mu_\cX)\to\Ent_X(\mu)$, which proves (iii) when $B=0$. In the general case, observe that (iii) trivially holds if $\Ent_{X,B}(\mu)=+\infty$, again by Lemma~\ref{lem:lsccv}. We therefore assume $\Ent_{X,B}(\mu)<+\infty$, and hence $\mu(X^\an\setminus X^\fld)=0$, by~\eqref{equ:entfld}. By Proposition~\ref{prop:ldvalbir}, we have $\ld_{X,B}=\ld_X-\p_B$ on $X^\val\supset X^\fld$, where $\p_B(v):=v(B)$. Write $B=B_1-B_2$ with $B_i\le 0$, $i=1,2$. Pick an ample line bundle $L$ such that $B_i+L$ is semiample, for $i=1,2$, and hence $\p_{B_i}\in\PSH(L)$ (see~\cite[Lemma~6.7]{trivval}). By Theorem~\ref{thm:ldan}~(ii), each $\p_{B_i}$ satisfies $0\le-\p_{B_i}\le C\ld_{X,B}$ on $X^\an$ for some $C>0$, and hence is integrable with respect to $\mu$ and all $\mu_\cX$. Since $\ld_{X,B}=\ld_X+\p_{B_2}-\p_{B_1}$ on $X^\fld$, $\ld_{X,B}$ is integrable with respect to $\mu$ and $\mu_\cX$ as well, and $\Ent_{X,B}(\mu_\cX)=\Ent_X(\mu_\cX)+\int\p_{B_2}\,\mu_\cX-\int\p_{B_1}\,\mu_\cX$. By (ii) and the case $B=0$ of (iii), this converges to $\Ent_{X,B}(\mu)=\Ent_X(\mu)+\int\p_{B_2}\,\mu-\int\p_{B_1}\,\mu$, and the general case of (iii) follows. 
\end{proof}

\begin{proof}[Proof of Theorem~\ref{thm:entapprox}] Observe first that the result is trivial if $\Ent_{X,B}(\mu)=+\infty$. By Lemma~\ref{lem:divdense}, we can indeed pick a sequence $(\mu_i)$ in $\cM^\div$ converging strongly to $\mu$, and we then have $\Ent_{X,B}(\mu_i)\to\Ent_{X,B}(\mu)$, by Lemma~\ref{lem:lsccv}. 

We therefore assume $\Ent_{X,B}(\mu)<+\infty$, and hence $\mu(X^\an\setminus X^\fld)=0$, by~\eqref{equ:entfld}. Since the strong topology of $\cM^1$ is defined by the quasi-metric $\ii$ (see Theorem~\ref{thm:I}), it is enough to show that, for any $\e>0$, there exists $\nu\in\cM^\div$ such that 
\begin{equation}\label{equ:estimapprox}
\ii(\mu,\nu)\le\e\quad\text{and}\quad|\Ent_{X,B}(\mu)-\Ent_{X,B}(\nu)|\le\e.
\end{equation}
To see this, pick a resolution of singularities $\pi\colon X'\to X$. Since the induced map $X'^\an\to X^\an$ is surjective, it follows from general theory that $\pi_\star\colon\cM(X')\to\cM(X)$ is surjective as well (see for instance~\cite[V.5.4]{SW}), and we can thus write $\mu=\pi_\star\mu'$ for some measure $\mu'\in\cM(X')$. 

For any snc test configuration $\cX$ for $X'$, set 
$$
\mu'_\cX:=(p_\cX)_\star\mu'\in\cM(\D_{\cX})\subset\cM(X'),\quad\mu_\cX:=\pi_\star\mu'_\cX\in\cM(X).
$$
By Lemma~\ref{lem:ldlisse}, we have $\lim_\cX\mu'_\cX=\mu'$ weakly in $\cM(X')$, and hence $\lim_\cX\mu_\cX=\mu$ weakly in $\cM(X)$. Next, pick $\om\in\Amp(X)$, and $\om'\in\Amp(X')$ such that $\pi^\star\om\le\om'$. For any $\f\in\cE^1(\om)$, $\pi^\star\f\in\PSH(\om')$ satisfies $(\pi^\star\f)\le(\pi^\star\f)\circ p_\cX$, and hence $\int\f\,\mu\le\int\f\,\mu_\cX$. Thus 
$$
\|\mu_\cX\|_\om=\sup_{\f\in\cE^1(\om)}\left(\en_\om(\f)-\int\f\,\mu_\cX\right)\le\sup_{\f\in\cE^1(\om)}\left(\en_\om(\f)-\int\f\,\mu\right)=\|\mu\|_\om. 
$$
By Lemma~\ref{lem:lsccv}, we infer $\lim_\cX\|\mu_\cX\|_\om=\|\mu\|_\om$, and hence $\mu_\cX\to\mu$ strongly in $\cM^1(X)$ (the strong topology being the coarsest refinement of the weak one in which $\|\cdot\|_\om$ is continuous). Finally, set $B':=\pi^\star K_{X,B}-K_{X'}$. By Proposition~\ref{prop:ldvalbir}, we have $\pi^\star\ld_{X,B}=\ld_{X',B'}$ on $X'^\val$. In particular, $(X',B')$ is subklt, and hence $\pi^\star\ld_{X,B}=\ld_{X',B'}\equiv+\infty$ outside $X'^\val$. This implies $\Ent_{X,B}(\mu_\cX)=\Ent_{X',B'}(\mu'_\cX)$, which converges to $\Ent_{X,B}(\mu)=\Ent_{X',B'}(\mu')$ by Lemma~\ref{lem:ldlisse}~(iii). 

Replacing $\mu$ with $\mu_\cX$, it is thus enough to show~\eqref{equ:estimapprox} when $\mu=\pi_\star\mu'$ with $\mu'\in\cM(\D_\cX)$. By density of the set of rational points $\D_\cX(\Q)$ in the simplicial complex $\D_\cX$, we can now write $\mu'$ as the weak limit of a sequence of measures $\mu'_i\in\cM(\D_\cX)$ with finite support in $\D_\cX(\Q)=\D_\cX\cap X'^\div$, and hence lying in $\cM^\div(X')$. We claim that $\mu_i:=\pi_\star\mu'_i\in\cM^\div(X)$ satisfies the desired estimate~\eqref{equ:estimapprox} for $i$ large enough. 

On the one hand, by continuity of $\pi^\star\ld_{X,B}=\ld_{X',B'}$ on $\D_\cX$, the weak convergence $\mu'_i\to\mu'$ implies $\Ent_{X,B}(\mu_i)=\int\ld_{X',B'}\,\mu'_i\to\int\ld_{X',B'}\,\mu'=\Ent_{X,B}(\mu)$. On the other hand, since $\pi^\star\f\in\PSH(\om')$ for any $\f\in\PSH(\om)$, Lemma~\ref{lem:approx}~(iii) implies $\int\f\,\mu_i=\int\pi^\star\f\,\mu'_i\to\int\pi^\star\f\,\mu'=\int\f\,\mu$ uniformly with respect to $\f\in\PSH(\om)$. By~\cite[Theorem~10.12]{trivval}, this implies, in particular, $\mu_i\to\mu$ strongly in $\cE^1(\mu)$. 
\end{proof}

%
%
\subsection{The $\d$-invariant and Ding-stability}\label{sec:Ding}
As in the previous section, we assume that $(X,B;\om)$ is a polarized \textbf{subklt} pair. Extending~\cite{Fujval}, we then show that the $\d$-invariant is a threshold for the Ding-stability of $(X,B;\om)$.
See also~\cite[Appendix A.2]{Hat}.

For any function $\f\colon X^\div\to\R$ we set 
$$
\elle_{X,B}(\f):=\inf_{X^\div}(\ld_{X,B}+\f)\in\R\cup\{-\infty\}.
$$
The (non-Archimedean) \emph{Ding functional} $\ding\colon\cE^1(\om)\to\R\cup\{-\infty\}$ of the polarized pair $(X,B;\om)$ is defined for $\f\in\cE^1(\om)$
\begin{equation}\label{equ:ding}
\ding(\f)=\ding_{X,B;\om}(\f):=\elle_{X,B}(\f)-\en_\om(\f),  
\end{equation}
where $\elle_{X,B}(\f):=\elle_{X,B}(\f|_{X^\div})$. It is easy to see that 
\begin{equation}\label{equ:dingtrans}
\ding(\f+c)=\ding(\f)\quad\text{and}\quad\ding(t\cdot\f)=t\ding(\f)
\end{equation}
for $c\in\R$ and $t\in\R_{>0}$. 

\begin{prop}\label{prop:Ding} The Ding functional is finite-valued, and satisfies a H\"older estimate
\begin{equation}\label{equ:dingholder}
|\ding(\f)-\ding(\p)|\le C\ii(\f,\p)^\a\max\{\ii(\f),\ii(\p)\}^{1-\a}
\end{equation}
for all $\f,\p\in\cE^1(\om)$, where $\a=\a_n\in (0,1)$ only depends on $n$ and $C=C(n,\d)>0$ only depends on $n$ and $\d:=\d(X,B;\om)>0$. In particular, $\ding$ is strongly continuous, and $\ding(\f)\ge-C\ii(\f)$. 
\end{prop}

\begin{proof} By translation invariance of the Ding functional (see~\eqref{equ:dingtrans}), we may assume that $\f,\p$ are normalized by $\sup\f=\sup\p=0$. We first claim that $\elle_{X,B}(\f)\ge-C\ii(\f)$, where $C=C(n,\d)>0$. Indeed, $\ld_{X,B}(v)\ge\d\|v\|$ together with~\eqref{equ:BBGZ} implies
\begin{equation}\label{equ:Aflow}
\ld_{X,B}(v)+\f(v)\ge\d\|v\|-C(n)\ii(\f)^\a \|v\|^{\frac 12}\max\{\|v\|,\ii(\f)\}^{\frac 12-\a}
\end{equation}
for each $v\in X^\div$. The right-hand side is easily seen to be bounded below by $-C(n,\d)\ii(\f)$, and the claim follows. 
By~\eqref{equ:Aflow}, we can further find $B=B(n,\d)>0$ such that 
$$
\|v\|\ge B\ii(\f)\Longrightarrow\ld_{X,B}(v)+\f(v)\ge 0=\ld_{X,B}(v_\triv)+\f(v_\triv). 
$$
Thus 
\begin{equation}\label{equ:Frestr}
\elle_{X,B}(\f)=\inf_{v\in X^\div,\,\|v\|\le B\ii(\f)}\{\ld_{X,B}(v)+\f(v)\},
\end{equation}
and the same holds for $\p$. Now set $M:=\max\{\ii(\f),\ii(\p)\}$. By~\eqref{equ:BBGZ} for any $v\in X^\div$ with $\|v\|\le B M$, we have $|\f(v)-\p(v)|\le B'\ii(\f,\p)^\a M^{1-\a}$ with $\a=\a_n$ and $B'=B'(n,\d)$. In view of~\eqref{equ:Frestr}, this implies
$$
|\elle_{X,B}(\f)-\elle_{X,B}(\p)|\le C\ii(\f,\p)^\a M^{1-\a}
$$ 
with $C=C(n,\d)$. By~\eqref{equ:enholder}, a similar estimate holds for $\en_\om$, and~\eqref{equ:dingholder} follows, by~\eqref{equ:ding}. 
\end{proof}

As we next show, the $\d$-invariant is a threshold for Ding-stability (see also~\cite{YTD,BL22}). 

\begin{thm}\label{thm:deltading} Set $\d:=\d(X,B;\om)$. Then: 
\begin{itemize}
\item[(i)] $\d\ge 1$ iff $(X,B;\om)$ is \emph{Ding-semistable}, \ie $\ding\ge 0$ on $\cE^1(\om)$; 
\item[(ii)] $\d>1$ iff $(X,B;\om)$ is \emph{uniformly Ding-stable}, \ie $\ding\ge\e\jj$ on $\cE^1(\om)$ for some $\e>0$. 
\end{itemize}
\end{thm}
By strong continuity of the Ding functional (see Proposition~\ref{prop:Ding}), it is enough to test the conditions in the second halves of (i) and (ii) on the dense subset $\PL\cap\PSH(\om)$.

\begin{proof} Assume first $\ding\ge\e\jj$ with $\e\ge 0$. Pick $v\in X^\div$, and choose a maximizing sequence $(\f_i)$ in $\cE^1(\om)$ for $\d_v$, \ie $\en_\om(\f_i)-\f_i(v)\to\|v\|$. Then $\mu_i:=\MA(\f_i)\to\d_v$ strongly in $\cM^1$. The assumption yields 
\begin{align*}
\ld_{X,B}(v)+\f_i(v)-\en_\om(\f_i) & \ge\elle_{X,B}(\f_i)-\en_\om(\f_i)\\
&  =\ding(\f_i)\ge\e\jj(\f_i)\approx\e\|\mu_i\|,
\end{align*}
where the first term tends to $\ld_{X,B}(v)-\|v\|$, while the last one tends to $\e\|v\|$. We infer $\ld_{X,B}(v)\ge(1+C_n\e)\|v\|$, and hence $\d(X,B;L)\ge 1+C_n\e$, see~\eqref{equ:deltaval}. This shows the `if' parts in (i) and (ii). Conversely, assume $\d:=\d(X,B;\om)\ge 1$, and pick $\f\in\cE^1_{\sup}(\om)$. Then $\d^{-1}\f\in\cE^1_{\sup}(\om)$. For each $v\in X^\div$, we thus have $\|v\|\ge\en_\om(\d^{-1}\f)-\d^{-1}\f(v)$, and hence 
$$
\ld_{X,B}(v)+\f(v)\ge\d\|v\|+\f(v)\ge\d\en_\om(\d^{-1}\f)\ge\d^{-1/n}\en_\om(\f), 
$$
see~\cite[Lemma~7.33]{trivval} for the last inequality. Since $\jj_\om(\f)=-\en_\om(\f)$, this implies
$\ding(\f)\ge\e\jj_\om(\f)$ with $\e:=1-\d^{-1/n}\ge 0$, and $\e>0$ if $\d>1$. This proves the `only if' parts of (i) and (ii). 
\end{proof}

%
%
%
%
  
\section{Divisorial stability}\label{sec:Divstab}
As before, $(X,B;\om)$ denotes an arbitrary polarized pair. We introduce the main concepts of this paper, the $\b$-invariant of a divisorial measure and the associated notion of divisorial stability, and prove Theorem~B and Corollary~C in the introduction. 
%
\subsection{The $\b$-functional}\label{sec:bfunc}
We now introduce the key functional in this paper. 

\begin{defi} The \emph{$\b$-functional} $\b=\b_{X,B;\om}\colon\cM^\div\to\R$ of the polarized pair $(X,B;\om)$ is defined by setting
\begin{equation}\label{equ:bfunc}
\b(\mu):=\Ent_{X,B}(\mu)+\nabla_{K_{X,B}}\|\mu\|_\om
\end{equation}
for $\mu\in\cM^\div$. 
\end{defi}
To simplify notation, we slightly abusively use $K_{X,B}$ to also denote the image of the log canonical class in $\Num(X)$. By Theorem~\ref{thm:diffen}, \eqref{equ:bfunc} can be rewritten
\begin{equation}\label{equ:bexp}
\b(\mu)=\int\ld_{X,B}\,\mu+\frac{d}{dt}\bigg|_{t=0}\|\mu\|_{\om+t K_{X,B}}.
\end{equation}
Note that
\begin{equation}\label{equ:homogb}
\b(t_\star\mu)=t\b(\mu),\quad\b_{X,B;s\om}(\mu)=\b_{X,B;\om}(\mu)
\end{equation}
for all $\mu\in\cM^\div$, $t\in\Q_{>0}$ and $s\in\R_{>0}$, by~\eqref{equ:homogent} and~\eqref{equ:homogeffe}.

\medskip

The $\b$-functional is closed related to the \emph{(non-Archimedean) Mabuchi K-energy functional} of $(X,B;\om)$. In view of~\cite[Proposition~7.22]{BHJ1} (extended to a possibly irrational class) and~\eqref{equ:nablaE}, the latter can indeed be defined as a functional 
$$
\mab=\mab_{X,B;\om}\colon\PL\cap\PSH(\om)\to\R
$$ 
by setting
\begin{equation}\label{equ:mabfunc} 
\mab(\f):=\Ent(\MA_\om(\f))+\nabla_{K_{X,B}}\en_\om(\f)
\end{equation}
for $\f\in\PL\cap\PSH(\om)$ (see~\eqref{equ:twEder} for the notation). The Monge--Amp\`ere operator induces an injection 
$$
\MA\colon\PL\cap\PSH(\om)/\R\hto\cM^\div,
$$
and~\eqref{equ:nablavee}, \eqref{equ:nablaE} yield
\begin{equation}\label{equ:mabbet}
\mab(\f)=\b(\MA(\f)). 
\end{equation}
The main advantage of $\b$ over $\mab$ lies in the fact that both the domain and the entropy term of the former are independent of $\om$. 

\medskip

Consider next $v\in X^\lin$, hence $\d_v\in\cM^1$, and set for simplicity 
$$
\b(v):=\b(\d_v)=\ld_{X,B}(v)+\nabla_{K_{X,B}}\|v\|_\om, 
$$
compare~\eqref{equ:enshort}. Then $\b(tv)=t\b(v)$ for $t\in\R_{>0}$. When $v\in X^\div$, Theorem~\ref{thm:endiv} provides the following explicit description: 

\begin{lem}\label{lem:DL} For any prime divisor $F$ on a smooth birational model $\pi\colon Y\to X$, $v:=\ord_F$ satisfies 
\begin{equation}\label{equ:bDL}
\b(v) = \ld_{X,B}(v)+V_\om^{-1}\int_0^{+\infty}\nabla_{K_{X,B}}\vol(\om-\la F)\,d\la-\tr_\om(K_{X,B})\|v\|.
\end{equation}
\end{lem}
Since $\|v\| =V_\om^{-1}\int_0^{+\infty}\vol(\om-\la F)\,d\la$, the right-hand side of~\eqref{equ:bDL} coincides (up to the factor $V_{\om}$) with the invariant introduced in~\cite{DL} (generalizing the Fano case of~\cite{Fujval}). 

\medskip

Assume, finally, that $(X,B)$ is subklt. The entropy then admits a natural extension $\Ent_{X,B}\colon\cM^1\to [0,+\infty]$, characterized as the maximal lsc extension (see~\S\ref{sec:extent}). The $\b$-functional thus admits as well an extension 
\begin{equation}\label{equ:bext}
\b=\b_{X,B;\om}\colon\cM^1\to\R\cup\{+\infty\}, 
\end{equation}
defined by~\eqref{equ:bexp} for any $\mu\in\cM^1$. By strong continuity of $\nabla_{K_{X,B}}\|\cdot\|_\om\colon\cM^1\to\R$, \eqref{equ:bext} is characterized as the maximal (strongly) lsc extension of the $\b$-functional. Note that~\eqref{equ:homogb} is then valid for all $\mu\in\cM^1$ and $t\in\R_{>0}$. 

The Mabuchi K-energy also admits an extension 
$$
\mab=\mab_{X,B;\om}\colon\cE^1(\om)\to\R\cup\{+\infty\},
$$
defined by~\eqref{equ:mabfunc} for $\f\in\cE^1(\om)$, that still satisfies~\eqref{equ:mabbet}.  

%
\subsection{Divisorial stability}\label{sec:divstab}
We are now in a position to introduce the main new concept in this paper. 

\begin{defi} For any polarized pair $(X,B;\om)$ with $\b$-functional $\b=\b_{X,B;\om}\colon\cM^\div\to\R$, we say that $(X,B,\om)$ is:
\begin{itemize}
\item[(i)] \emph{divisorially semistable} if $\b\ge 0$ on $\cM^\div$;   
\item[(ii)] \emph{divisorially stable} if $\b\ge\e\|\cdot\|$ on $\cM^\div$ for some $\e>0$. 
\end{itemize}
\end{defi}

\begin{rmk}
  This notion is stronger than the one introduced in~\cite{Fujvol}, see Remark~\ref{rmk:Fujita}.
\end{rmk}  

As noted above, Lemma~\ref{lem:DL} shows that the restriction of $\b$ to $X^\div\hto\cM^\div$ coincides (up to the factor $V_{\om}$) with the invariant introduced in~\cite{DL}. As a consequence, we get: 

\begin{lem} If $(X,B;\om)$ is divisorially semistable (resp.~divisorially stable), then $(X,B;\om)$ is valuatively semistable (resp.~uniformly valuatively stable) in the sense of~\cite{DL,LiuYa}. 
\end{lem}

As we next show, divisorial stability also implies the usual notion of K-stability (see \S\ref{sec:divKstab} for a more detailed discussion). 

\begin{lem}\label{lem:Kbstab} Assume $\om=c_1(L)$ with $L\in\Pic(X)_\Q$ ample. If $(X,B;L)$ is divisorially semistable (resp.~divisorially stable), then $(X,B;L)$ is K-semistable (resp.~uniformly K-stable). 
\end{lem}

\begin{proof} By~\cite[Corollary~2.32]{trivval}, the set of normal, ample test configurations $(\cX,\cL)$ for $(X,L)$ is in 1--1 correspondence with the set $\cH(L)\subset\PL\cap\PSH(L)$ of \emph{Fubini--Study functions} $\f$ for $L$, and $\b(\MA(\f))=\mab(\f)$ further coincides with the K-energy functional of~\cite{BHJ1}, see~\eqref{equ:mabbet}. If $(X,B;L)$ is divisorially semistable, we thus get $\mab(\f)\ge 0$ for all $\f\in\cH(L)$, which is equivalent to $(X,B;L)$ being K-semistable, by~\cite[Proposition~8.2]{BHJ1}. If $(X,B;L)$ is divisorially stable, then $\mab(\f)\ge\e\|\MA(\f)\|$ for a uniform constant $\e>0$. Now $\|\MA(\f)\|\approx\jj(\f)$, and we conclude that $(X,B;L)$ is uniformly K-stable, again by~\cite[Proposition~8.2]{BHJ1}.
\end{proof}

When $(X,B)$ is subklt, divisorial (semi)stability turns out to be equivalent to K-(semi)stability with respect to norms/filtrations on the section ring of $(X,L)$, the restriction to norms of finite type corresponding to K-stability (se~\S\ref{sec:divstabnorms}). We conjecture, however, that the converse direction in Lemma~\ref{lem:Kbstab} holds as well, see \S\ref{sec:divKstab}. 

\medskip

According to an important result of Odaka~\cite{Oda} (see also~\cite[Theorem~9.1]{BHJ1}), any pair $(X,B)$ with $B\ge 0$ such that $(X,B;L)$ is K-semistable for some ample $L\in\Pic(X)_\Q$ is necessarily log canonical, \ie $\ld_{X,B}\ge 0$ on $X^\div$. Theorem~\ref{thm:delta} directly yields the following version for divisorial stability. 

\begin{cor}\label{cor:Odaka} If $(X,B;\om)$ is divisorially semistable for some $\om\in\Amp(X)$, then $(X,B)$ is sublc.\end{cor}
\begin{proof} The assumption amounts to $\b=\Ent_{X,B}+\nabla_{K_{X,B}}\|\cdot\|_\om\ge 0$ on $\cM^\div$. By~\eqref{equ:twisteddualhold2}, we further have $\nabla_{K_{X,B}}\|\cdot\|_\om\le C\|\cdot\|_\om$ for some $C>0$. Thus $\d(X,B;\om)\ge -C$, and the result follows from Theorem~\ref{thm:delta}~(i).
\end{proof}

%
\subsection{Divisorial stability threshold and openness}

 \begin{defi}\label{defi:stabthresh} We define the \emph{divisorial stability threshold} of a polarized pair $(X,B;\om)$ as 
$$
\sigma_\div(X,B;\om):=\inf_{\mu\in\cM^\div\setminus\{\mu_\triv\}}\frac{\b_{X,B;\om}(\mu)}{\|\mu\|_\om}\in\R\cup\{-\infty\}. 
$$
\end{defi}
By definition, we thus have 
\begin{align*}
(X,B;\om)\ \text{divisorially semistable } & \Longleftrightarrow \sigma_\div(X,B;\om)\ge 0,\\
(X,B;\om)\ \text{divisorially stable } & \Longleftrightarrow \sigma_\div(X,B;\om)>0.
\end{align*}
Note also that 
\begin{equation}\label{equ:homogsig}
\sigma_\div(X,B;s\om)=s^{-1}\sigma_\div(X,B;\om)
\end{equation}
for $s\in\R_{>0}$, by~\eqref{equ:homogb} and~\eqref{equ:homogenvee}.

As in Corollary~\ref{cor:Odaka}, Theorem~\ref{thm:delta}~(i) and~\eqref{equ:twisteddualhold2} directly imply: 

\begin{prop}\label{prop:thresh} For any polarized pair $(X,B;\om)$, we have 
$$
\sigma_\div(X,B;\om)>-\infty\Longleftrightarrow (X,B)\text{ sublc}, 
$$
and we then further have 
\begin{equation}\label{equ:sigbd}
\left|\sigma_\div(X,B;\om)-\d(X,B;\om)\right|\lesssim\|K_{X,B}\|_\om. 
\end{equation}
In particular, if $(X,B;\om)$ is divisorially semistable then $(X,B)$ is sublc. 
\end{prop}
This proves Theorem~B~(i) and the first part of Corollary~C in the introduction. 

In line with~\cite{OS15} (see also~\cite[\S9.1]{BHJ1}), we next show: 

\begin{prop}\label{prop:twistedKE} Let $(X,B;\om)$ be a polarized sublc pair such that $-K_{X,B}\equiv\la\om$, $\la\in\R$. Then
\begin{equation}\label{equ:sigKE}
\sigma_\div(X,B;\om)=\d(X,B;\om)-\la.
\end{equation}
Further: 
\begin{itemize}
\item[(i)] if $\la<0$  (log canonically polarized case), then $(X,B;\om)$ is divisorially stable; 
\item[(ii)] if $\la=0$ (log Calabi--Yau case), then $(X,B;\om)$ is divisorially semistable, and it is further divisorially stable iff $(X,B)$ is subklt; 
\item[(iii)] if $\la>0$ (log Fano case), then $(X,B;\om)$ is divisorially semistable (resp.~stable) iff $\d(X,B;\om)\ge\la$ (resp.~$\d(X,B;\om)>\la$), and $(X,B)$ is then necessarily subklt. 
\end{itemize}
\end{prop}
 
\begin{proof} By~\eqref{equ:Euler} and linearity of $\nabla_{\theta}\|\cdot\|_\om$ with respect to $\theta$, we have 
$$
\nabla_{K_{X,B}}\|\cdot\|_\om=-\la\nabla_\om\|\mu\|_\om=-\la\|\cdot\|_\om, 
$$
and the $\b$-functional thus takes the simpler form
\begin{equation}\label{equ:btwF}
\b=\Ent_{X,B}-\la\|\cdot\|_\om. 
\end{equation}
This directly implies~\eqref{equ:sigKE}, and the rest follows from Theorem~\ref{thm:delta}. 
\end{proof}

\begin{cor}\label{cor:divding} For any polarized subklt pair $(X,B;\om)$ such that $\om\equiv-K_{X,B}$, we have:
\begin{itemize}
\item[(i)] $(X,B;\om)$ is divisorially semistable iff it is Ding-semistable; 
\item[(ii)] $(X,B;\om)$ is divisorially stable iff it is uniformly Ding-stable.
\end{itemize}
If we further assume that $B$ is effective, \ie $(X,B;\om)$ is log Fano, then (i) and (ii) are also equivalent to $(X,B)$ being K-semistable and uniformly K-stable, respectively. 
\end{cor}
\begin{proof} (i) and (ii) are direct consequences of Theorem~\ref{thm:deltading} and Proposition~\ref{prop:twistedKE}, while  the final assertion follows from~\cite[Corollary~6.11]{Fujval}.

\end{proof}

We may now state one of the main results of this article. 

\begin{thm}\label{thm:stabopen} For any sublc pair $(X,B)$, the divisorial stability threshold $\sigma_\div(X,B;\om)$ is a (locally H\"older) continuous function of $\om\in\Amp(X)$. In particular, the set of $\om\in\Amp(X)$ such that $(X,B;\om)$ is divisorially stable is open.
\end{thm}

\begin{proof} Set, for simplicity, 
$$
\sigma_\om:=\sigma_\div(X,B;\om),\quad\b_\om:=\b_{X,B;\om}=\Ent_{X,B}+\nabla_{K_{X,B}}\|\cdot\|_\om
$$
By~\eqref{equ:sigbd} and Theorem~\ref{thm:delta}~(iv), $\sigma_\om$ is a locally bounded function of $\om\in\Amp(X)$. Next, recall that the (Thompson) metric 
$$
\d(\om,\om'):=\sup\{\d\in\R_{\ge 0}\mid e^{-\d}\om\le\om'\le e^\d\om\}
$$
of the open convex cone $\Amp(X)$ is locally equivalent to the metric defined by any norm on $\Num(X)$. Pick $\om,\om'\in\Amp(X)$ such that $\d:=\d(\om,\om')\le 1$. By Theorem~\ref{thm:twistedhold}, we have 
$$
\nabla_{K_{X,B}}\|\cdot\|_{\om'}\ge \nabla_{K_{X,B}}\|\cdot\|_\om+O(\d^\a)\|\cdot\|_\om
$$ 
for a constant $\a>0$ only depending on $n$. Since 
$$
\b_\om=\Ent_{X,B}+\nabla_{K_{X,B}}\|\cdot\|_\om\ge\sigma_\om\|\cdot\|_\om\text{ on }\cM^\div,
$$
we get
$$
\b_{\om'}=\Ent_{X,B}+\nabla_{K_{X,B}}\|\cdot\|_{\om'}\ge(\sigma_\om+O(\d^\a))\|\cdot\|_\om\text{ on }\cM^\div. 
$$
On the other hand, \eqref{equ:Eveelip} yields $\|\cdot\|_{\om'}\ge e^{-C_n\d}\|\cdot\|_\om$. Thus 
$$
\b_{\om'}\ge (1+O(\d))(1+O(\d^\a))\sigma_\om\|\cdot\|_{\om'},
$$
and hence
$$
\sigma_{\om'}\ge (1+O(\d))\left(\sigma_\om+O(\d^\a)\right).
$$
Since $\sigma_\om$ is locally bounded, this implies $\sigma_{\om'}\ge\sigma_\om+O(\d^\a)$ locally uniformly in $\Amp(X)$. By symmetry, we get the desired H\"older estimate $|\sigma_\om-\sigma_{\om'}|=O(\d^\a)$ locally uniformly in $\Amp(X)$. 
\end{proof}

\begin{rmk}\label{rmk:valthropen} As in~\cite{LiuYa}, one can introduce the \emph{valuative stability threshold}
$$
\sigma_\val(X,B;\om):=\inf_{v\in X^\div\setminus\{v_\triv\}}\frac{\b_{X,B;\om}(v)}{\|v\|_\om}\ge\sigma_\div(X,B;\om). 
$$
Restricting the above argument to $X^\div\hto\cM^\div$ shows that $\sigma_\val(X,B;\om)$ is also a (locally H\"older) continuous function of $\om\in\Amp(X)$; this recovers~\cite[Theorem~5.3]{LiuYa}. 
\end{rmk}

\begin{rmk}\label{rmk:sigM1} Assume $(X,B)$ is subklt, and let $\b\colon\cM^1\to\R\cup\{+\infty\}$ be the maximal lsc extension of the $\b$-functional, cf.~\S\ref{sec:bfunc}. The inequality $\b\ge\sigma_\div(X,B;\om)\|\cdot\|$, which holds on $\cM^\div$ by definition, extends to $\cM^1$, as in Corollary~\ref{cor:sigM1}. This yields
$$
\sigma_\div(X,B;\om)=\inf_{\mu\in\cM^1\setminus\{\mu_\triv\}}\frac{\b(\mu)}{\|\mu\|}=\inf_{\mu\in\cM^1,\,\|\mu\|=1}\b(\mu),
$$ 
where the right-hand equality holds by homogeneity of $\b$ and $\|\cdot\|$. 
\end{rmk}

%
%
%
%
\section{Comparison to other stability notions}\label{sec:DivKstab}
In this final section, we consider a polarized pair $(X,B;L)$, where $L$ is ample $\Q$-line bundle, and provide a detailed comparison of divisorial stability with usual K-stability, and its version for norms/filtrations. 
%
%
\subsection{Norms and filtrations}
We use~\cite{nakstab1} as a reference for what follows. For any $d\in\Z_{>0}$ such that $dL$ is an honest line bundle, we set 
$$
R_d:=\Hnot(X,dL),\quad R(X,dL):=\bigoplus_{m\in\N} R_{md}.
$$
We denote by $\cN_\R=\cN_\R(L)$ the set of (superadditive, $k^\times$-invariant, linearly bounded) \emph{norms} $\n\colon R(X,dL)\to\R\cup\{+\infty\}$ for some sufficiently divisible $d\in\Z_{>0}$; these are in 1--1 correspondence with the more commonly used (multiplicative, graded, linearly bounded) \emph{filtrations}, via the inverse maps
$$
F^\la R_m:=\{s\in R_m\mid\n(s)\ge\la\},\quad\n(s):=\max\{\la\in\R\mid v\in F^\la R_m\},\quad\la\in\R. 
$$
The space $\cN_\R$ comes with a natural scaling action $(t,\n)\mapsto t\n$ of $\R_{>0}$, which fixes the \emph{trivial norm} $\n_\triv\in\cN_\R$, such that $\n_\triv(s)=0$ for $s\in R_m\setminus\{0\}$; there is also a translation action $(c,\n)\mapsto \n+c$ of $\R$, where $(\n+c)(s):=\n(s)+mc$ for $s\in R_m$. 

The space $\cN_\R$ is equipped with a natural pseudometric $\dd_1$, and we endow it with the corresponding (non-Hausdorff) topology (see~\cite[\S 3]{nakstab1}). It contains as a dense subset the space $\cT_\Z\subset\cN_\R$ of $\Z$-valued norms of finite, which can be identified with the set of ample test configurations $(\cX,\cL)$ for $(X,L)$, thanks to the Rees construction. 

The \emph{Monge--Amp\`ere operator}
$$
\MA\colon\cN_\R\to\cM^1
$$
is the unique continuous map (with respect to the $\dd_1$-topology on $\cN_\R$ and the strong topology on $\cM^1$) that takes the norm $\n\in\cT_\Z$ corresponding to an ample test configuration $(\cX,\cL)$ to the divisorial measure $\MA(\n)=\sum_E m_E\,\d_{v_E}$, where $E$ ranges over the irreducible components of the central fiber $\tcX_0$ of the normalization $(\tcX,\tcL)$ of $(\cX,\cL)$, $v_E\in X^\div$ is the associated divisorial valuation, and $m_E:=\ord_E(\tcX_0)(E\cdot\tcL^n)$ (see~\cite[\S 7]{nakstab1}). 

Extending~\cite{Der}, we define the \emph{minimum norm} of $\n\in\cN_\R$ as the energy
\begin{equation}\label{equ:minnorm}
\|\n\|:=\|\MA(\n)\|_L
\end{equation}
of the Monge--Amp\`ere measure $\MA(\n)$. It depends continuously on $\n$, and vanishes iff $\dd_1(\n,\n_\triv+c)=0$ for some $c\in\R$. 

Any $v\in X^\div$ defines a norm $\n_v\in\cN_\R$, such that $\n_v(s):=v(s)$ for $s\in R_m$. More generally, a \emph{divisorial norm} is defined as a norm of the form 
\begin{equation}\label{equ:signorm}
\n=\min_{v\in\Sigma}\{\n_v+c_v\} 
\end{equation}
for a finite subset $\Sigma\subset X^\div$ and $c_v\in\R$. The notation means $\n(s)=\min_{v\in\Sigma}\{v(s)+mc_v\}$ for $s\in R_m$, the corresponding filtration thus being
$$
F^\la R_m=\{s\in R_m\mid v(s)+m c_v\ge \la\text{ for all }v\in\Sigma\},\quad\la\in\R.
$$
The subset $\cN^\div_\R\subset\cN_\R$ of divisorial norms is preserved by the scaling action of $\Q_{>0}$ and the translation action of $\R$. A divisorial norm $\n\in\cN^\div_\R$ is \emph{rational} if it is $\Q$-valued, which equivalently means that the $c_i$ in~\eqref{equ:signorm} can be chosen in $\Q$. They form a subset $\cN^\div_\Q\subset\cN^\div_\R$, which is dense in $\cN_\R$. By~\cite[Corollary~7.16]{nakstab1}, the Monge--Amp\`ere operator induces a homeomorphism 
$$
\MA\colon\cN^\div_\R/\R\simto\cM^\div,
$$
where the left-hand side is equipped with the quotient topology and the right-hand side with the strong topology. In particular, a divisorial norm $\n\in\cN^\div_\R$ satisfies $\|\n\|=0$ iff $\n=\n_\triv+c$, $c\in \R$. Further, if $\n\in\cN^\div_\R$ is written as in~\eqref{equ:signorm} then $\MA(\n)$ is supported in $\Sigma$. 

Recall from~\cite[\S2.2]{nakstab1} that a norm $\n\in\cN_\R$ is \emph{homogeneous} if $\n(s^d)=d\n(s)$ for any $s\in R_m$ and $d\ge 1$. To any $\n\in\cN_\R$ we can associate a \emph{homogenization} $\n^{\hom}$ defined by $\n^{\hom}(s):=\lim_dd^{-1}\n(s^d)$. This is the smallest homogeneous norm such that $\n\le\n^{\hom}$. We have $\dd_1(\n,\n^{\hom})=0$ and $\MA(\n^{\hom})=\MA(\n)$ for any $\n\in\cN_\R$, see Corollary~5.2 and Proposition~7.4 in~\cite{nakstab1}.

To any (not necessarily ample, nor normal) test configuration $(\cX,\cL)$ for $(X,L)$ we can associate a norm $\n=\n_{\cX,\cL}\in\cN_\R$, see~\cite{WN12,BHJ1}.
\begin{lem}\label{lem:tcdivnorm}
  We have $\cN_\Q^{\div}=\{\n_{\cX,\cL}^{\hom}\mid (\cX,\cL)\ \text{test configuration for $(X,L)$}\}$. In particular, $\MA(\n_{\cX,\cL})=\MA(\n_{\cX,\cL}^{\hom})$ is a divisorial measure for any test configuration $(\cX,\cL)$.
\end{lem}
\begin{proof}
  For any test configuration $(\cX,\cL)$ for $(X,L)$, the homogenization $\n^{\hom}_{\cX,\cL}$ of $\n_{\cX,\cL}$ lies in $\cN_\Q^\div$, see~\cite[Corollary~A.8]{nakstab1}. Conversely, consider any norm $\n\in\cN_\Q^{\div}$. By~\cite[Theorem~6.12]{nakstab1}, $\n=\IN(\f)$ for some PL function $\f\in\PL(X)$. Moreover, such a PL function $\f$ is of the form $\f_{\cX,\cL}$ for a normal test configuration $(\cX,\cL)$ for $(X,L)$, see~\cite[Theorem~2.31]{trivval}. By~\cite[Proposition~A.3]{nakstab1} $\n=\n^{\hom}_{\cX,\cL}$.
\end{proof}
%
\subsection{Divisorial stability in terms of divisorial norms}\label{sec:divstabnorms}
Consider the functional  $\b=\b_{X,B;L}\colon\cM^\div\to\R$ of the polarized pair $(X,B;L)$; thus
$$
\b(\mu)=\Ent_{X,B}(\mu)+\nabla_{K_{X,B}}\|\mu\|_L
$$
for $\mu\in\cM^\div$. As we just saw, the Monge--Amp\`ere operator induces a homeomorphism $\MA\colon\cN^\div_\R/\R\simto\cM^\div$; in line with~\eqref{equ:mabbet}, we define the \emph{Mabuchi K-energy functional} 
\begin{equation}\label{equ:MNdiv}
\mab\colon\cN^\div_\R\to\R
\end{equation}
by setting
\begin{equation}\label{equ:mabnorm}
\mab(\chi):=\mab_{X,B;L}(\n):=\b(\MA(\n))\in\R. 
\end{equation}
Then
\begin{equation}\label{equ:mabaction}
\mab(\n+c)=\mab(\n),\quad\mab(t\n)=t\mab(\n)
\end{equation}
for all $c\in\R$ and $t\in\Q_{>0}$. 

\begin{thm}\label{thm:threshnorm} The divisorial stability threshold of any polarized pair $(X,B;L)$ satisfies 
$$
\sigma_\div(X,B;L)=\inf_{\n\in\cN^\div_\R,\,\|\n\| >0}\frac{\mab(\n)}{\|\n\| }=\inf_{\n\in\cN^\div_\Q,\,\|\n\| >0}\frac{\mab(\n)}{\|\n\| }.
$$
In particular, $(X,B;L)$ is divisorially semistable iff $\mab(\n)\ge 0$ for all rational divisorial norms $\n\in\cN^\div_\Q$, and it is divisorially stable iff there exists $\e>0$ such that $\mab(\n)\ge\e\|\n\| $ for all $\n\in\cN^\div_\Q$.
\end{thm}
Recall that a divisorial norm $\n\in\cN^\div_\R$ satisfies $\|\n\|=0\Leftrightarrow\n=\n_\triv+c$ for $c\in\R$.

\begin{lem}\label{lem:ratdivdense} For any $\n\in\cN^\div_\R$, there exists a sequence $(\n_i)$ in $\cN^\div_\Q$ such that $\dd_1(\n_i,\n)\to 0$ and $\mab(\n_i)\to\mab(\n)$. 
\end{lem}
\begin{proof} Pick a finite subset $\Sigma\subset X^\div$ such that $\n=\min_{v\in\Sigma}\{\n_v+c_v\}$ with $c\in\R^\Sigma$. As recalled above, we then have $\supp\MA(\n)\subset\Sigma$. Writing $c$ as the limit of a sequence $(c_i)$ in $\Q^\Sigma$ defines a sequence $\n_i:=\max_{v\in\Sigma}\{\n_v+c_{i,v}\}\in\cN^\div_\Q$, such that 
$$
\dd_1(\n_i,\n)\le\dd_\infty(\n_i,\n)\le\max_{v\in\Sigma}|c_{i,v}-c_v|\to 0, 
$$
see~\cite[(6.2)]{nakstab1}. By continuity of $\MA\colon\cN_\R\to\cM^1$, this implies $\mu_i=\MA(\n_i)\to\mu=\MA(\n)$ strongly in $\cM^1$, and hence $\nabla_{K_{X,B}}\|\mu_i\|_L\to\nabla_{K_{X,B}}\|\mu\|_L$. Further, $\mu_i=\sum_{v\in\Sigma} m_{i,v}\d_v$, $\mu=\sum_{v\in\Sigma} m_v\d_v$. Thus $\mu_i\to\mu$ implies $m_{i,v}\to m_v$ for all $v\in\Sigma$, and hence 
$$
\Ent_{X,B}(\mu_i)=\sum_{v\in\Sigma} m_{i,v}\ld_{X,B}(v)\to\Ent_{X,B}(\mu)=\sum_{v\in\Sigma} m_v\d_v.
$$
Thus $\mab(\n_i)=\b(\mu_i)=\Ent_{X,B}(\mu_i)+\nabla_{K_{X,B}}\|\mu_i\|_L\to\mab(\n)=\Ent_{X,B}(\mu)+\nabla_{K_{X,B}}\|\mu\|_L$, and we are done. 
\end{proof}

\begin{proof}[Proof of Theorem~\ref{thm:threshnorm}] For any $\n\in\cN^\div_\R$ we have 
$$
\mab(\n)=\b(\MA(\n)),\quad\|\n\| =\|\MA(\n)\|,
$$
Since $\MA\colon\cN^\div_\R\to\cM^\div$ is onto, this directly yields 
$$
\sigma_\div(X,B;L)=\inf_{\mu\in\cM^\div,\,\|\mu\|>0}\frac{\b(\mu)}{\|\mu\|}=\inf_{\n\in\cN^\div_\R,\,\|\n\|>0}\frac{\mab(\n)}{\|\n\| }. 
$$
Using Lemma~\ref{lem:ratdivdense}, it is further easy to see that the latter infimum can be computed using norms in $\cN^\div_\Q$, and the result follows.  
\end{proof}
\begin{rmk}\label{rmk:Fujita}
  In~\cite{Fujvol}, K.~Fujita defined a notion of divisorial (semi)stability for a $\Q$-Fano variety $X$ by requiring that $\eta(D):=1-V^{-1}\int_0^\infty\vol(L-\la D)d\la$ be (semi)positive for each effective Weil divisor $D$ on $X$, where $L=-K_X$ (see also~\cite{Gri}). By monotonicity of the volume, it suffices to check this when $D=\sum_i D_i$ is reduced. In this case, one can show that
  $\eta(D)=\mab(\n_D)$ with $\n_D:=\min_i\n_{\ord_{D_i}}$. This shows that the present notion of divisorial stability implies that of~\cite{Fujvol}.
 \end{rmk}
%
%
\subsection{K-stability for filtrations}\label{sec:Kstabfilt}
In this section, we assume that $(X,B)$ is \textbf{subklt}. Recall from~\S\ref{sec:extent} that $\b\colon\cM^\div\to\R$ admits a natural extension $\b\colon\cM^1\to\R\cup\{+\infty\}$, characterized as its maximal (strongly) lsc extension. We may thus introduce: 

\begin{defi}\label{defi:MN} We define the \emph{Mabuchi K-energy functional} 
\begin{equation}\label{equ:MN}
\mab\colon\cN_\R\to\R\cup\{+\infty\}
\end{equation}
of the polarized subklt pair by setting
$$
\mab(\n):=\mab_{X,B;L}(\n):=\b(\MA(\n))\in\R\cup\{+\infty\}
$$
for any $\n\in\cN_\R$. 
\end{defi}
When $\n\in\cT_\Z$ corresponds to an ample test configuration $(\cX,\cL)$, this definition is compatible with~\cite[\S7.3]{BHJ1}. Indeed, \eqref{equ:mabbet} yields $\mab(\n)=\b(\MA(\cX,\cL))=\mab(\cX,\cL)$. 

By continuity of $\MA\colon\cN_\R\to\cM^1$, \eqref{equ:MN} is an lsc extension of~\eqref{equ:MNdiv}; it is in fact characterized as the maximal lsc extension, as follows from the next result. 

\begin{lem} For any $\n\in\cN_\R$, there exists a sequence $(\n_i)$ in $\cN^\div_\Q$ such that $\dd_1(\n_i,\n)\to 0$ and $\mab(\n_i)\to\mab(\n)$. 
\end{lem}
\begin{proof} By Lemma~\ref{lem:ratdivdense}, it is enough to produce such a sequence in $\cN^\div_\R$. By Theorem~\ref{thm:entapprox}, we can find a sequence $(\mu_i)$ in $\cM^\div$ such that $\mu_i\to\MA(\n)$ strongly and $\b(\mu_i)\to\b(\MA(\n))=\mab(\n)$. Since $\MA\colon\cN^\div_\R/\R\simto\cM^\div$ is a homeomorphism, we can find $\n_i\in\cN^\div_\R$ such that $\MA(\n_i)=\mu_i$ and $\dd_1(\n_i,\n)\to 0$. Then $\mab(\n_i)=\b(\mu_i)\to\mab(\n)$, and the result follows.  
\end{proof}

\begin{defi}\label{defi:Kstabfilt} A polarized subklt pair $(X,B;L)$ is \emph{K-semistable for filtrations} (resp.~\emph{uniformly K-stable for filtrations}) if $\mab\ge 0$ on $\cN_\R$ (resp.~$\mab\ge\e\|\cdot\|$ on $\cN_\R$ for some $\e>0$). 
\end{defi}

\begin{thm}\label{thm:threshnorm2} The divisorial stability threshold of any polarized subklt pair $(X,B;L)$ satisfies 
$$
\sigma_\div(X,B;L)=\inf_{\n\in\cN_\R,\,\|\n\| >0}\frac{\mab(\n)}{\|\n\| }=\inf_{\n\in\cN_\R,\,\|\n\|=1}\mab(\n).
$$
In particular, $(X,B;L)$ is divisorially semistable iff it is K-semistable for filtrations, and it is divisorially stable iff it is uniformly K-stable for filtrations. 
\end{thm} 
Note that Theorem~\ref{thm:threshnorm2} and Theorem~\ref{thm:stabopen} together imply the Main Theorem stated in the beginning of the introduction.
\begin{proof} Set $\sigma:=\sigma_\div(X,B;L)$. For any $\n\in\cN_\R$, we then have 
$$
\mab(\n)=\b(\MA(\n))\ge\sigma\|\MA(\n)\|=\sigma\|\n\| ,
$$ 
see Remark~\ref{rmk:sigM1}. Thus $\sigma\ge\inf_{\n\in\cN_\R,\,\|\n\| >0}\mab(\n)/\|\n\| $, while the other inequality follows from Theorem~\ref{thm:threshnorm}. Finally, the last equality holds by homogeneity with respect to the scaling action of $\R_{>0}$. 
\end{proof}

\begin{rmk} The analogue of the last expression involving the `unit sphere' $\|\n\|=1$ does not feature in Theorem~\ref{thm:threshnorm}. The reason for this is that it would involve rescaling by $\|\n\|\in\R_{>0}$, which is possibly irrational even when $\n=\n_v$ with $v\in X^\div$, whereas divisorial norms are only preserved by the scaling action of $\Q_{>0}$.
\end{rmk}

In view of Corollary~\ref{cor:divding}, we also get the following result in the log Fano case.
\begin{cor}\label{cor:normding} For any polarized subklt pair $(X,B;L)$ with $L=-K_{X,B}$, we have: 
\begin{itemize}
\item[(i)] $(X,B;L)$ is K-semistable for filtrations iff it is Ding-semistable; 
\item[(ii)] $(X,B;L)$ is uniformly K-stable for filtrations iff it is uniformly Ding-stable.
\end{itemize}
\end{cor}
%
%
\subsection{K-stability for models}\label{sec:uKstabmod}
Let $(X,B;L)$ be a polarized subklt pair.
As a direct consequence of Lemma~\ref{lem:tcdivnorm} and Theorem~\ref{thm:threshnorm}, we have 
\begin{prop}\label{prop:divstabmod}
  The divisorial stability threshold satisfies
  \begin{equation}\label{equ:divstabmod}
    \sigma_{\div}(X,B;L)=
    \sup\{\sigma\in\R\mid \mab(\n_{\cX,\cL})\ge\sigma\|\n_{\cX,\cL}\|\ \text{for all test
      configurations $(\cX,\cL)$}\}.
  \end{equation}    
  In particular, $(X,B;L)$ is divisorially semistable (resp.\ divisorially stable) iff $\mab(\n_{\cX,\cL})\ge0$ (resp.\ $\mab(\n_{\cX,\cL})\ge\e\|\n_{\cX,\cL}\|$ for some $\e>0$) for all test configurations $(\cX,\cL)$ for $(X,L)$.
\end{prop}
This means that divisorial stability is equivalent to Li's notion of \emph{uniform K-stability for models} introduced in~\cite{Li22}.\footnote{Li also incorporates the action of a reductive group; in our case, this group is trivial.} Indeed, a (normal) test configuration is called a \emph{model} in \emph{loc.cit.}  Li assumes $X$ is smooth as his definition relies on continuity of envelopes as in~\cite{nakstabold}, but we can bypass this using~\cite{nakstab1}. We note that Proposition~6.3 in~\cite{Li22} can be used to show that uniform K-stability for models is equivalent to uniform K-stability for filtrations, in the setting considered there.
%
%
\subsection{Divisorial stability vs.~K-stability}\label{sec:divKstab}
We return to the setting of an arbitrary polarized pair $(X,B;L)$. 

Let $\cN_\R^\hom\subset\cN_\R$ denote the set of homogeneous norms.
We have $\cN^\div_\R\subset\cN^\hom_\R$, and 
$$
\cT^\hom_\Q\subset\cN^\div_\Q\subset\cN^\hom_\Q,
$$
where $\cT^\hom_\Q$ denotes the set of rational, homogeneous norms of finite type. 

Homogenization $\n\mapsto\n^{\hom}$ maps $\cT_\Q$ onto $\cT_\Q^\hom\subset\cN^\div_\Q$ (see~\cite[Corollary~2.18]{nakstab1}), and we may thus define the Mabuchi K-energy of $\n\in\cT_\Q$ by $\mab(\n):=\mab(\n^{\hom})$. This is also equal to $\b(\MA(\n^{\hom}))=\b(\MA(\n))$, as $\dd_1(\n,\n^{\hom})=0$ implies $\MA(\n)=\MA(\n^{\hom})$, showing compatibility with Definition~\ref{defi:MN} in the subklt case.

Recall that the Rees construction yields an identification of the set $\cT_\Z$ of $\Z$-valued norms of finite type with that of ample test configurations for $(X,L)$. We denote by $\cT^\inte_\Z\subset\cT_\Z$ the subset corresponding to normal (\ie integrally closed) test configurations; it contains the set $\cT^\hom_\Z$ of homogeneous $\Z$-valued norms of finite type, which correspond to test configurations with reduced central fiber. Homogenization $\n\mapsto\n^{\hom}$ induces a 1--1 map $\cT^\inte_\Z\simto\cT^\hom_\Q$ (see~\cite[Theorem~A.11]{nakstab1}). 

To each $\n\in\cT^\inte_\Z$ is associated its \emph{Donaldson--Futaki invariant} 
$$
\DF(\n)=\DF_{X,B;L}(\n)\in\Q, 
$$
(see for instance~\cite[Definition~3.17]{BHJ1}), and $(X,B;L)$ is said to be 
\begin{itemize}
\item \emph{K-semistable} if $\DF(\n)\ge 0$ for all $\n\in\cT^\inte_\Z$;
\item \emph{uniformly K-stable} if there exists $\e>0$ such that $\DF(\n)\ge\e\|\n\| $ for all $\n\in\cT^\inte_\Z$. 
\end{itemize}
In analogy with Definition~\ref{defi:stabthresh} (see also~\cite[Theorem~1.1]{BLZ}), we introduce: 

\begin{defi} The \emph{K-stability threshold} of $(X,B;L)$ is defined as 
\begin{equation}\label{equ:Kthresh}
\sigma_\Ka(X,B;L):=\inf_{\n\in\cT^\inte_\Z,\,\|\n\| >0}\frac{\DF(\n)}{\|\n\| }.
\end{equation}
\end{defi}
Thus $(X,B;L)$ is K-semistable iff $\sigma_\Ka(X,B;L)\ge 0$, and it is uniformly K-stable iff $\sigma_\Ka(X,B;L)>0$.

\begin{thm}\label{thm:Kthresh} The K-stability threshold of any polarized pair $(X,B;L)$ satisfies
$$
\sigma_\Ka(X,B;L)=\inf_{\n\in\cT^\hom_\Z,\,\|\n\| >0}\frac{\mab(\n)}{\|\n\| }=\inf_{\n\in\cT_\Q,\,\|\n\| >0}\frac{\mab(\n)}{\|\n\| }=\inf_{\n\in\cT_\Q,\,\|\n\|=1}\mab(\n). 
$$
In particular, $\sigma_\Ka(X,B;L)\ge\sigma_\div(X,B;L)$. 
\end{thm}

\begin{lem}\label{lem:DFmab} Any $\n\in\cT^\inte_\Z$ satisfies $\DF(\n)\ge\mab(\n)$, and equality holds iff $\n\in\cT^\hom_\Z$. 
\end{lem}
\begin{proof} Let $(\cX,\cL)$ be the normal, ample test configuration that corresponds to $\n$. By~\cite[(7.7)]{BHJ1} we have $\DF(\n)=\mab(\n)+V^{-1}(\cX_0-\cX_{0,\redu})\cdot\cL^n$. Thus $\DF(\n)\ge\mab(\n)$, and equality holds iff $(\cX_0-\cX_{0,\redu})\cdot\cL^n=0\Leftrightarrow\cX_0=\cX_{0,\redu}\Leftrightarrow\n\in\cT^\hom_\Z$, since $\cL$ is ample. 
\end{proof}

\begin{proof}[Proof of Theorem~\ref{thm:Kthresh}] Set 
$$
\sigma_\Ka:=\sigma_\Ka(X,B;L),\quad\sigma_\Q:=\inf_{\n\in\cT_\Q,\,\|\n\| >0}\frac{\mab(\n)}{\|\n\| },\quad\sigma_\Z:=\inf_{\n\in\cT^\hom_\Z,\,\|\n\| >0}\frac{\mab(\n)}{\|\n\| }.
$$
For any $\n\in\cT^\inte_\Z\subset\cT_\Q$, Lemma~\ref{lem:DFmab} yields $\DF(\n)\ge\mab(\n)\ge\sigma_\Q\|\n\| $, and hence $\sigma_\Ka\ge\sigma_\Q$. For $\n\in\cT^\hom_\Z\subset\cT^\inte_\Z$, it yields, on the other hand, $\mab(\n)=\DF(\n)\ge\sigma_\Ka\|\n\| $, and hence $\sigma_\Z\ge\sigma_\Ka$.  For any $\n\in\cT^\hom_\Q$, $d\n$ lies in $\cT^\hom_\Z$ for $d\in\Z_{>0}$ sufficiently divisible. By homogeneity of $\mab$ and $\|\cdot\|$, we infer $\sigma_\Z=\inf_{\n\in\cT^\hom_\Q,\,\|\n\| >0}\frac{\mab(\n)}{\|\n\| }$. Since $\mab$ and $\|\cdot\|$ are both invariant under homogenization, this is also equal to $\sigma_\Q$. By homogeneity, $\sigma_\Q$ is also equal to $\inf_{\n\in\cT_\Q,\,\|\n\|=1}\mab(\n)$, since $\cT_\Q$ is preserved by the scaling action of $\Q_{>0}$ and the minimum norm $\|\n\|$ of any $\n\in\cT_\Q$ is rational. The proof is complete. 
\end{proof}

We conjecture that equality holds in Theorem~\ref{thm:Kthresh}: 

\begin{conj}\label{conj:sigma} For any polarized pair $(X,B;L)$  we have  
$\sigma_\div(X,B;L)=\sigma_\Ka(X,B;L)$. In particular, $(X,B;L)$ is divisorially semistable (resp.~divisorially stable) iff it is K-semistable (resp.~uniformly K-stable). 
\end{conj}
Note that the last part holds in the log Fano case, see Corollary~\ref{cor:divding}. 

\smallskip

Assuming $k=\C$, $X$ smooth and $B=0$, the last point of Conjecture~\ref{conj:sigma} would complete the proof of the `uniform version' of the Yau--Tian--Donaldson conjecture, to wit
$$
(X,L)\text{ uniformly K-stable }\Longleftrightarrow (X,L)\text{ uniquely cscK},
$$
where uniquely cscK means that $c_1(L)$ contains a unique K\"ahler form of constant scalar curvature---uniqueness being equivalent to the triviality of the so-called reduced automorphism group $\Aut^0(X,L)/\C^\times$. Indeed, the main result of~\cite{Li22} (combined with~\S\ref{sec:uKstabmod}) and~\cite{BDL20,BHJ2} respectively yield
$$
(X,L)\text{ divisorially stable }\Longrightarrow (X,L)\text{ uniquely cscK}\Longrightarrow (X,L)\text{ uniformly K-stable}. 
$$
Conjecture~\ref{conj:sigma} would in turn follow from the following: 

\begin{conj}\label{conj:entapp} Pick $\n\in\cN^\div_\Q$, and denote by $(\n_d)$ its sequence of \emph{canonical approximants}, where $\n_d$ is the norm on $R(X,dL)$ generated in degree $1$ by $\n|_{R_d}$, for $d$ sufficiently divisible. Then $\mab(\n_d)\to\mab(\n)$. 
\end{conj}
Conjecture~\ref{conj:entapp} can be viewed as a more precise version of~\cite[Conjecture~2.5]{nakstabold} (see also~\cite[Conjecture~4.4]{Li23}). Indeed, it is equivalent to requiring 
$$
\Ent_{X,B}(\mu_d)\to\Ent_{X,B}(\mu)
$$
with $\mu_d:=\MA(\n_d)$ and $\mu:=\MA(\n)$. Indeed, $\mab(\n_d)=\Ent_{X,B}(\mu_d)+\nabla_{K_{X,B}}\|\mu_d\|_L$, $\mab(\n)=\Ent_{X,B}(\mu)+\nabla_{K_{X,B}}\|\mu\|_L$, see~\eqref{equ:mabnorm}, where $\dd_1(\n_d,\n)\to 0$ (see~\cite[Theorem~3.18]{nakstab1}), which implies $\mu_d\to\mu$ strongly, and hence $\nabla_{K_{X,B}}\|\mu_d\|_L\to\nabla_{K_{X,B}}\|\mu\|_L$. 

\medskip

Conjecture~\ref{conj:entapp} is trivially true when $X$ is a curve, or in the toric case, if we restrict to toric norms. Indeed, in both cases $\n$ will automatically be of finite type, and hence equal to $\n_d$ for all $d$ sufficiently divisible.

\begin{rmk} In his pioneering work on K-stability for filtrations, G.~Sz\'ekelyhidi defined the Donaldson--Futaki invariant of a norm/filtration as $\DF(\n):=\liminf_{d\to\infty}\DF(\n_d)$, see~\cite[Definition~4]{Sze}. Following~\cite{BHJ1}, it is perhaps more natural to replace $\DF$ with its homogeneous counterpart $\mab$, and hence to consider instead $\liminf_{d\to\infty}\mab(\n_d)$; this is greater than $\mab(\n)$, by lower semicontinuity of~\eqref{equ:MN}, and Conjecture~\ref{conj:entapp} predicts that equality holds. 
\end{rmk}
%
%
%
\appendix
%
\section{Log discrepancy and dual complexes}\label{sec:logdisc}
The goal of this appendix is to review the extension of the log discrepancy function on a projective klt pair from the set of divisorial valuations to the whole Berkovich space. We also describe this extension using dual complexes of test configurations.
%
%
\subsection{The log discrepancy function on valuations}\label{sec:ldval}
First let $(X,B)$ be an arbitrary pair, that is, $X$ is a normal projective variety, and $B$ is a $\Q$-Weil divisor such that $K_{X,B}:=K_X+B$ is $\Q$-Cartier. We wish to extend the classical log discrepancy function
\begin{equation}\label{equ:lddiv2}
  \ld_{X,B}\colon X^\div\to\Q
\end{equation}
to $X^\val$ in a canonical way. To this end, we will prove
\begin{lem}\label{lem:ldlsc}
  The log discrepancy function $\ld_{X,B}\colon X^\div\to\Q$ is lsc.
\end{lem}
Granted this result, we can define the extension as follows.
\begin{defi}\label{defi:ldval}
  We define the log discrepancy $\ld_{X,B}\colon X^\val\to\R\cup\{\pm\infty\}$ as the greatest lsc extension of $\ld_{X,B}\colon X^\div\to\R$. In other words, we set
  \begin{equation}\label{equ:ldval}
    \ld_{X,B}(v):=\liminf_{w\in X^\div,\,w\to v} \ld_{X,B}(w)
  \end{equation}
  for all $v\in X^\val$.
\end{defi}
In fact, $\ld_{X,B}$ never takes the value $-\infty$, see Proposition~\ref{prop:ldvalbir} below.

To prove Lemma~\ref{lem:ldlsc}, and to get a more precise understanding of the log discrepancy function, we rely on~\cite{JM}, which gave a different way of extending the log discrepancy in the case when $X$ is smooth and $B=0$; see also~\cite{hiro,BdFFU}.

By a \emph{log smooth} pair $(Y,D)$ over $(X,B)$, we mean the data of 
a birational morphism $\pi\colon Y\to X$ with $Y$ smooth and projective, and 
$D$ a reduced snc divisor on $Y$ such that $\pi$ is an isomorphism outside $D$, and the support of $B$ is contained in the image of $D$.

Let $\fM=\fM_{X,B}$ be the set of isomorphism classes of log smooth pairs $(Y,D)$ over $(X,B)$.
Given $(Y,D)\in\fM$, the $\Q$-divisor $K_{Y/(X,B)}:=K_Y-\pi^\star K_{X,B}$ on $Y$
has support contained in $D$.
Any divisorial valuation $v\in X^\div$ is of the form $v=t\ord_E$, where $t\in\Q_{\ge0}$ and $E$ is an irreducible component of $D$, for some $(Y,D)\in\fM$, and we then have \begin{equation*}
  \ld_{X,B}(v)=t(1+\ord_E(K_{Y/(X,B)})). 
 \end{equation*}
Given $(Y,D)\in\fM$, let $\tau_Y$ be the continuous function on $X^\val=Y^\val$ given by
 $\tau_Y(v)=v(K_{Y/(X,B)})$.
 It follows that
\begin{equation}\label{equ:lds}
  \ld_{X,B}=\ld_Y+\tau_Y\ \text{on $X^\div=Y^\div$}.
\end{equation}
\begin{proof}[Proof of Lemma~\ref{lem:ldlsc}]
  By~\cite[Lemma~5.7]{JM}, the log discrepancy function $\ld_Y$ is lsc on $Y^\div$. As $\tau_Y$ is continuous,~\eqref{equ:lds} implies that $\ld_{X,B}$ is lsc on $X^\div=Y^\div$.
\end{proof}
\begin{prop}\label{prop:ldval}
  When $X$ is smooth, the log discrepancy function $\ld_X\colon X^\val\to\R$ in Definition~\ref{defi:ldval} coincides with the one defined in~\cite{JM}.
\end{prop}
To prove this, we need to recall the construction in~\cite{JM}. 
Given $(Y,D)\in\fM$, the dual cone complex $\hD(Y,D)$ embeds in an $\R_{>0}$-equivariant way into $Y^\val=X^\val$.
The subset $X^\div\subset X^\val$ is the union of the rational points of $\hD(Y,D)$ over all log smooth pairs.
For any log smooth pair $(Y,D)$ of $X$, there is also a continuous retraction 
$$
p_{Y,D}\colon X^\val\to\hD_{Y,D} 
$$
satisfying 
$p_{Y,D}(X^\div)\subset X^\div$.

We say that a log smooth pair $(Y',D')$ over $X$ dominates another, $(Y,D)$, if the canonical birational map $Y'\dashrightarrow Y$ is a morphism and $D'$ contains the support of the pullback of $D$. In this case, $\hD_{Y,D}\subset\hD_{Y',D'}$ and $p_{Y,D}\circ p_{Y',D'}=p_{Y,D}$. 
This turns the set of (isomorphism classes of) log smooth pairs into a directed set, and the retractions $p_{Y,D}$ induce a homeomorphism 
\begin{equation}\label{equ:berkhomeo1}
X^\val\simto\varprojlim_{Y,D}\hD_{Y,D},
\end{equation}
see~\cite[Theorem 4.9]{JM}. In particular, $\lim_{Y,D}p_{Y,D}=\id$ pointwise on $X^\val$.

Let us temporarily write $\ld'_X\colon X^\val\to\R_{\ge0}\cup\{+\infty\}$ for the log discrepancy function in~\cite[\S5]{JM}. It is characterized by the following properties:
\begin{itemize}
\item[(a)]
$\ld'_X=\ld_X$ on $X^\div$;
\item[(b)]
  for any $(Y,D)\in\fM$, the restriction of $\ld'_X$ to $\hD_{Y,D}$ is continuous, and integral linear on each cone; furthermore, 
\item[(c)]
  for any $(Y,D)\in\fM$, $\ld'_X\ge\ld'_X\circ p_{Y,D}$ on $X^\val$, with equality precisely on $\hD_{Y,D}$;
\item[(d)]
  $\ld'_X=\lim_{Y,D}\ld'_X\circ p_{Y,D}$.
\end{itemize}
\begin{proof}[Proof of Proposition~\ref{prop:ldval}]
  We must prove that $\ld_X=\ld'_X$ on $X^\val$. It follows from~(b) and~(d) that $\ld'_X$ is lsc, so $\ld'_X\le\ld_X$ by the definition of $\ld_X$. On the other hand, for each $(Y,D)\in\fM$, the function $\ld'_X\circ p_{Y,D}$ is continuous, whereas $\ld_X\circ p_{Y,D}$ is lsc. As the two functions agree on the dense subset $X^\div$, we have $\ld'_X\circ p_{Y,D}\ge\ld_X\circ p_{Y,D}$. This implies
  \begin{equation*}
    \ld'_X
    =\lim_{Y,D}\ld'_X\circ p_{Y,D}
    \ge\varliminf_{Y,D}\ld_X\circ p_{Y,D}
    \ge\ld_X,
  \end{equation*}
  by~(c) and the lower semicontinuity of $\ld_X$. The proof is complete.
\end{proof}

The log discrepancy functions satisfy the following expected properties.

\begin{prop}\label{prop:ldvalbir}
  Let $(X,B)$ be any pair. Then:
  \begin{itemize}
  \item[(a)]
    $\ld_{X,B}>-\infty$ on $X^\val$;
  \item[(b)]
    for any  $\Q$-Cartier divisor $D$ on $X$, we have $\ld_{X,B+D}=\ld_{X,B}-\p_D$ on $X^\val$, where $\p_D(v):=v(D)$;
  \item[(c)]
    for any projective birational morphism $\pi\colon X'\to X$, we have $\ld_{X,B}=\ld_{X',B'}$ on $X^\val=X^{\prime\val}$, where $B'=\pi^\star K_{X,B}-K_{X'}$.
  \end{itemize}
\end{prop}

\begin{proof}
  The equality in~(b) holds on $X^\div$, and the function $\p_D$ is continuous. By the nature of Definition~\ref{defi:ldval}, equality must hold on $X^\val$. A similar argument proves~(c), and~(a) now follows since $\ld_X\ge0$ when $X$ is smooth.
\end{proof}
%
%
\subsection{Extension to the whole Berkovich space}\label{sec:ldan}
Given a pair $(X,B)$, we seek to extend the function $\ld_{X,B}$ on $X^\val$ defined above to an lsc function $\ld_{X,B}\colon X^\an\to\R\cup\{+\infty\}$ on the whole Berkovich space. This is only possible when $(X,B)$ is \emph{sublc}, that is, $\ld_{X,B}\ge0$ on $X^\div$. Indeed, such an extension is necessarily bounded below, by compactness of $X^\an$, and hence nonnegative on $X^\div$, by homogeneity of $\ld_{X,B}$ on $X^\div$.

\begin{defi} For any sublc pair $(X,B)$, we denote by 
\begin{equation}\label{equ:ldan}
\ld_{X,B}\colon X^\an\to [0,+\infty]
\end{equation}
the greatest lsc extension of $\ld_{X,B}$ on $X^\div$, given by~\eqref{equ:ldval}. 
\end{defi}
Note that a sublc pair $(X,B)$ is \emph{subklt} if $\ld_{X,B}>0$ on $X^\div\setminus\{v_\triv\}$.
Such pairs can now be characterized as follows: 

\begin{thm}\label{thm:ldan} For any sublc pair $(X,B)$, the following are equivalent:
\begin{itemize}
\item[(i)] $(X,B)$ is subklt;
\item[(ii)] for any $\om\in\Amp(X)$, there exists $\a>0$ such that $\ld_{X,B}\ge\a\te_\om$ on $X^\an$; 
\item[(iii)] $\ld_{X,B}\equiv+\infty$ on $X^\an\setminus X^\val$.
\end{itemize}
\end{thm}

\begin{proof} Assume that $(X,B)$ is subklt. To prove (ii), we may assume $\om=c_1(L)$ with $L\in\Pic(X)_\Q$ ample. Denote by $|L|_\Q$ the set of effective $\Q$-Cartier divisors $D$ such that $D\sim_\Q L$. As is well-known (see for instance~\cite[Theorem~9.14]{BHJ1}), there exists $\a>0$ such that $(X,B+\a D)$ is subklt for all $D\in |L|_\Q$, and hence 
$$
0\le \ld_{X,B+\a D}(v)=\ld_{X,B}(v)-\a v(D)
$$ 
for all $v\in X^\div$ and $D\in |L|_\Q$. Since the maximal vanishing order satisfies 
\begin{align*}
\te_L(v) & =\sup\{m^{-1}v(s)\mid s\in\Hnot(X,mL)\setminus\{0\},\,m\ \text{sufficiently divisible}\} \\
& =\sup\{ v(D)\mid D\in |L|_\Q\}, 
\end{align*}
we get $\ld_{X,B}\ge\a\te_L$ on $X^\div$, and hence also on $X^\an$, since $\te_L$ is lsc. This proves (i)$\Rightarrow$(ii), and (ii)$\Rightarrow$(iii) is trivial, since $X^\lin=\{\te_L<+\infty\}$ is contained in $X^\val$. Finally, assume (iii). Pick $v\in X^\div\setminus\{v_\triv\}$, and denote by $Z\subset X$ the closure of the center of $v$, which is a strict subvariety since $v\ne v_\triv$. We obtain a trivial semivaluation $v_{Z,\triv}\in X^\an\setminus X^\val$, such that $\lim_{t\to+\infty} t v=v_{Z,\triv}$ in $X^\an$. As $\ld_{X,B}$ is lsc on $X^\an$, we infer
$$
\liminf_{t\to+\infty} t\ld_{X,B}(v)\ge\ld_{X,B}(v_{Z,\triv})=+\infty. 
$$
Thus $\ld_{X,B}(v)>0$, which proves (iii)$\Rightarrow$(i). 
\end{proof}
%
%
\subsection{Valuations of finite log discrepancy}\label{sec:Xfld}
If $(X,B)$ is a pair, Proposition~\ref{prop:ldvalbir} shows that the locus $\{\ld_{X,B}<\infty\}\subset X^\val$ does not depend on $B$, and it is a birational invariant. This naturally leads to the following notion.
\begin{defi}\label{defi:Xfld}
  For any (not necessarily normal) projective variety $X$, we define the set $X^\fld\subset X^\val$ of  \emph{valuations of finite log discrepancy} as the subset $\{\ld_Y<+\infty\}$ of $Y^\val=X^\val$ for some (or any) projective birational morphism $Y\to X$ with $Y$ smooth.
\end{defi}
Clearly $X^\fld$ is a birational invariant of $X$. We also note
\begin{lem}\label{lem:fldBorel}
  The set $X^\fld$ is a Borel subset of $X^\an$ for any projective variety $X$.
\end{lem}
\begin{proof}
  We may assume $X$ is smooth. Applying Theorem~\ref{thm:ldan} to the klt pair $(X,0)$ shows that $X^\fld=\{\ld_X<+\infty\}\subset X^\an$, which is a Borel set since $\ld_X$ is lsc.
\end{proof}
\begin{rmk}
  Similarly $X^\lin\subset X^\an$ is a Borel set as $X^\lin=\{\te_\om<+\infty\}$ for any $\om\in\Amp(X)$. We do not know whether $X^\val\subset X^\an$ is a Borel set when $k$ is uncountable.
\end{rmk}
%
%
\subsection{Log discrepancy via snc test configurations}\label{sec:ldtc}
We refer to~\cite[Appendix~A]{trivval} for details on the following discussion. For any projective variety $X$, Gauss extension provides an embedding 
$$
\sigma\colon X^\an\hto (X\times\P^1)^\an
$$
onto the set of $k^\times$-invariant semivaluations $w\in (X\times\P^1)^\an$ such that $w(\varpi)=1$, where $\varpi$ denotes the coordinate on $\A^1\subset\P^1$. If $v\in X^\an$, then $\sigma(v)$ is a valuation iff $v$ is.

Now assume that $X$ is smooth, and consider an \emph{snc test configuration} 
$$
\cX\to\A^1=\Spec k[\unipar]
$$ 
for $X$, \ie a test configuration dominating the trivial test configuration, such that $\cX$ is nonsingular and $\cX_{0,\redu}$ is snc. The canonical compactification $\bar\cX\to\P^1$ provides a log smooth pair $(\bar\cX,\cX_{0,\redu})$ over $X\times\P^1$. As in~\S\ref{sec:ldval}, the dual cone complex 
$$
\hD_\cX:=\hD(\cX,\cX_{0,\redu})
$$ 
embeds in $(X\times\P^1)^\val$. The preimage $\D_\cX:=\sigma^{-1}(\hD_\cX)\subset X^\val$ is compact, and $\sigma_\div(\D_\cX)\subset\hD_\cX$ is a compact simplicial complex cut out by the equation $w(\unipar)=1$. We view $\D_\cX$ as a simplicial complex itself, and each simplex has an integral affine structure with respect to which the rational points are the divisorial points; hence $\D_\cX\cap X^\div$ is dense in $\D_\cX$.

Let us write $\hat p_\cX\colon(X\times\P^1)^\val\to\hD_\cX$ for the retraction above. Then we have a continuous retraction $p_\cX:=\sigma^{-1}\circ\hat p_\cX\circ\sigma\colon X^\val\to\D_\cX$. This extends continuously to $X^\an$, and we obtain a retraction $p_\cX\colon X^\an\to\D_\cX$; this in turn induces a homeomorphism 
\begin{equation}\label{equ:berkhomeo2}
X^\an\simto\varprojlim_\cX\D_\cX,
\end{equation}
see~\cite[Theorem~A.1]{trivval}. Thus $\lim_\cX p_\cX=\id$ pointwise on $X^\an$. 

\begin{thm}\label{thm:lddual} If $X$ is a smooth projective variety, then
  \begin{equation}\label{equ:ldtest} 
    \ld_X=\ld_{X\times\P^1}\circ\sigma-1
  \end{equation}
  on $X^\an$. Moreover, for any snc test configuration $\cX$, we have: 
  \begin{itemize}
  \item[(i)] $\ld_X$ is integral affine on each face of $\D_\cX$; 
  \item[(ii)] $\ld_X\ge\ld_X\circ p_\cX$ on $X^\an$, with equality precisely on $\D_\cX$;
  \end{itemize}
  and $\ld_X=\lim_\cX\ld_X\circ p_\cX$ pointwise on $X^\an$. 
\end{thm}

\begin{proof}
  Set $\ld'_X:=\ld_{X\times\P^1}\circ\sigma-1$.
  By~\cite[Proposition~4.11]{BHJ1}, we have $\ld'_X=\ld_X$ on $X^\div$.
  Note that $\ld'_X$ is lsc since $\ld_{X\times\P^1}$ is lsc and $\sigma$ continuous. As $\ld_X$ is the maximal lsc extension of its restriction to $X^\div$, see~Definition~\ref{defi:ldval}, we get $\ld'_X\le\ld_X$ on $X^\an$.

  Now $\sigma\colon\D_\cX\to\hD_\cX$ is continuous and $\ld_{X\times\P^1}$ is continuous on $\hD_\cX$, so $\ld'_X$ is continuous on $\D_\cX$. As $\ld_X$ is lsc, and $\ld_X=\ld'_X$ on the dense subset $X^\div\cap\D_\cX$, we get $\ld_X\le\ld'_X$, and therefore $\ld'_X=\ld_X$ on $\D_\cX$. As $\ld_{X\times\P^1}$ is integral affine on each face of $\hD_\cX$, we obtain~(i).

  By~\cite{JM} we have $\ld_{X\times\P^1}\ge\ld_{X\times\P^1}\circ\hat p_\cX$ on $(X\times\A^1)^\val$, with equality exactly on $\hD_\cX$. This implies
  $\ld'_X\ge\ld'_X\circ p_\cX$ on $X^\val$, with equality exactly on $\D_\cX$. As $\ld_X\ge\ld'_X$ with equality on $\D_\cX$, we obtain $\ld_X\ge\ld_X\circ p_\cX$ on $X^\val$, with equality exactly on $\D_\cX$. This implies~(ii), since $\ld_X=+\infty$ on $X^\an\setminus X^\val$, see Theorem~\ref{thm:ldan}.

  The equality $\ld_X=\lim_\cX\ld_X\circ p_\cX$ pointwise on $X^\an$ is now a consequence of~(ii) and~\eqref{equ:berkhomeo2} since $\ld_X$ is lsc, see Lemma~\ref{lem:lsccv}. Similarly, we saw above that $\ld'_X\ge\ld'_X\circ p_\cX$ on $X^\val$, and hence on $X^\an$ since $A'_X=+\infty$ on $X^\an\setminus X^\val$. The same argument as above now gives $\ld'_X=\lim_\cX\ld'_X\circ p_\cX$ pointwise on $X^\an$. 
As $\ld'_X=\ld_X$ on $\D_\cX$, this shows that $\ld'_X=\ld_X$ on $X^\an$, and we are done.
\end{proof}

%
%
%
%
%

\end{document}